\pdfoutput=1
\RequirePackage{ifpdf}
\ifpdf 
\documentclass[pdftex]{sigma}
\else
\documentclass{sigma}
\fi

\numberwithin{equation}{section}

\newtheorem{Theorem}{Theorem}[section]
\newtheorem*{Theorem*}{Theorem}
\newtheorem{Corollary}[Theorem]{Corollary}
\newtheorem{Lemma}[Theorem]{Lemma}
\newtheorem{Proposition}[Theorem]{Proposition}
 { \theoremstyle{definition}
\newtheorem{Definition}[Theorem]{Definition}

\newtheorem{Remark}[Theorem]{Remark} }

\begin{document}

\allowdisplaybreaks

\newcommand{\arXivNumber}{2406.09800}

\renewcommand{\PaperNumber}{105}

\FirstPageHeading

\ShortArticleName{$R$-Matrix Presentation of Quantum Affine Superalgebra for Type $\mathfrak{osp}(2m+1|2n)$}

\ArticleName{$\boldsymbol{R}$-Matrix Presentation of Quantum Affine\\ Superalgebra for Type $\boldsymbol{\mathfrak{osp}(2m+1|2n)}$}

\Author{Xianghua WU~$^{\rm ab}$, Hongda LIN~$^{\rm ac}$ and Honglian ZHANG~$^{\rm ad}$}

\AuthorNameForHeading{X.~Wu, H.~Lin and H.~Zhang}

\Address{$^{\rm a)}$~Department of Mathematics, Shanghai University, Shanghai 200444, P.R.~China}

\Address{$^{\rm b)}$~School of Mathematics and Computing Science, Guilin University of Electronic Technology,\\
\hphantom{$^{\rm b)}$}~Guilin 541000, P.R.~China}
\EmailD{\href{mailto:shuxianghua@outlook.com}{shuxianghua@outlook.com}}

\Address{$^{\rm c)}$~Shenzhen International Center for Mathematics, SUSTech, Guangdong 518055, P.R.~China}
\EmailD{\href{mailto:hdlin@shu.edu.cn}{hdlin@shu.edu.cn}}

\Address{$^{\rm d)}$~Newtouch Center for Mathematics, Shanghai University, Shanghai 200444, P.R.~China}
\EmailD{\href{mailto:hlzhangmath@shu.edu.cn}{hlzhangmath@shu.edu.cn}}

\ArticleDates{Received June 19, 2024, in final form November 15, 2024; Published online November 22, 2024}

\Abstract{In our preceding research, we introduced the Drinfeld presentation of the quantum affine superalgebra associated to the orthosymplectic Lie superalgebra $\mathfrak{osp}(2m+1|2n)$ for $m>0$. We provided the isomorphism between its Drinfeld--Jimbo presentation and Drinfeld presentation using braid group actions as a fundamental method. Based on this work, our current study delves into its $R$-matrix presentation, wherein we establish a clear isomorphism between the $R$-matrix presentation and the Drinfeld presentation. In particular, our contribution extends the investigations of Jing, Liu and Molev concerning quantum affine algebra in type BCD to the realm of supersymmetry.}

\Keywords{quantum affine superalgebra; $R$-matrix presentation; Drinfeld presentation; universal $R$-matrix}

\Classification{17B37; 17B69}

\section{Introduction}

The quantum affine algebras $U_q(\hat{\mathfrak{g}})$ associated with affine Lie algebras $\hat{\mathfrak{g}}$ manifest at least three distinct presentations. The original definition of quantum affine algebras was defined as q-deformation of the universal enveloping algebras of affine Lie algebras, independently introduced by Drinfeld \cite{VGD} and Jimbo \cite{MJB}, collectively referred to as the Drinfeld--Jimbo presentation.
Drinfeld's pivotal contribution in 1987 \cite{VGD2} introduced a highly significant presentation of quantum affine algebras, commonly termed the Drinfeld presentation. The Drinfeld presentation has yielded a multitude of applications, including vertex representations and finite-dimensional representations. Subsequently, the $R$-matrix presentation was proposed by Reshetikhin and Semenov-Tian-Shansky \cite{RS}, later refined by Frenkel and Reshetikhin \cite{FIBR}.
The $R$-matrix presentation incorporates a matrix $R(z)$ associated with the quantum affine algebra, which satisfies the Yang--Baxter equation
\begin{gather*}
 R_{12}(z)R_{13}(zw)R_{23}(w)=R_{23}(w)R_{13}(zw)R_{12}(z)
\end{gather*}as documented in the work \cite{RS}.

In the study of quantum affine algebras, significant advancements have been made in elucidating the isomorphism among their presentations. Initially, Beck \cite{JBECK} pioneered the establishment of the isomorphism between the Drinfeld--Jimbo and Drinfeld presentations for untwisted algebras, while Jing--Zhang \cite{NJHZ1,NJHZ}, and \cite{HNJING} extended this for the twisted case. Different to Beck's methods, Damiani \cite{Damiani1,Damiani2} also constructed the isomorphism between the Drinfeld--Jimbo and Drinfeld presentations. Concurrently, Frenkel and Ding \cite{JDIB} established the isomorphism between the Drinfeld and $R$-matrix presentations in type~A. Building on Ding--Frenkel's approach, Jing, Liu and Molev \cite{JMA,JMAB} extended the {isomorphism} to types~B, C and~D, and a~similar result for Yangians~\cite{JMAB-Y}. These developments have established the isomorphism among these presentations, allowing for various approaches to studying the representation theory of quantum affine algebras. This provides a~rich framework for understanding and exploring these algebraic structures.\looseness=-1

As a natural extension of quantum affine algebras, quantum affine superalgebras were introduced to accommodate the $\mathbb{Z}_2$-grading through the incorporation of additional generators.
In~\cite{HYAM}, H.~Yamane introduced the Drinfeld--Jimbo presentations of quantum affine superalgebras, by considering the classified type A-G affine Lie superalgebras~\cite{VGD3,VANDE} as deformations of the universal enveloping algebras of the corresponding affine Lie algebras. In particular, using the method of Beck \cite{JBECK}, Yamane also provided the Drinfeld presentation including the complete Serre relations specifically for type A.
Quantum affine superalgebras possess a richer structure and representation theory due to their grading structure, with predominant focus on type A in research endeavors. While detailed enumerations are beyond the scope of this discussion, it is noteworthy that Cai, Wang, Wu and Zhao \cite{JSKW}, Zhang \cite{YZC}, and Fan, Hou and Shi \cite{HBKD} constructed the Drinfeld presentation of quantum affine superalgebras \smash{$U_q \bigl(\widehat{\mathfrak{gl}(m|n)}\bigr)$} using Frenkel--Ding's isomorphism theorem. However, these constructions do not explicitly present the complete Serre relations. Furthermore,
Zhang \cite{Huafeng} utilizes the $R$-matrix presentation of the quantum affine superalgebra associated with the Lie superalgebra $\mathfrak{gl}(m|n)$ to explore its finite-dimensional representations and their tensor products. Employing the Gauss decomposition, Lu \cite{K-Lu} established a direct and explicit isomorphism between the twisted $q$-Yangians and affine $i$quantum groups associated with symmetric pair of type AI.

While progress has been made in understanding the relations among these presentations, the specific relations for the quantum affine superalgebra are still an open question. Further investigation and research are required to unveil the connections and establish the desired isomorphisms. Exploring the relations and structure of the quantum orthosymplectic affine superalgebra through these presentations will undoubtedly provide valuable insights into its representation theory and algebraic properties, such as some work in~\cite{YZC2}. As for orthosymplectic Yangians, Frassek and Tsymbaliuk \cite{Frassek} studied the $R$-matrix presentations of orthosymplectic super Yangians and presented their Drinfeld presentations for any parity sequence, which
generalizing the results of \cite{AIM} of the standard parity sequence.
Recently, we have developed an efficient method for verifying the isomorphism between the Drinfeld--Jimbo and Drinfeld presentations of the quantum affine superalgebra of orthosymplectic Lie superalgebra with a~standard parity sequence, please see \cite{XWLHZ}. At the same time, Bezerra, Futorny and Kashuba \cite{Bezerra} also provided a Drinfeld presentation for the quantum affine superalgebra of type B with any parity sequence. They constructed a surjective homomorphism from this Drinfeld presentation to the Drinfeld--Jimbo presentation using a method as well as Beck's.
 Consequently, this paper will continue to focus on the $R$-matrix presentation of the quantum orthosymplectic affine superalgebra for a standard parity sequence, aiming to broaden the results of the quantum affine algebra to the super case. Specifically, we will establish an isomorphism between the Drinfeld presentation and the $R$-matrix presentation of the quantum affine superalgebra associated with the Lie superalgebra \smash{$\mathfrak{osp}_{2m+1|2n}$} ($m>0$).

The paper is organized as follows. In Section \ref{sec2}, we introduce the necessary notations and present the Drinfeld--Jimbo and Drinfeld formulations of the quantum orthosymplectic affine superalgebra. Additionally, we review the isomorphism established between the Drinfeld--Jimbo and Drinfeld formulations, as discussed in our previous works.
In Section \ref{sec3}, we discuss some results related to the universal $R$-matrix of $U_q\bigl[\mathfrak{osp}(2m+1|2n)^{(1)}\bigr]$, which holds significant importance in the theory of quantum affine superalgebras.
Section \ref{sec4} begins with the construction of a level-0 representation using Drinfeld generators. Consequently, we explicitly construct the $R$-matrix $R(z)$ and introduce a super version of the $R$-matrix algebras based on these explicit $R$-matrices.
Moving to Section \ref{sec5}, we establish the Drinfeld formulation within the $R$-matrix algebras by employing the Gaussian generators. This presentation facilitates a comprehensive exploration and analysis of the quantum affine superalgebra.
In Section \ref{sec6}, our main focus is on establishing the isomorphism between the Drinfeld and $R$-matrix presentations. Technically, we extend the methods discussed in \cite{JMA,JMAB} to the super case. Indeed, the original methods were provided by Frenkel and Mukhin \cite[Section 3.2]{EFEM} for the quantum affine algebra of type~A.

\section{Quantum affine superalgebra}\label{sec2}

\subsection{Basic notations of Lie superalgebra}
 Unless stated otherwise, throughout this paper, we consistently set $\mathfrak{g}=\mathfrak{osp}(2m+1|2n)$ and $\mathfrak{\hat{g}}=\mathfrak{osp}(2m+1|2n)^{(1)}$. First, we provide some notations on the set $\{1,\dots,n,n+1,\dots,n+2m+1,\dots,2n+2m+1\}$. Let the grading of $a$ be represented by $[a]$, such that
\begin{equation*}
 [a]=
 \begin{cases}
 0,&n+1\leq a\leq2m+n+1,\\
 1,& \text{otherwise},
 \end{cases}
\end{equation*}and the involution $\overline{a}=2n+2m+2-a$. Let $\varepsilon_1,\dots,\varepsilon_{n+m}$ is an orthogonal basis of a vector space,
 and then denote the invariant bilinear form on the set $\{\varepsilon_i,\, 1\leq i\leq2n+2m+1\}$ as follows:
\begin{equation*}
 (\varepsilon_i,\varepsilon_j)=-\delta_{ij},\qquad (\varepsilon_\mu,\varepsilon_\nu)=\delta_{\mu\nu}, \qquad 1\leq i,j\leq n,\qquad n+1\leq \mu,\nu\leq n+m.
\end{equation*}
The remaining symbols are indicated by $\varepsilon_{\overline{i}} =-\varepsilon_i$. In particular, we set
${\varepsilon_{n+m+1}=-\varepsilon_{n+m+1}=0}$.
As is well known, a Lie superalgebra is a $\mathbb{Z}_2$-graded algebra, denoted as $\mathfrak{g}=\mathfrak{g}_{\bar{0}}\oplus\mathfrak{g}_{\bar{1}}$, where the elements of $\mathfrak{g}_{\bar{0}}$ are referred to as even, and those of $\mathfrak{g}_{\bar{1}}$ as odd. For homogeneous elements~${X,Y\in \mathfrak{g}}$, the graded commutator is defined as
\begin{equation*}
[X, Y]_{a}=XY-(-1)^{[X][Y]}aYX,
\end{equation*}where $[X]\in \mathbb{Z}_2$,
 ensuring that $[X, Y]_{1}=[X, Y]$.
The tensor product multiplication is given by
$(X\otimes Y)(Z\otimes W)=(-1)^{[Y][Z]}(XZ\otimes YW)$.
 As for our notations, we adopt the following convention for the simple roots of $\mathfrak{g}$
 \begin{align*}
 \alpha_{i}=\varepsilon_i-\varepsilon_{i+1}, \qquad 1\leq i<n+m, \qquad \alpha_{n+m}=\varepsilon_{n+m}.
\end{align*}
The Cartan matrix \smash{$A=(A_{ij})_{i,j=1}^{m+n}$} of $\mathfrak{g}$ is defined by
\begin{equation*}
 A_{ij}=
 \begin{cases}
 (\alpha_i, \alpha_j), &i<m+n,\\
 2(\alpha_i, \alpha_j), &i=m+n.
 \end{cases}
\end{equation*}
Note that
the half sum of positive roots can be written as
\begin{align*}
 \rho=\frac{1}{2}\sum\limits_{\mu=1}^{n}(2n-2m+1-2\mu)\varepsilon_\mu+\frac{1}{2}\sum\limits_{i=1}^{m}(2m+1-2i)\varepsilon_{n+i}.
\end{align*}
Therefore, $(\rho,\alpha)=\frac{1}{2}(\alpha,\alpha)$ for all simple roots $\alpha$.
Set $\alpha_0=\delta-\theta$, where $\theta=2\varepsilon_1$ is the highest root of $\mathfrak{g}$. Then
\smash{$\hat{\Pi}=\{\alpha_0,\alpha_1,\ldots,\alpha_{n+m}\}$} is the affine root base of $\mathfrak{\hat{g}}$. Hence,the Cartan matrix \smash{$\hat{A}$} of $\mathfrak{\hat{g}}$ is derived by appending the 0th row and column, satisfying $A_{00}=-2$, $ A_{10}=2A_{01}=-2$, $ A_{j0}=A_{0j}=0$ for $1<j\leqslant n+m$.

\subsection[Drinfeld--Jimbo presentation U\_q(g)]{Drinfeld--Jimbo presentation $\boldsymbol{U_q(\mathfrak{\hat{g}})}$}

Let $q$ be a formal parameter and $\alpha$ be a root of $U_q(\mathfrak{\hat{g}})$, set
\begin{gather*}
 q_i := q^{\frac{|(\alpha_i, \alpha_i)|}{2}}\qquad \textrm{for}\quad(\alpha_i, \alpha_i)\neq0, \qquad q_i := q\qquad\textrm{for}\quad(\alpha_i, \alpha_i)=0,\qquad
 q_\alpha := q^{\frac{|(\alpha, \alpha)|}{2}},
\end{gather*}
where $i=1,\dots,n+m$, and for $ a \in \mathbb{Z}_+$,
\[
 [a]_i = \frac{q_i^a - q_i^{-a}}{q_i - q_i^{-1}}, \qquad [a]_i! = [1]_i \cdots [a-1]_i [a]_i.
 \]

We now recall the Drinfeld--Jimbo presentation of quantum affine superalgebra
$U_q(\mathfrak{\hat{g}})$, initially introduced by H. Yamane \cite{HYAM}.
\begin{Definition}
The quantum affine superalgebra $U_q(\mathfrak{\hat{g}})$ over $\mathbb{C}\bigl(q^{1/2}\bigr)$ is an associative superalgebra generated by \emph{Chevalley generators} $\chi_i^{\pm}\doteq\chi_{\alpha_i}^{\pm}$, $K_i\doteq K_{\alpha_i}$ for $i=0,1,\dots,m+n$ with the parity of $\bigl[\chi_i^{\pm}\bigr]=[\alpha_i]$ and $[K_i]=0$ and the following relations:
\begin{gather*}
 K_i^{\pm 1}K_i^{\mp 1}=1,\qquad K_iK_j=K_jK_i, \qquad
 K_i\chi_j^{\pm}K_i^{-1}=q_i^{\pm A_{ij}}\chi_j^{\pm}, \\
\bigl[\chi_i^+, \chi_j^-\bigr]=\delta_{ij}\frac{K_i-K_i^{-1}}{q_i-q_i^{-1}}, \qquad
\bigl[\chi_i^{\pm}, \chi_j^{\pm}\bigr]=0\qquad \textrm{for}\quad A_{ij}=0,\\
 \bigl[\!\!\bigl[\chi_i^{\pm},\bigl[\!\!\bigl[\chi_i^{\pm}, \chi_{i+1}^{\pm}\bigr]\!\!\bigr] \bigr]\!\!\bigr]=0\qquad\textrm{for}\quad i\neq n,m+n,
\\
 \bigl[\!\!\bigl[\chi_i^{\pm},\bigl[\!\!\bigl[\chi_i^{\pm}, \chi_{i-1}^{\pm}\bigr]\!\!\bigr] \bigr]\!\!\bigr]=0\qquad\textrm{for}\quad 1<i<m+n,\quad i\neq n,
\\
 \bigl[\!\!\bigl[\chi_i^{\pm}, \bigl[\!\!\bigl[\chi_i^{\pm},\bigl[\!\!\bigl[\chi_i^{\pm}, \chi_{i-1}^{\pm}\bigr]\!\!\bigr] \bigr]\!\!\bigr] \bigr]\!\!\bigr]=0\qquad\textrm{for}\quad i=1\quad\textrm{or}\quad m+n, \\
 \bigl[ \bigl[\!\!\bigl[ \bigl[\!\!\bigl[\chi_{n-1}^{\pm}, \chi_n^{\pm}\bigr]\!\!\bigr], \chi_{n+1}^{\pm}\bigr]\!\!\bigr], \chi_n^{\pm}\bigr]=0\qquad\textrm{for}\quad n>1, \\
 \bigl[ \bigl[\!\!\bigl[ \bigl[\!\!\bigl[ \bigl[\!\!\bigl[ \bigl[\!\!\bigl[ \bigl[\!\!\bigl[\chi_3^{\pm}, \chi_2^{\pm}\bigr]\!\!\bigr], \chi_1^{\pm}\bigr]\!\!\bigr], \chi_0^{\pm}\bigr]\!\!\bigr], \chi_1^{\pm}\bigr]\!\!\bigr], \chi_2^{\pm}\bigr]\!\!\bigr], \chi_1^{\pm}\bigr]=0
\qquad\textrm{for} \quad n=1,\quad m\geq 2,\\
 \bigl[\!\!\bigl[ \bigl[\!\!\bigl[\chi_2^{\pm}, \chi_1^{\pm}\bigr]\!\!\bigr], \bigl[\!\!\bigl[ \bigl[\!\!\bigl[\chi_2^{\pm}, \chi_1^{\pm}\bigr]\!\!\bigr], \bigl[\!\!\bigl[ \bigl[\!\!\bigl[\chi_2^{\pm}, \chi_1^{\pm}\bigr]\!\!\bigr], \chi_0^{\pm}\bigr]\!\!\bigr] \bigr]\!\!\bigr] \bigr]\!\!\bigr] \\
\qquad =(1-\bigl[2\bigr]_1)\bigl[\!\!\bigl[ \bigl[\!\!\bigl[ \bigl[\!\!\bigl[\chi_2^{\pm}, \chi_1^{\pm}\bigr]\!\!\bigr], \bigl[\!\!\bigl[\chi_2^{\pm}, \bigl[\!\!\bigl[\chi_2^{\pm}, \bigl[\!\!\bigl[\chi_1^{\pm}, \chi_0^{\pm}\bigr]\!\!\bigr] \bigr]\!\!\bigr] \bigr]\!\!\bigr] \bigr]\!\!\bigr], \chi_1^{\pm}\bigr]\!\!\bigr]\qquad\textrm{for}\quad (n,m)=(1,1),
\end{gather*}
where the notation $\bigl[\!\!\bigl[X_{\alpha}, X_{\beta}\bigr]\!\!\bigr]=[X_{\alpha}, X_{\beta}]_{q^{-(\alpha,\beta)}}$ if $K_iX_{\alpha}K_i^{-1}=q^{(\alpha_i,\alpha)}X_{\alpha}$, $K_iX_{\beta}K_i^{-1}=q^{(\alpha_i,\beta)}X_{\beta}$ for homogeneous elements $X_{\alpha},X_{\beta}\in U_q(\mathfrak{\hat{g}})$ and $i=1,\ldots,n+m$.
\end{Definition}

Let $U_q^{+}$ (resp.\ $U_q^{-}$) be the subalgebra of $U_q(\mathfrak{\hat{g}})$ generated by $\chi_i^{+}$ (resp.\ $\chi_i^{-}$), and $U_q^{0}$ be the subalgebra of $U_q(\mathfrak{\hat{g}})$ generated by $K_i$. Then we have the following triangular decomposition of~$U_q(\mathfrak{\hat{g}})$,
$U_q(\mathfrak{\hat{g}})=U_q^{-}\otimes U_q^{0}\otimes U_q^{+}$.

Quantum affine superalgebra $U_q(\mathfrak{\hat{g}})$
as a
Hopf superalgebra equipped with the comultiplication~$\Delta$, counit $\varepsilon$, and antipode $S$ defined as follows:
\begin{gather*}
 \Delta\bigl(\chi_i^{+}\bigr) = \chi_i^{+}\otimes 1 + K_i\otimes\chi_i^{+},\qquad
 \Delta(\chi_i^{-}) = \chi_i^{-}\otimes K_i^{-1} + 1\otimes\chi_i^{-}, \qquad
 \Delta(K_i) = K_i\otimes K_i,\\ \varepsilon\bigl(\chi_i^{\pm}\bigr) = 0,\qquad \varepsilon\bigl(K_i^{\pm}\bigr) = 1,\qquad
 S\bigl(\chi_i^{+}\bigr) = K_i^{-1}\chi_i^{+},\qquad S(\chi_i^{-}) = -\chi_i^{+}K_i, \\ S(K_i) = K_i^{-1}.
\end{gather*}

\subsection[The Drinfeld presentation U\_q(g)]{The Drinfeld presentation $\boldsymbol{\mathcal{U}_q(\mathfrak{\hat{g}})}$ }

We recall the Drinfeld presentation of the quantum affine superalgebra (see \cite{Bezerra, XWLHZ,HYAM}), which is expected to be isomorphic to the above Drinfeld--Jimbo presentation.

\begin{Definition}
The \emph{Drinfeld presentation} of quantum affine superalgebra denoted as $\mathcal{U}_q(\mathfrak{\hat{g}})$ over $\mathbb{C}\bigl(q^{1/2}\bigr)$ is an associative superalgebra generated by \emph{current generators} $x_{i,k}^{\pm}$, $a_{i,r}$, $k_i^{\pm 1}$, $i=1,\dots,n+m$ and the central element \smash{$q^{\pm \frac{1}{2}c}$}, with the following defining relations. The parity of generators $x_{i,k}^{\pm}$ is denoted by $\bigl[x_{i,k}^{\pm}\bigr]=[\alpha_i]$, while all other generators have parity~0,
\begin{gather*}
 q^{\pm\frac{1}{2}c}q^{\mp\frac{1}{2}c}=k_i^{\pm 1}k_i^{\mp 1}=1, \qquad k_ik_j=k_jk_i, \qquad
 k_ia_{j,r}=a_{j,r}k_i, \qquad k_ix_{j,k}^{\pm}k_i^{-1}=q_i^{\pm A_{ij}}x_{j,k}^{\pm}, \\
 [a_{i,r}, a_{j,s}]=\delta_{r,-s}\frac{[r A_{ij}]_{i}}{r}\cdot\frac{q^{r c}-q^{-r c}}{q_j-q_j^{-1}}, \qquad
 \bigl[a_{i,r}, x_{j,k}^{\pm}\bigr]=\pm\frac{[r A_{ij}]_{i}}{r}q^{\mp\frac{|r|}{2}c}x_{j,r+k}^{\pm}, \\
 \bigl[x_{i,k}^+, x_{j,l}^-\bigr]=\delta_{ij}\frac{q^{\frac{k-l}{2}c}\Phi_{i,k+l}^+-q^{\frac{l-k}{2}c}\Phi_{i,k+l}^-}{q_i-q_i^{-1}}, \\
 \bigl[x_{i,k+1}^{\pm}, x_{j,l}^{\pm}\bigr]_{q_i^{\pm A_{ij}}}+(-1)^{[i][j]}\bigl[x_{j,l+1}^{\pm}, x_{i,k}^{\pm}\bigr]_{q_i^{\pm A_{ij}}}=0 \qquad\textrm{if}\quad A_{ij}\neq0, \\
 \bigl[x_{i,k}^{\pm}, x_{j,l}^{\pm}\bigr]=0\qquad \textrm{if}\quad A_{ij}=0, \\
 \operatorname{Sym}_{k_1,k_2}\bigl[\!\!\bigl[x_{i,k_1}^{\pm},\bigl[\!\!\bigl[x_{i,k_2}^{\pm}, x_{i+s,l}^{\pm}\bigr]\!\!\bigr] \bigr]\!\!\bigr]=0\qquad\textrm{for}\quad i\neq n,m+n,\quad s=\pm 1, \\
 \operatorname{Sym}_{k_1,k_2,k_3}\bigl[\!\!\bigl[x_{m+n,k_1}^{\pm}, \bigl[\!\!\bigl[x_{m+n,k_2}^{\pm},\bigl[\!\!\bigl[x_{m+n,k_3}^{\pm}, x_{m+n-1,l}^{\pm}\bigr]\!\!\bigr] \bigr]\!\!\bigr] \bigr]\!\!\bigr]=0, \\
 \operatorname{Sym}_{l_1,l_2}\bigl[ \bigl[\!\!\bigl[ \bigl[\!\!\bigl[x_{n-1,r_1}^{\pm}, x_{n,l_1}^{\pm}\bigr]\!\!\bigr], x_{n+1,k_2}^{\pm}\bigr]\!\!\bigr], x_{n,l_2}^{\pm}\bigr]=0
\qquad\textrm{for}\quad n>1,
\end{gather*}
where $\Phi_{i,\pm r}^{\pm}$($r\geqslant 0$) is given by the formal power series
\begin{gather*}
 \sum\limits_{r\geqslant 0}\Phi_{i,\pm r}^{\pm}z^{\pm r}:=k_i^{\pm 1}\exp\biggl(\pm \bigl(q_i-q_i^{-1}\bigr)\sum\limits_{r>0}a_{i,\pm r}z^{\pm r}\biggr).
\end{gather*}
\end{Definition}

In our previous works (refer to \cite{XWLHZ}), we introduce the affine root vectors denoted as $\mathfrak{E}_{\alpha+r\delta}$, $\mathfrak{F}_{\alpha+r\delta}$, $\mathfrak{E}_{r\delta^{(i)}}$ and $\mathfrak{F}_{r\delta^{(i)}}$ of quantum superalgebra using braid group actions. Here $\alpha$ runs over the positive roots \smash{$\widehat{\Delta}_+$}. For our purpose, we review the
isomorphism between Drinfeld presentation~$\mathcal{U}_q(\mathfrak{\hat{g}})$ and Drinfeld--Jimbo presentation $U_q(\mathfrak{\hat{g}})$ as follows.
\begin{Theorem}[{\cite[Theorem 3.6]{XWLHZ}}]\label{DR-DRJ}
There exists an isomorphism between the Drinfeld presentation~$\mathcal{U}_q(\mathfrak{\hat{g}})$ and the Drinfeld--Jimbo presentation $U_q(\mathfrak{\hat{g}})$. The isomorphism is expressed in terms of root vectors as
\begin{gather*}
q^{\pm\frac{1}{2}c} \mapsto K_{\delta}^{\pm\frac{1}{2}},\qquad k_i^{\pm 1} \mapsto K_i^{\pm 1}\qquad \text{for}\quad i=1,\ldots,m+n, \\
x_{i,r}^+ \mapsto \mathfrak{E}_{r\delta+\alpha_i},\quad x_{i,-r}^+ \mapsto (-d_i)^rd_{i+1}\mathfrak{F}_{r\delta-\alpha_i}K_{\delta}^{r}K_i^{-1}, \qquad r\geq0,\\
x_{i,r}^- \mapsto K_{\delta}^{-r}K_i\mathcal{E}_{r\delta-\alpha_i},\qquad x_{i,-r}^- \mapsto (-d_i)^rd_{i+1}\mathfrak{F}_{r\delta+\alpha_i}, \qquad r\geq0,\\
a_{i,r} \mapsto K_{\delta}^{-\frac{r}{2}}\mathfrak{E}_{r\delta^{(i)}},\qquad a_{i,-r} \mapsto (-d_i)^rK_{\delta}^{\frac{r}{2}}\mathfrak{F}_{r\delta^{(i)}},\qquad r>0,
\end{gather*}
where $d_i=(\varepsilon_i,\varepsilon_i)$. Or in terms of the Drinfeld--Jimbo and Drinfeld generators
\begin{gather*}
\chi_i^{\pm}\mapsto x_{i,0}^{\pm},\qquad i=1,\ldots,m+n, \qquad
K_0^{\pm 1}\mapsto \bigl(q^ck_1^2k_2^2\cdots k_{m+n}^2\bigr)^{\pm 1},\qquad K_{\delta}^{\pm\frac{1}{2}}\mapsto q^{\pm\frac{1}{2}c}, \\
\chi_0^+\mapsto \nu_0^+ \bigl(q^ck_1^2k_2^2\cdots k_{m+n}^2\bigr) \bigl[\!\!\bigl[\dots\bigl[\!\!\bigl[x_{1,1}^-,x_{2,0}^-\bigr]\!\!\bigr],\ldots,x_{m+n,0}^-\bigr]\!\!\bigr],x_{m+n,0}^-\bigr]\!\!\bigr],\ldots,x_{1,0}^-\bigr]\!\!\bigr], \\
\chi_0^-\mapsto \nu_0^- \bigl[\!\!\bigl[\dots\bigl[\!\!\bigl[x_{1,-1}^+,x_{2,0}^+\bigr]\!\!\bigr],\ldots,x_{m+n,0}^+\bigr]\!\!\bigr],x_{m+n,0}^+\bigr]\!\!\bigr],\ldots,x_{1,0}^+\bigr]\!\!\bigr] \bigl(q^ck_1^2k_2^2\cdots k_{m+n}^2\bigr)^{-1},
\end{gather*}
where $\nu_0^+=-([2]_N)^{-1}$, $\nu_0^-=(-1)^{[\alpha_1]}([2]_N)^{-1}q^{2n-2m-(\alpha_1,\alpha_2)}$.
\end{Theorem}

\section[The universal R-matrix]{The universal $\boldsymbol{R}$-matrix}\label{sec3}

Consider the extended algebra $\widetilde{U}_q(\mathfrak{\hat{g}})$ of $U_q(\mathfrak{\hat{g}})$, which is obtained by adjoining an additional element $d$ with the relations
$ \bigl[d, \chi_i^{\pm}\bigr]=\pm \delta_{i,0}\chi_i^{\pm}$, $ [d, k_i]=0$.
The algebra \smash{$\widetilde{U}_q(\mathfrak{\hat{g}})$} is also a Hopf superalgebra, possessing the same comultiplication
$\Delta$, counit~$\varepsilon$ and antipode $S$ as those of $U_q(\mathfrak{\hat{g}})$, and
\begin{gather*}
 \Delta(d)=d\otimes1+1\otimes d,\qquad \varepsilon(d)=1,\qquad S(d)=d^{-1}.
\end{gather*}
By Drinfeld double construction, the universal $R$-matrix $\mathfrak{R}$ of \smash{$\widetilde{U}_q(\mathfrak{\hat{g}})$} is a solution of the Yang--Baxter equation
$\mathfrak{R}_{12}\mathfrak{R}_{13}\mathfrak{R}_{23}=\mathfrak{R}_{23}\mathfrak{R}_{13}\mathfrak{R}_{12}$,
and satisfies the coproduct properties
\begin{gather*}
 (\Delta\otimes 1)\mathfrak{R}=\mathfrak{R}_{13}\mathfrak{R}_{23},\qquad (1\otimes \Delta)\mathfrak{R}=\mathfrak{R}_{13}\mathfrak{R}_{12}, \qquad
 \mathfrak{R}\Delta(X)=\Delta^T(X)\mathfrak{R}, \qquad X\in\widetilde{U}_q(\mathfrak{\hat{g}}),
\end{gather*}
where $\Delta^T=T\Delta$ and $T(X\otimes Y)=(-1)^{[X][Y]}Y\otimes X$.

Consider a formal variable $z$ and define an automorphism $D_z$ of $\widetilde{U}_q(\mathfrak{\hat{g}})\otimes\mathbb{C}\bigl[z, z^{-1}\bigr]$ as follows:
\begin{gather}\label{Dz}
 D_z(\chi_i^{\pm})=z^{\pm\delta_{i,0}}\chi_i^{\pm}, \qquad D_z(k_i)=k_i, \qquad D_z(d)=d.
\end{gather}
We define a universal $R$-matrix $\mathfrak{R}(z)$ that depends on the spectral parameter $z$ using the formula~${
\mathfrak{R}(z)=(D_z\otimes1)\mathfrak{R}q^{c\otimes d+d\otimes c}}$.
It then satisfies the following Yang--Baxter equation
\begin{gather}\label{oper0}
 \mathfrak{R}_{12}(z)\mathfrak{R}_{13}(zwq^{-c_2})\mathfrak{R}_{23}(w)=\mathfrak{R}_{23}(w)\mathfrak{R}_{13}(zwq^{c_2})\mathfrak{R}_{12}(z),
\end{gather}where $c_2=1\otimes c\otimes1$.
Furthermore, the universal $R$-matrix $\mathfrak{R}(z)$ satisfies the following properties:
\begin{gather*}
 (S\otimes1)(\mathfrak{R}(z))=\mathfrak{R}\bigl(zq^{-c\otimes1}\bigr)^{-1},\qquad (1\otimes S^{-1})(\mathfrak{R}(z))=\mathfrak{R}\bigl(zq^{1\otimes c}\bigr)^{-1}.
\end{gather*}

Let $\pi$ denote a representation of $\widetilde{U}_q(\mathfrak{\hat{g}})$, for any two finite-dimensional modules $V$ and $W$, we define an operator as follows
$
R^{VW}(z)=(\pi_V\otimes\pi_W)(\mathfrak{R}(z))$.
Since for any finite-dimensional representation, $\pi_V(c)=\pi_W(c)=0$, it follows from \eqref{oper0} that~$R^{VW}(z)$ satisfies the Yang--Baxter equation.
We define right dual module $V^\ast$ and left dual module $^\ast V$ as follows:
\begin{gather*}
 \pi_{V^\ast}(a)=\pi_V(S(a))^{\rm st},\qquad \pi_{^\ast V}(a)=\pi_V\bigl(S^{-1}(a)\bigr)^{\rm st},
\end{gather*}
where ${\rm st}$ denotes the super-transposition operation on the module $V$ such that
\[(E^{a}_b)^{\rm st}=(-1)^{[a]([a]+[b])}E^{b}_a,
\]
here $E^{i}_j\in \operatorname{End}V$ is the elementary matrix with 1 in the $(i,j)$ position and zeros elsewhere.

Let $\pi_V$ denote a finite-dimensional representation of $U_q(\mathfrak{\hat{g}})$, and let $D_z$ be the automorphism defined in equation \eqref{Dz} of $U_q(\mathfrak{\hat{g}})\otimes\mathbb{C}\bigl[z, z^{-1}\bigr]$. Then we can define a representation
\begin{gather*}
 \pi_{V(z)}\colon\ U_q(\mathfrak{\hat{g}})\rightarrow \operatorname{End}(V)\otimes\mathbb{C}\bigl[z, z^{-1}\bigr]
\end{gather*}
by setting $\pi_{V(z)}(a)=\pi_{V}(D_z(a)), a\in U_q(\mathfrak{\hat{g}})$.

Let $h_{\hat{\rho}}$ denote the unique element of the Cartan subalgebra of $\mathfrak{\hat{g}}$ satisfying $h_{\hat{\rho}}(\alpha_i)=\frac{1}{2}(\alpha_i, \alpha_i)$. We define $h_\rho$ as
$
 h_\rho=h_{\hat{\rho}}-gd$,
where $g=\frac{1}{2}(\theta, \theta+2\rho)$.
For the representation $V^{\ast\ast}(z)$ and $^{\ast\ast}V(z)$, the square of the antipode is given by
\begin{gather}\label{SQUARE}
 S^2(a)=q^{-2h_{\rho}}D_{q^{-2g}}(a)q^{2h_{\rho}}, \qquad S^{-2}(a)=q^{2h_{\rho}}D_{q^{2g}}(a)q^{-2h_{\rho}},\qquad a\in U_q(\mathfrak{\hat{g}}).
\end{gather}

\begin{Proposition} With the notations established above, we obtain
 \begin{gather*}
 V(z)^{\ast\ast}\tilde{\rightarrow}V(zq^{-g}), \qquad ^{\ast\ast}W(z)\tilde{\rightarrow}W(zq^{g}), \qquad v\mapsto q^{2h_\rho}v, \qquad w\mapsto q^{-2h_\rho}w.
 \end{gather*}
\end{Proposition}
\begin{proof}
It is straightforward to check the action on generators by the antipodes \eqref{SQUARE}.
\end{proof}

\begin{Proposition}[{\cite[equations (2.30) and (2.32)]{MDGY}}]\label{R:V-W} The following equations hold that\samepage
\begin{itemize}\itemsep=0pt
\item[$(1)$] $R^{V^\ast, W}(z)=\bigl(R^{VW}(z)^{-1}\bigr)^{{\rm st}_1}$, $R^{V, ^\ast W}(z)=\bigl(R^{VW}(z)^{-1}\bigr)^{{\rm st}_2}$,
\item[$(2)$] $ \bigl(\bigl(\bigl(R^{VW}(z)^{-1}\bigr)^{{\rm st}_1}\bigr)^{-1}\bigr)^{{\rm st}_1}=\bigl(\pi_V\bigl(q^{-2h_\rho}\bigr)\otimes1_W\bigr)
\bigl(\bigl(R^{VW}\bigl(zq^{-2g}\bigr)\bigr)^{{\rm st}_1}\bigr)^{{\rm st}_1}\bigl(\pi_V\bigl(q^{2h_\rho}\otimes1_W\bigr)\bigr)$,
\item[$(3)$] $ \bigl(\bigl(\bigl(R^{VW}(z)^{-1}\bigr)^{{\rm st}_2}\bigr)^{-1}\bigr)^{{\rm st}_2}=\bigl(1_V\otimes\pi_W(q^{2h_\rho}\bigr)\bigr)
\bigl(\bigl(R^{VW}\bigl(zq^{2g}\bigr)\bigr)^{{\rm st}_2})^{{\rm st}_2}\bigl(1_V\otimes\pi_W\bigl(q^{-2h_\rho}\bigr)\bigr)$.
\end{itemize}
\end{Proposition}

\begin{Theorem}\label{RVWQVW}
Let $V$ and $W$ be two finite-dimensional irreducible $U_q(\mathfrak{\hat{g}})$-module. Then the operator $R^{VW}(z)$ is given by the formula
\begin{gather}\label{formu1}
 R^{VW}(z)=f_{VW}(z)Q^{VW}(z),
\end{gather}where $Q^{VW}(z)$ is a matrix polynomial over $z$ without common zeros. The function $f_{VW}(z)$ is a meromorphic function on $\mathbb{C}$ such that $f_{VW}(0)=1$ and $f_{VW}(0)\sim z^{-p(V,W)}$, where $p(V,W)$ is the degree of the polynomial $Q^{VW}(z)$. Moreover,
\begin{gather}\label{formu2}
 f_{VW}(z)=\prod\limits_{i=1}^{p(V,W)}\frac{\bigl(zq^{a_i}; q^{-2g}\bigr)_\infty}{\bigl(zq^{b_i}; q^{-2g}\bigr)_\infty},
\end{gather}
which are unique over $\mathbb{C}[[z]]\otimes\mathbb{C}\bigl[\bigl[q^{1/2}\bigr]\bigr]$, where
\begin{gather*}
 (z; q)_\infty=\prod\limits_{n\geq0}(1-zq^n), \qquad \sum_{i=1}^{p(V,W)}(a_i-b_i)=2gp(V,W),\qquad a_i, b_i\in\mathbb{C}.
\end{gather*}
\end{Theorem}

\begin{proof}
Let us introduce the permutation operator $P^{VW}$ on the tensor product module $V\otimes W$: $P^{VW}(v_a\otimes v_b)=(-1)^{[a][b]}(v_b\otimes v_a)$, $\forall$ $v_a\in V$, $v_b\in W$. Consider the irreducible modules $V(z)\otimes W$ and $W\otimes V(z)$, where $z$ is a formal variable. Note that $P^{VW}R^{VW}(z)$ is an intertwining operator: $V(z)\otimes W\rightarrow W\otimes V(z)$, and it is unique up to scalar factor.
Hence, the following equality holds:
\begin{gather*}
 R^{VW}(z)(\pi_V\otimes\pi_W)(D_z\otimes1)(\Delta(a))=(\pi_V\otimes\pi_W)(D_z\otimes1)\bigl(\Delta^T(a)\bigr)R^{VW}(z)
\end{gather*}
for $a=\chi_i^{\pm}$, $K_i$, $i=0,\dots,n+m$. This equation, linear over $z$, $z^{-1}$, dictates that the factorized representation \eqref{formu1} and $Q^{VW}(z)$ are uniquely determined up to a constant. We determine this constant by imposing the condition $f_{VW}(0)=1$.

Let us consider intertwiners $V^{\ast\ast}(z)\otimes W\rightarrow W\otimes V^{\ast\ast}(z)$, where each such intertwining operator varies by a scalar multiplier determined by the irreducible modules $V(z)^{\ast\ast}\otimes W$ and $W\otimes V^{\ast\ast}(z)$. From the definition of $V^{\ast\ast}$, we derive the intertwining operator \smash{$P^{V^{\ast\ast},W} \bigl(\bigl(\bigl(Q^{VW}(z)^{-1}\bigr)^{{\rm st}_1}\bigr)^{-1}\bigr)^{{\rm st}_1}$}. Conversely, the isomorphism $V(z)^{\ast\ast}\cong V\bigl(zq^{-2g}\bigr)$ yields another intertwining operator given by
 \smash{$P^{V^{\ast\ast},W}\bigl(\pi_V\bigl(q^{-2h_\rho}\bigr)\otimes1_W\bigr)
\bigl(\bigl(Q^{VW}(z)\bigr)^{{\rm st}_1}\bigr)^{{\rm st}_1}\bigl(\pi_V\bigl(q^{2h_\rho}\otimes1_W\bigr)\bigr)$}. Thus, there exist rational functions~$r_{VW}(z)$ such that
\begin{gather}
 \bigl(\bigl(\bigl(Q^{VW}(z)^{-1}\bigr)^{{\rm st}_1}\bigr)^{-1}\bigr)^{{\rm st}_1}\nonumber\\
 \qquad=r_{VW}(z)\bigl(\pi_V\bigl(q^{-2h_\rho}\bigr)\otimes1_W\bigr)
\bigl(\bigl(Q^{VW}\bigl(zq^{-2g}\bigr)\bigr)^{{\rm st}_1}\bigr)^{{\rm st}_1}\bigl(\pi_V\bigl(q^{2h_\rho}\otimes1_W\bigr)\bigr).\label{formu3}
\end{gather}
Let $p(V,W)$ denote the degree of the polynomial $Q^{VW}(z)$, then
\begin{gather*}
 r_{VW}(0)=1,\qquad r_{VW}(z)\cong q^{p(V,W)2g},\qquad z\rightarrow\infty.
\end{gather*}
By Proposition \ref{R:V-W}\,(2) and \eqref{formu3}, we derive
\begin{gather}\label{formu4}
 f_{VW}\bigl(zq^{-2g}\bigr)= r_{VW}(z) f_{VW}(z),
\end{gather}
with $f_{VW}(0)=1$.
Let
\begin{gather*}
 r_{VW}(z)=\prod\limits_{i=1}^{p(V,W)}\frac{1-zq^{a_i}}{1-zq^{b_i}},
\end{gather*}
through a straightforward computation, equation \eqref{formu4} admits a unique solution over $\mathbb{C}[[z]]\otimes\mathbb{C}\bigl[\bigl[q^{1/2}\bigr]\bigr]$, in the form provided by equation \eqref{formu2}.
\end{proof}

\section[R-matrix algebras]{$\boldsymbol{R}$-matrix algebras}\label{sec4}
Utilizing the Drinfeld generators, we will formulate a level-$0$ representation of the quantum affine superalgebra dependent on the spectral parameter $z$. This representation encompasses a~vector representation when $z$ is regarded as a spectral constant. Consequently,
by this vector representation, we can explicitly derive an $R$-matrix denoted as $R(z)$, satisfying the Yang--Baxter equation
$R_{12}(z)R_{13}(zw)R_{23}(w)=R_{23}(w)R_{13}(zw)R_{12}(z)$. The explicit expression of the $R$-matrix
$R(z)$ enables us to investigate and analyze the quantum affine superalgebra, facilitating the construction of a super version of the $R$-matrix algebras. This super version corresponds to the non-super case of the $R$-matrix algebra of quantum affine algebra introduced by Reshetikhin--Semenov-Tian-Shansky (see~\cite{RS}).

\subsection[The explicit R-matrix of R(z)]{The explicit $\boldsymbol{R}$-matrix of $\boldsymbol{R(z)}$ }\label{sec4.1}

For the sake of convenience, we adopt the following notation
\smash{$\nu_i=\sum_{j=1}^{i}d_j$}, $ d_i=(\varepsilon_i,\varepsilon_i)$.
\begin{Proposition}[level-0 representation]\label{Level-0}
 Consider the graded vector space $V=\mathbb{C}^{2m+1|2n}$. The following map gives a representation $\pi_{V_{(z)}}$ of the quantum affine superalgebra $\mathcal{U}_q(\mathfrak{\hat{g}})$ on ${\operatorname{End}(V)\otimes\mathbb{C}\bigl[z, z^{-1}\bigr]}$:
\begin{gather*}
q^{c/2}\mapsto1,\qquad
 x_{i,k}^{-}\mapsto \bigl(zq^{\nu_i}\bigr)^kE^{i+1}_i-\bigl(zq^{2m-2n-\nu_i-1}\bigr)^kE^{\overline{i}}_{\overline{i+1}},\\
 x_{i,k}^{+}\mapsto \bigl(zq^{\nu_i}\bigr)^kE^{i}_{i+1}-\bigl(zq^{2m-2n-\nu_i-1}\bigr)^kE^{\overline{i+1}}_{\overline{i}},\\
 k_i\mapsto d_iq^{d_{i}}\bigl(E^{i}_{i}+E^{\overline{i+1}}_{\overline{i+1}}\bigr)+d_{i+1}q^{-d_{i+1}}\bigl(E^{i+1}_{i+1}+E^{\overline{i}}_{\overline{i}}\bigr)
 +\sum\limits_{s\neq i,\overline{i},i+1,\overline{i+1}}d_sE^{s}_{s},\\
 a_{i,k}\mapsto \frac{[k]_{q_i}}{k}\bigl( \bigl(zq^{\nu_i}\bigr)^k\bigl(d_iq^{-d_ik}E^{i}_{i}-d_{i+1}q^{d_{i+1}k}E^{i+1}_{i+1}\bigr)+
 \bigl(zq^{2m-2n-\nu_i-1}\bigr)^k\\
\phantom{ a_{i,k}\mapsto}{} \times\bigl(d_{i+1}q^{-d_{i+1}k}E^{\overline{i+1}}_{\overline{i+1}}-d_{i}q^{d_{i}k}E^{\overline{i}}_{\overline{i}}\bigr)\bigr)
\end{gather*}
for $1\leq i< m+n-1$, and
\begin{gather*}
 x_{n+m,k}^{-}\mapsto [2]_{q_{m+n}}^{1/2}\Bigl( (zq^{m-n})^kE^{n+m+1}_{n+m}-\bigl(zq^{m-n-1}\bigr)^kE^{\overline{n+m}}_{\overline{n+m+1}}\Bigr), \\
 x_{n+m,k}^{+}\mapsto [2]_{q_{m+n}}^{1/2}\Bigl( (zq^{m-n})^kE^{n+m}_{n+m+1}-\bigl(zq^{m-n-1}\bigr)^kE^{\overline{n+m+1}}_{\overline{n+m}}\Bigr), \\
 k_{n+m}\mapsto qE^{n+m}_{n+m}+q^{-1}E^{\overline{n+m}}_{\overline{n+m}}
 +\sum\limits_{s\neq n+m,\overline{n+m}}d_sE^{s}_{s},\\
 a_{n+m,k}\mapsto \frac{[2k]_{q_{n+m}}}{k}\Bigl( -\bigl(zq^{m-n-1}\bigr)^kE^{n+m}_{n+m}+\bigl((zq^{m-n})^k-\bigl(zq^{m-n-1}\bigr)^k\bigr)E^{n+m+1}_{n+m+1}\bigr)\\
 \phantom{ a_{n+m,k}\mapsto}{}+
 (zq^{m-n})^kE^{\overline{n+m}}_{\overline{n+m}}\Bigr).
\end{gather*}
\end{Proposition}
\begin{proof}
It is straightforward to check the action on the generators.
\end{proof}

Notice that we have the equivalence $V(1)=V$ by setting $z=1$. Therefore, it gives rise to a~vector representation
 $\pi_V\colon \mathcal{U}_q(\mathfrak{\hat{g}})\rightarrow \operatorname{End}(V)$ from the above proposition.
Let
$R(z)\doteq R^{VV}(z)=(\pi_V\otimes\pi_V)\mathfrak{R}(z)$,
where $\mathfrak{R}(z)$ is the universal $R$-matrix of $\mathcal{U}_q(\mathfrak{\hat{g}})$ via the isomorphism Theorem~\ref{DR-DRJ}.

{For the formal variable $z$, denote
$z_{\pm}=zq^{\pm c/2}$, and introduce the $L$-operators in $\mathcal{U}_q(\mathfrak{\hat{g}})$ by the formulas}
\begin{gather}\label{L-operators}
 \mathfrak{L}^+(z)=(1\otimes \pi_V)\mathfrak{R}(z_-), \qquad \mathfrak{L}^-(z)=(1\otimes \pi_V)\mathfrak{R}_{21}\bigl(z_-^{-1}\bigr)^{-1}.
\end{gather}
Therefore, the Yang--Baxter equation implies the following proposition.

\begin{Proposition}\label{U(ENDV)} In $\mathcal{U}_q(\mathfrak{\hat{g}})\otimes \operatorname{End}V^{\otimes2}$, we have
\begin{align*}
 &R(z/w)\mathfrak{L}_1^{\pm}(z)\mathfrak{L}_2^{\pm}(w)=\mathfrak{L}_2^{\pm}(w)\mathfrak{L}_1^{\pm}(z) R(z/w),\\
 &R(z_+/w_-)\mathfrak{L}_1^{+}(z)\mathfrak{L}_2^{-}(w)=\mathfrak{L}_2^{-}(w)\mathfrak{L}_1^{+}(z) R(z_-/w_+).
\end{align*}
\end{Proposition}

Next, we aim to provide an explicit $R$-matrix
$R(z)$ in the form of a matrix polynomial.
 To achieve this, we begin by considering the polynomials $Q^{VV}$ in Theorem \ref{RVWQVW} for $V=W=\mathbb{C}^{2m+1|2n}$ as follows: for $\zeta=q^{2m-2n-1}$,
\begin{gather*}
 Q^{VV}(z)=\bigl(\bigl(1-q^{2}\bigr)(z-\zeta)zP+\bigl(q^2-1\bigr)(z-1)(z-\zeta)Q+q(z-1)(z-\zeta)\nonumber\\
\phantom{ Q^{VV}(z)=}{}\times \biggl\{I+
 \bigl(q^{\frac{1}{2}}-q^{-\frac{1}{2}}\bigr)\sum\limits_{a\neq \overline{a}}(-1)^{[a]}E^a_a\otimes\hat{\sigma}^a_a+\bigl(q-q^{-1}\bigr)\sum\limits_{a> b}(-1)^{[b]}E^a_b\otimes \hat{\sigma}^b_a\biggr\}\biggr).
\end{gather*}
The graded operators of above are given as follows:
\[
 P=\sum\limits_{a,b}(-1)^{[b]}E^a_b\otimes E^b_a,\qquad
 Q=\sum\limits_{a,b}(-1)^{[a][b]}\xi_a\xi_bq^{(\rho,\varepsilon_a- \varepsilon_b)}E^a_b\otimes E^{\bar{a}}_{\bar{b}},
\]
and
$\hat{\sigma}^a_b=E^a_b-(-1)^{[a]([a]+[b])}\xi_a\xi_bE^{\bar{b}}_{\bar{a}}$, $
 \hat{\sigma}^a_a=q^{1/2(\varepsilon_a, \varepsilon_a)}E^a_a-q^{-1/2(\varepsilon_a, \varepsilon_a)}E^{\bar{a}}_{\bar{a}}$,
 where \[
 \xi_a=
 \begin{cases}
 1,&[a]=0,\\
 (-1)^a,& 1\leq a\leq n,\\
 -(-1)^a,& \overline{n}\leq a\leq \overline{1}.
 \end{cases}
\]

\begin{Remark}
Let $\widetilde{Q}(z)=PQ^{VV}(z)$. According to \cite{WGMJ}, $\widetilde{Q}(z)$ satisfies the inversion relation
\begin{gather*}
 \widetilde{Q}(z)\widetilde{Q}\bigl(z^{-1}\bigr)=(z-\zeta)\bigl(z-q^2\bigr)\bigl(z^{-1}-\zeta\bigr)\bigl(z^{-1}-q^2\bigr)\times I.
\end{gather*}
\end{Remark}

Let $t$ represent the matrix involution super-transposition defined by
$
\bigl(E_j^i\bigr)^t=(-1)^{[i][j]+[j]}\xi_{\overline{i}}\xi_{\overline{j}}\allowbreak\times\smash{ E_{\overline{i}}^{\overline{j}}}$,
 and define the diagonal matrix
 \begin{gather*}
 D=\operatorname{diag}\bigl[q^{a_1},\dots,q^{a_n},q^{a_{n+1}},\dots, q^{a_{n+m+1}}=q^{a_{\overline{(n+m+1)}}},\dots,
 q^{a_{\overline{n+1}}},q^{a_{\overline{n}}},\dots,q^{a_{\overline{1}}}\bigr],
 \end{gather*}
 where
\begin{gather*}
 a_{\overline{(n+m+1)}}=a_{n+m+1}=0,\qquad a_i=-a_{\overline{i}} \qquad i\neq n+m+1,\qquad (\rho, \varepsilon_i-\varepsilon_j)=a_{\bar{i}}-a_{\bar{j}}.
\end{gather*}
Denote $t_s$ as the transposition with the $s$-th tensor space.
Consequently, we can demonstrate that
\begin{gather}\label{d5}
 Q=D_1^{-1}P^{t_1}D_1, \qquad P^{t_1}D_1=P^{t_1}D_2^{-1}, \qquad D_2P^{t_1}=D_1^{-1}P^{t_1}.
\end{gather}

Furthermore, in accordance with the reference \cite{MKMJ}, we introduce a new $R$-matrix as follows:
\begin{align}\label{R1}
 \widetilde{R}(z)={}&\frac{\bigl(q-q^{-1}\bigr)zP}{\bigl(q-q^{-1}z\bigr)}-\frac{\bigl(q-q^{-1}\bigr)z(z-1)Q}{\bigl(q-q^{-1}z\bigr)(z-\zeta)}-\frac{(z-1)}{\bigl(q-q^{-1}z\bigr)}
\nonumber\\
 &\times\biggl\{I+\bigl(q^{\frac{1}{2}}-q^{-\frac{1}{2}}\bigr)\sum\limits_{a\neq \overline{a}}(-1)^{[a]}E^a_a\otimes\hat{\sigma}^a_a+\bigl(q-q^{-1}\bigr)\sum\limits_{a> b}(-1)^{[b]}E^a_b\otimes \hat{\sigma}^b_a\biggr\}.
\end{align}
Consider the function $g(z)=f(z)\bigl(z-q^2\bigr)(z-\zeta)$. Consequently, we obtain the expression
\begin{gather}\label{rq(z)}
 R(z)=f(z)Q^{VV}(z)=g(z)\widetilde{R}(z).
\end{gather}
It is worth noting that we can compute
\begin{gather}
\label{d7}
 \widetilde{R}(z)D_1\widetilde{R}(z\zeta)^{t_1}D_1^{-1}=y(z)=\frac{\bigl(z-q^{-2}\bigr)(z\zeta-1)}{(1-z)\bigl(1-q^{-2}\zeta z\bigr)}.
\end{gather}
By performing similar calculations on $Q^{VV}(z)$ and combining with equation \eqref{rq(z)}, we derive
\begin{gather}\label{d6}
R(z)D_1R(z\zeta)^{t_1}D_1^{-1}=q^2\zeta^2,
\end{gather}
and
\[
 f(z)f(z\zeta)=\frac{1}{\bigl(1-zq^2\bigr)\bigl(1-zq^{-2}\bigr)\bigl(1-z\zeta^{-1}\bigr)(1-z\zeta)}.
\]
Moreover, based on Theorem \ref{RVWQVW}, the meromorphic function $f(z)$ takes the form
\begin{gather*}
 f(z)\doteq \prod\limits_{i=0}^{\infty}\frac{\bigl(1-z\zeta^{2i}\bigr)\bigl(1-zq^{-2}\zeta^{2i+1}\bigr)\bigl(1-zq^{2}\zeta^{2i+1}\bigr)\bigl(1-z\zeta^{2i+2}\bigr)} {\bigl(1-z\zeta^{2i-1}\bigr)\bigl(1-z\zeta^{2i+1}\bigr)\bigl(1-zq^{2}\zeta^{2i}\bigr)\bigl(1-zq^{-2}\zeta^{2i}\bigr)},
\end{gather*}
Hence, the explicit form of the $R$-matrix $R(z)$ is as follows:
\begin{align*}
 R(z)={}&f(z)\bigl(\bigl(1-q^{2}\bigr)(z-\zeta)zP+\bigl(q^2-1\bigr)z(z-1)Q+q(z-1)(z-\zeta)\\
 &\times\biggl\{I+
\bigl(q^{\frac{1}{2}}-q^{-\frac{1}{2}}\bigr)\sum\limits_{a\neq \overline{a}}(-1)^{[a]}E^a_a\otimes\hat{\sigma}^a_a+\bigl(q-q^{-1}\bigr)\sum\limits_{a> b}(-1)^{[b]}E^a_b\otimes \hat{\sigma}^b_a\biggr\}\biggr).
\end{align*}

\subsection[The superalgebras U(R) and U(R)]{The superalgebras $\boldsymbol{U(R)}$ and $\boldsymbol{U\bigl(\widetilde{R}\bigr)}$ }

\begin{Definition}\quad
\begin{enumerate}\itemsep=0pt
\item[(1)] The associative superalgebra $U(R)$ over $\mathbb{C}\bigl(q^{1/2}\bigr)$ is generated by an invertible central element $q^{c/2}$ and elements $l_{ij}^{\pm}[\mp p]$, where the indices satisfy $ 1\leq i,j\leq2n+2m+1$ and~${p\in\mathbb{Z}_+}$,
 subject to the following relations:
\begin{gather}
 l^+_{ii}[0]l^-_{ii}[0]=l^-_{ii}[0] l^+_{ii}[0]=1, \qquad l^+_{ij}[0]=l^-_{ij}[0]=0,\qquad \text{for} \quad i>j,\nonumber\\
 R(z/w)L_1^{\pm}(z)L_2^{\pm}(w)=L_2^{\pm}(w)L_1^{\pm}(z) R(z/w),\nonumber\\
 R(z_+/w_-)L_1^{+}(z)L_2^{-}(w)=L_2^{-}(w)L_1^{+}(z) R(z_-/w_+).\label{a1}
\end{gather}
\item[(2)] The associative superalgebra $U\bigl(\widetilde{R}\bigr)$ over $\mathbb{C}\bigl(q^{1/2}\bigr)$ is generated by an invertible central element \smash{$q^{c/2}$} and elements $l_{ij}^{\pm}[\mp p]$, where the indices satisfy $ 1\leq i,j\leq2n+2m+1$ and {$p\in\mathbb{Z}_+$,}
 following the same relations as \eqref{a1}, and
\begin{align}
 \label{d2}
 &\widetilde{R}(z/w)\mathcal{L}_1^{\pm}(z)\mathcal{L}_2^{\pm}(w)=\mathcal{L}_2^{\pm}(w)\mathcal{L}_1^{\pm}(z) \widetilde{R}(z/w),\\
 \label{d3}
 &\widetilde{R}(z_+/w_-)\mathcal{L}_1^{+}(z)\mathcal{L}_2^{-}(w)=\mathcal{L}_2^{-}(w)\mathcal{L}_1^{+}(z) \widetilde{R}(z_-/w_+),
\end{align}
here $L_i^{\pm}(z)\in \operatorname{End} \mathbb{C}^{2m+1|2n}\otimes \operatorname{End} \mathbb{C}^{2m+1|2n}\otimes
 U(R)$ \big(resp.\ $\mathcal{L}_i^{\pm}(z)\in \operatorname{End} \mathbb{C}^{2m+1|2n}\otimes \operatorname{End} \mathbb{C}^{2m+1|2n}\allowbreak\otimes
 U\bigl(\widetilde{R}\bigr)$\big), $i=1,2$, written by
\begin{gather*}
 L_1^{\pm}(z)=\sum\limits_{i,j=1}E^i_j\otimes1\otimes l_{ij}^{\pm}(z),\qquad \biggl(\text{resp}.\quad \mathcal{L}_1^{\pm}(z)=\sum\limits_{i,j=1}E^i_j\otimes1\otimes l_{ij}^{\pm}(z)\biggr), \\
 L_2^{\pm}(z)=\sum\limits_{i,j=1}1\otimes E^i_j\otimes l_{ij}^{\pm}(z), \qquad \biggl(\text{resp}.\quad \mathcal{L}_2^{\pm}(z)=\sum\limits_{i,j=1}1\otimes E^i_j\otimes l_{ij}^{\pm}(z)\biggr),
\end{gather*}
with
\begin{equation*}
 l_{ij}^{\pm}(z)=\sum\limits_{p=0}l_{ij}^{\pm}[\mp p]z^{\pm p}.
\end{equation*}
\end{enumerate}
\end{Definition}

Note that $\widetilde{R}(z)$ possesses two properties:
\begin{itemize}\itemsep=0pt
 \item[(1)] $P_{12}\widetilde{R}_{12}(z)P_{12}=\widetilde{R}_{21}(z)$,
 \item[(2)] $\widetilde{R}_{12}\left(\frac{z}{w}\right)\times
 \widetilde{R}_{21}\left(\frac{w}{z}\right)=1$.
\end{itemize}
Based on the aforementioned properties of $\widetilde{R}(z)$, we obtain
\[
 \widetilde{R}(z_-/w_+)\mathcal{L}_1^{-}(z)\mathcal{L}_2^{+}(w)=\mathcal{L}_2^{+}(w)\mathcal{L}_1^{-}(z)\widetilde{R}(z_+/w_-).
\]
\begin{Remark}\label{RM2}\quad
\begin{itemize}\itemsep=0pt
 \item[(1)] When $n=0$, we consider the $R$-matrix algebra associated with the quantum affine algebra $U_q\bigl(\widehat{\mathfrak{o}_{2m+1}}\bigr)$.
 In this case, the $R$-matrix $R_{q^{-1}}(z)$ $\bigl(q\rightarrow q^{-1}\bigr)$ coincides with the $R$-matrix defined in reference \cite{MJB2}.
 \item[(2)] The defining relations satisfied by the series $l_{ij}^{\pm}(z)$, $1\leq i,j \leq m+n$ coincide with those~for the quantum affine superalgebra \smash{$U_q\bigl(\widehat{\mathfrak{gl}(n|m)}\bigr)$} in \cite{HBKD} and also in \cite{YZC}.
\end{itemize}
\end{Remark}

Let $U^{\pm}\bigl(\widetilde{R}\bigr)$ be the subalgebras of $U\bigl(\widetilde{R}\bigr)$ generated by the coefficients of all the series $l_{ij}^{\pm}(z)$.
\begin{Proposition}\label{Apr2}
In superalgebras $U(R)$ and $U^{\pm}\bigl(\widetilde{R}\bigr)$, there exist elements $c^{\pm}(z)\in U(R)$ and $\widetilde{c}^{\pm}(z)\in U^{\pm}\bigl(\widetilde{R}\bigr)$ such as
\begin{gather}
 DL^{\pm}(z\zeta)^tD^{-1}L^{\pm}(z)=L^{\pm}(z)DL^{\pm}(z\zeta)^tD^{-1}=c^{\pm}(z),\nonumber\\
 D\mathcal{L}^{\pm}(z\zeta)^tD^{-1}\mathcal{L}^{\pm}(z)=\mathcal{L}^{\pm}(z)D\mathcal{L}^{\pm}(z\zeta)^tD^{-1}=\widetilde{c}^{\pm}(z), \label{d9}
\end{gather}
and all coefficients of the series $c^{\pm}(z)$ and $\widetilde{c}^{\pm}(z)$ belong to the center in $U(R)$ and $U^{\pm}\bigl(\widetilde{R}\bigr)$, respectively.
\end{Proposition}

\begin{proof}{Considering the $R$-matrix $\widetilde{R}(z)$ in \eqref{R1},}
multiplying both sides of the defining relation~\eqref{d2} by $z/w-\zeta$ and setting $z/w=\zeta$, we obtain
\begin{gather*}
Q\mathcal{L}_1^{\pm}(z\zeta)\mathcal{L}_2^{\pm}(z)=\mathcal{L}_2^{\pm}(z)\mathcal{L}_1^{\pm}(z\zeta)Q.
\end{gather*}
It is noteworthy that $P^{t_1}\mathcal{L}_1^{\pm}(z\zeta)=P^{t_1}\mathcal{L}_2^{\pm}(z\zeta)^t$ and $\mathcal{L}_1^{\pm}(z\zeta)P^{t_1}=\mathcal{L}_2^{\pm}(z\zeta)^tP^{t_1}$.
Thus, it follows from \eqref{d5}
\begin{align*}
 &D_1^{-1}P^{t_1}D_1\mathcal{L}_1^{\pm}(z\zeta)\mathcal{L}_2^{\pm}(z)=\mathcal{L}_2^{\pm}(z)\mathcal{L}_1^{\pm}(z\zeta)D_1^{-1}P^{t_1}D_1,\\
 &D_1^{-1}P^{t_1}\mathcal{L}_1^{\pm}(z\zeta)D_2^{-1}\mathcal{L}_2^{\pm}(z)=\mathcal{L}_2^{\pm}(z)D_2\mathcal{L}_1^{\pm}(z\zeta)P^{t_1}D_1,\\
 &D_1^{-1}P^{t_1}\mathcal{L}_2^{\pm}(z\zeta)^tD_2^{-1}\mathcal{L}_2^{\pm}(z)=\mathcal{L}_2^{\pm}(z)D_2\mathcal{L}_2^{\pm}(z\zeta)^tP^{t_1}D_1,\\
 &P^{t_1}\mathcal{L}_2^{\pm}(z\zeta)^tD_2^{-1}\mathcal{L}_2^{\pm}(z)D_1^{-1}=D_1\mathcal{L}_2^{\pm}(z)D_2\mathcal{L}_2^{\pm}(z\zeta)^tP^{t_1},\\
 &P^{t_1}D_1^{-1}\mathcal{L}_2^{\pm}(z\zeta)^tD_2^{-1}\mathcal{L}_2^{\pm}(z)=\mathcal{L}_2^{\pm}(z)D_2\mathcal{L}_2^{\pm}(z\zeta)^tD_1P^{t_1},\\
 &P^{t_1}D_2\mathcal{L}_2^{\pm}(z\zeta)^tD_2^{-1}\mathcal{L}_2^{\pm}(z)=\mathcal{L}_2^{\pm}(z)D_2\mathcal{L}_2^{\pm}(z\zeta)^tD_2^{-1}P^{t_1}.
\end{align*}
Given that the image of the operator $P^{t_1}$ in $\operatorname{End}\bigl(\mathbb{C}^{2m+1|2n}\bigr)^{\otimes2}$ is one-dimensional, we have
\begin{gather*}
 P^{t_1}D_2\mathcal{L}_2^{\pm}(z\zeta)^tD_2^{-1}\mathcal{L}_2^{\pm}(z)=\mathcal{L}_2^{\pm}(z)D_2\mathcal{L}_2^{\pm}(z\zeta)^tD_2^{-1}P^{t_1}=\widetilde{c}^{\pm}(z)P^{t_1},
\end{gather*}
where we take the trace of the first copy of $\operatorname{End}\mathbb{C}^{2m+1|2n}$, yielding equations \eqref{d9}.

To demonstrate that $\widetilde{c}^+(z)$ is central, we only need to verify the case of $\widetilde{c}^+(z)$ and $\mathcal{L}_2^-(w)$, as the other cases follow similarly. It is worth noting that by applying the partial transposition to both sides of \eqref{d3}, we have
 \begin{gather*}
 \mathcal{L}_1^{+}(z\zeta)^t\widetilde{R}(z_+\zeta/w_-)^{t_1}\mathcal{L}_2^{-}(w)=\mathcal{L}_2^{-}(w) \widetilde{R}(z_-\zeta/w_+)^{t_1}\mathcal{L}_1^{+}(z\zeta)^t.
 \end{gather*}Thus
 \begin{align*}\widetilde{c}^+(z)\mathcal{L}_2^{-}(w)&=
 D_1\mathcal{L}_1^{+}(z\zeta)^tD_1^{-1}\mathcal{L}_1^{+}(z)\mathcal{L}_2^-(w)\\
 &= D_1\mathcal{L}_1^{+}(z\zeta)^tD_1^{-1}\widetilde{R}(z_+/w_-)^{-1}\mathcal{L}_2^-(w)\mathcal{L}_1^{+}(z)\widetilde{R}(z_-/w_+)\\
 &=D_1\mathcal{L}_1^{+}(z\zeta)^t\widetilde{R}(z_+\zeta/w_-)^{t_1}\mathcal{L}_2^{-}(w)D_1^{-1}\mathcal{L}_1^{+}(z)\widetilde{R}(z_-/w_+)\cdot\frac{1}{y(z)}\\
 &=D_1\mathcal{L}_2^{-}(w)\widetilde{R}(z_-\zeta/w_+)^{t_1}\mathcal{L}_1^{+}(z\zeta)^tD_1^{-1}\mathcal{L}_1^{+}(z)\widetilde{R}(z_-/w_+)\cdot\frac{1}{y(z)}\\
 &=\mathcal{L}_2^{-}(w)D_1\widetilde{R}(z_-\zeta/w_+)^{t_1}D_1^{-1}\widetilde{R}(z_-/w_+)\cdot\frac{1}{y(z)}c^+(z)=\mathcal{L}_2^{-}(w)\widetilde{c}^+(z),
 \end{align*}
 where we have used \eqref{d7}.
The proof for $c^{\pm}(z)$ follows a similar pattern as that for $\widetilde{c}^{\pm}(z)$. It is important to note that relation \eqref{d6} ensures that $c^{\pm}(z)$ are central within the entire superalgebra~$U(R)$.
\end{proof}

Introduce a Heisenberg algebra $\mathcal{H}_q(m+n)$ related to the superalgebras $U(R)$ and $U\bigl(\widetilde{R}\bigr)$. The Heisenberg algebra $\mathcal{H}_q(m+n)$ is generated by the elements $\beta_p$ ($p\in\mathbb{Z}\setminus\{0\}$) and the central element $q^c$, which satisfy the following relation
$ [\beta_p, \beta_s]=\delta_{p,-s}\vartheta_p$, $ p\geq1$,
where the elements $\vartheta_p$ are defined by the expansion
\[
\exp\biggl(\sum_{p=1}\vartheta_pz^p\biggr)=\frac{g(zq^{-c})}{g(zq^c)},
\]
and
\begin{gather*}
 g(zq^{c}/w) \exp\biggl(\sum_{p=1}\vartheta_p z^p\biggr) \cdot \exp\biggl(\sum_{s=1}\vartheta_{-s}w^{-s}\biggr)\\
 \qquad=g(zq^{-c}/w)\exp\biggl(\sum_{s=1}\vartheta_{-s}w^{-s}\biggr) \cdot \exp\biggl(\sum_{p=1}\vartheta_pz^p\biggr).
\end{gather*}
Therefore, we immediately have the following proposition.
\begin{Proposition}\label{Apr3}
There exists a homomorphism $U\bigl(\widetilde{R}\bigr)\mapsto\mathcal{H}_q(m+n)\otimes_{\mathbb{C}[q^c, q^{-c}]} U(R)$ defined by
\begin{gather*}
\mathcal{L}^+(z)\mapsto \exp\biggl(\sum_{p=1}\vartheta_{-p}z^{-p}\biggr) \cdot L^+(z),\qquad \mathcal{L}^-(z)\mapsto \exp\biggl(\sum_{p=1}\vartheta_{p}z^{p}\biggr) \cdot L^-(z).
\end{gather*}
\end{Proposition}

\subsection{Quasideterminants and Gauss decomposition }\label{sec4.3}
Let $A=[a_{ij}]$ be an $N\times N$ matrix, where $N=2m+2n+1$. Denote $A^{ij}$ as the matrix obtained from $A$ by deleting the $i$-th row and $j$-th column. Suppose $A^{ij}$ is invertible. The $ij$-th quasideterminant of $A$ is defined as follows:
\begin{equation*}
 |A|_{ij}=a_{ij}-r^j_i\bigl(A^{ij}\bigr)^{-1}c^i_j,
\end{equation*}
where $r^j_i$ is the row matrix obtained from the $i$-th row of $A$ by deleting $a_{ij}$, and $c^i_j$ is the column matrix obtained from the $j$-th column of $A$ by deleting $a_{ij}$ (please refer to \cite{IMG,DKBM}). For example, the quasideterminants of $A=[a_{ij}]_{2\times2}$ are
\begin{gather*}
 |A|_{11}= a_{11}-a_{12}{a_{22}}^{-1}a_{21}, \qquad |A|_{12}= a_{12}-a_{11}{a_{21}}^{-1}a_{21},\\
 |A|_{21}= a_{21}-a_{22}{a_{12}}^{-1}a_{11}, \qquad |A|_{22}= a_{22}-a_{21}{a_{11}}^{-1}a_{12}.
\end{gather*}
Furthermore, we denote the quasideterminant $|A|_{ij}$ that boxes the entry $a_{ij}$ as
\begin{align*}
|A|_{ij}=
\left|
 \begin{matrix}
 a_{11} & \cdots & a_{1j} & \cdots & a_{1N} \\
 & \vdots & & \vdots & \\
 a_{i1} & \cdots & \boxed{a_{ij}} & \cdots & a_{iN} \\
 & \vdots & & \vdots & \\
 a_{N1} & \cdots & a_{Nj} & \cdots & a_{NN} \\
 \end{matrix}
\right|.
\end{align*}

Now, we introduce the Gaussian generators in the super $R$-matrix algebras. In $U\bigl(\widetilde{R}\bigr)$, set the universal quasideterminant formulas as below
\begin{gather*}
\mathfrak{h}_i^{\pm}(z)=\left|
 \begin{matrix}
 l_{11}^{\pm}(z) & \dots & l_{1i-1}^{\pm}(z) & l_{1i}^{\pm}(z) \\
 \vdots & \ddots & \vdots & \vdots \\
 l_{i-11}^{\pm}(z) & \dots & l_{i-1i-1}^{\pm}(z) & l_{i-1i}^{\pm}(z) \\
 l_{i1}^{\pm}(z) & \dots &l_{ii-1}^{\pm}(z) & \boxed{l_{ii}^{\pm}(z)}
 \end{matrix}\right|, \\
 i=1,\dots,n,\dots,n+2m+1,\dots,2m+2n+1,
\end{gather*}
where
\begin{gather*}\mathfrak{e}_{ij}^{\pm}(z)=\mathfrak{h}_i^{\pm}(z)^{-1}\left|
 \begin{matrix}
 l_{11}^{\pm}(z) & \dots & l_{1i-1}^{\pm}(z) & l_{1j}^{\pm}(z) \\
 \vdots & \ddots & \vdots & \vdots \\
 l_{i-11}^{\pm}(z) & \dots & l_{i-1i-1}^{\pm}(z) & l_{i-1j}^{\pm}(z) \\
 l_{i1}^{\pm}(z) & \dots &l_{ii-1}^{\pm}(z) & \boxed{l_{ij}^{\pm}(z)}
 \end{matrix}\right|,
\\
\mathfrak{f}_{ji}^{\pm}(z)=\left|
 \begin{matrix}
 l_{11}^{\pm}(z) & \dots & l_{1i-1}^{\pm}(z) & l_{1i}^{\pm}(z) \\
 \vdots & \ddots & \vdots & \vdots \\
 l_{i-11}^{\pm}(z) & \dots & l_{i-1i-1}^{\pm}(z) & l_{i-1i}^{\pm}(z) \\
 l_{j1}^{\pm}(z) & \dots &l_{ji-1}^{\pm}(z) & \boxed{l_{ji}^{\pm}(z)}
 \end{matrix}\right|\mathfrak{h}_i^{\pm}(z)^{-1}.
\end{gather*}

Denote the matrices as follows:
\begin{align*}\widetilde{F}^{\pm}(z)=\left(
 \begin{matrix}
 1 & 0 & \cdots & 0 \\
 \mathfrak{f}_{21}^{\pm}(z) & 1 & \cdots & 0 \\
 \vdots & \vdots & \ddots & \vdots \\
 \mathfrak{f}_{N1}^{\pm}(z) & \mathfrak{f}_{N2}^{\pm}(z) & \cdots & 1
 \end{matrix}\right),\qquad
 \widetilde{E}^{\pm}(z)=\left(\begin{matrix}
 1 & \mathfrak{e}_{12}^{\pm}(z) & \cdots & \mathfrak{e}_{1N}^{\pm}(z) \\
 0 & 1 & \cdots & \mathfrak{e}_{2N}^{\pm}(z) \\
 \vdots & \vdots & \ddots & \vdots \\
 0& 0 & \cdots & 1
 \end{matrix}\right),
\end{align*}
and \smash{$\widetilde{H}^{\pm}(z)=\operatorname{diag}\bigl(\mathfrak{h}_1^{\pm}(z),\dots,\mathfrak{h}_{N}^{\pm}(z)\bigr)$}.
Then
$\mathcal{L}^{\pm}(z)$ has unique Gauss decomposition
$\mathcal{L}^{\pm}(z)=\widetilde{F}^{\pm}(z)\widetilde{H}^{\pm}(z)\widetilde{E}^{\pm}(z)$.
Since $\widetilde{H}^{\pm}(z)$, $\widetilde{F}^{\pm}(z)$ and $\widetilde{E}^{\pm}(z)$ are invertible, we can
write the inversions of $\mathcal{L}^{\pm}(z)$ as follows:
\begin{align*}
 \mathcal{L}^{\pm}(z)^{-1}={}&\left(
 \begin{matrix}
 1 & -\mathfrak{e}_{12}^{\pm}(z) & &* \\
 \ddots & \ddots& &\\
 & & & -\mathfrak{e}_{N-1,N}^{\pm}(z)\\
 0 & & &1
 \end{matrix}
 \right)\left(
 \begin{matrix}
 \mathfrak{h}_1^{\pm}(z)^{-1} & & 0 \\
 & \ddots & \\
 0 & & h_N^{\pm}(z)^{-1}
 \end{matrix}
 \right) \nonumber\\
 & \times\left(
 \begin{matrix}
 1 & & & 0 \\
 -\mathfrak{f}_{21}^{\pm}(z) & & & \\
 \ddots & \ddots & & \\
 *& & -\mathfrak{f}_{N-1,N}^{\pm}(z) & 1
 \end{matrix}
 \right).
\end{align*}
In $U(R)$, we denote the entries $h_i^{\pm}(z)$, \smash{$e_{ij}^{\pm}(z)$}, \smash{$f_{ji}^{\pm}(z)$} to express the respective triangular matrices $F^{\pm}(z)$, $E^{\pm}(z)$, and the diagonal matrices $H^{\pm}(z)$, which are used in the same terms as the formal series $l_{ij}^{\pm}(z)$. Thus, we have
$
 L^{\pm}(z)=F^{\pm}(z)H^{\pm}(z)E^{\pm}(z)$.

\begin{Proposition}\label{Apr4}
From Proposition {\rm\ref{Apr3}}, we can express the homomorphism $U\bigl(\widetilde{R}\bigr)\mapsto\mathcal{H}_q(m+n)\otimes_{\mathbb{C}[q^c, q^{-c}]} U(R)$ as follows:
\begin{align*}
 & \mathfrak{e}_{ij}^{\pm}(z)\mapsto e_{ij}^{\pm}(z),\qquad \mathfrak{f}_{ji}^{\pm}(z)\mapsto f_{ji}^{\pm}(z),\qquad
 \mathfrak{h}_i^{\pm}(z)\mapsto \exp\biggl(\sum_{p=1}\vartheta_{\mp p}z^{\mp p}\biggr) \cdot h_i^{\pm}(z).
 \end{align*}
\end{Proposition}
\begin{proof}
It follows from the formulas of Gaussian generators.
\end{proof}

\subsection{The homomorphism theorem }
Consider the superalgebra $U\bigl(\widetilde{R}^{(m|n-l)}\bigr)$ for $l<n$ (resp.\ $U\bigl(\widetilde{R}^{(m+n-l|0)}\bigr)$ for $l\geq n$) corresponding to the $R$-matrices \smash{$\widetilde{R}^{(m|n-l)}(z)$} (resp.\ \smash{$\widetilde{R}^{(m+n-l|0)}(z)$}), which possess generators \smash{$l_{ij}^{\pm}[\mp p]$} for $l<i,j <\overline{l}$ and $p=0,1,2,\dots$.
It is worth noting that when $l=0$, $\widetilde{R}^{(m|n)}(z)=\widetilde{R}(z)$. In this section, we will describe the connection of the superalgebras between $U\bigl(\widetilde{R}\bigr)$ and $U\bigl(\widetilde{R}^{(m|n-l)}\bigr)$ for $l<n$ (resp.\ \smash{$U\bigl(\widetilde{R}^{(m+n-l|0)}\bigr)$} for $l\geq n$).

\begin{Theorem}\label{th1}
The mapping
\begin{align*}
l_{ij}^{\pm}(z)\mapsto
\left|
 \begin{matrix}
 l_{11}^{\pm}(z) & l_{1j}^{\pm}(z) \\
 l_{i1}^{\pm}(z) &\boxed{l_{ij}^{\pm}(z)} \\
 \end{matrix}
 \right|, \qquad i,j \neq1,\overline{1},
\end{align*}
defines a homomorphism $U\bigl(\widetilde{R}^{(m|n-1)}\bigr)\rightarrow U\bigl(\widetilde{R}^{(m|n)}\bigr)$.
\end{Theorem}

To proving Theorem \ref{th1}, we first establish some preliminary results. For $v\in \operatorname{End}\bigl(\mathbb{C}^{2m+1|2n}\bigr)^{\otimes t}$ and fixed $a_i,b_i\in\mathbb{Z}_+$, $i=1,\dots,t$, define the operators on \smash{$\bigl(\mathbb{C}^{2m+1|2n}\bigr)^{\otimes t}$} as follows:
\begin{align*}
 &|b_1,\dots, b_t\rangle\cdot v\doteq \sum\limits_{x_i} E_{x_1}^{b_1}\otimes \dots\otimes E_{x_t}^{b_t}\cdot v ,\qquad \langle a_1,\dots, a_t| \cdot v\doteq \sum\limits_{x_i}v\cdot E^{x_1}_{a_1}\otimes \dots\otimes E^{x_t}_{a_t},
\end{align*}
where $x_i\in\mathbb{Z}$ are indeterminate. So that $E_{a_1}^{b_1}\otimes \dots\otimes E_{a_t}^{b_t}=|b_1,\dots, b_t\rangle \langle a_1,\dots, a_t|$.

 Now, we treat elements of the tensor product algebra
\smash{$\operatorname{End}\bigl(\mathbb{C}^{2m+1|2n}\bigr)^{\otimes t} \otimes U\bigl(\widetilde{R}\bigr)\bigl[\bigl[z,z^{-1}\bigr]\bigr]$} as operators on the space \smash{$\bigl(\mathbb{C}^{2m+1|2n}\bigr)^{\otimes t}$}, with coefficients in \smash{$U\bigl(\widetilde{R}^{(m|n)}\bigr)$}.
Therefore, for an element of the form
\begin{equation*}
 X = \sum\limits_{a_i,b_i} E^{a_1}_{b_1} \otimes \cdots \otimes E^{a_t}_{b_t} \otimes X^{a_1 \cdots a_t}_{b_1 \cdots b_t},
\end{equation*}
we adopt the following standard notation
\smash{$
 X^{a_1 \cdots a_t}_{b_1 \cdots b_t} = \langle a_1, \ldots, a_t | X | b_1, \ldots, b_t \rangle$}.
The operators $\langle a_1, \ldots, a_t | X$ and $Y | b_1, \ldots, b_t \rangle$ are defined in the usual manner.
In fact, for fixed $b_1, \ldots, b_t \in \mathbb{Z}_+$ and \smash{$X \in \operatorname{End}\bigl(\mathbb{C}^{2m+1|2n}\bigr)^{\otimes t} \otimes U\bigl(\widetilde{R}\bigr)\bigl[\bigl[z,z^{-1}\bigr]\bigr]$}, we have
\begin{equation*}
 X | b_1, \ldots, b_t \rangle = \sum\limits_{x_i} E_{x_1}^{b_1} \otimes \cdots \otimes E_{x_t}^{b_t} \otimes X_{x_1, \ldots, x_t}^{b_1, \ldots, b_t} | x_1, \ldots, x_t \rangle,
\end{equation*}
where $x_i \in \mathbb{Z}$ are indeterminate.

Introduce the quasideterminant
\begin{align*}s^{\pm}_{ij}(z)=
 \left|
 \begin{matrix}
 l_{11}^{\pm}(z) & l_{1j}^{\pm}(z) \\
 l_{i1}^{\pm}(z) & \boxed{l_{ij}^{\pm}(z)} \\
 \end{matrix}
 \right|= l_{ij}^{\pm}(z)- l_{i1}^{\pm}(z) l_{11}^{\pm}(z)^{-1} l_{1j}^{\pm}(z),
\end{align*}and let the quantum minors series \smash{$l^{\pm a_1a_2}_{b_1b_2}(z)$} with coefficients in \smash{$U\bigl(\widetilde{R}^{(m|n)}\bigr)$} as
\[
 l^{\pm a_1a_2}_{b_1b_2}(z)=\big\langle a_1, a_2|\widehat{R}\bigl(q^{-2}\bigr)\mathcal{L}_1^{\pm}(z)\mathcal{L}_2^{\pm}\bigl(zq^{2}\bigr)|b_1, b_2\big\rangle,
\]
where $a_i, b_i\in\{1,\dots,2m+2n+1\}$, and set
\smash{$\widehat{R}(z)=\frac{q-q^{-1}z}{z-1}\widetilde{R}(z)$}.

\begin{Lemma}\label{le1}
For any $1<i,j <\overline{1}$, we have
\begin{equation}\label{12}
 s^{\pm}_{ij}(z)=- l_{11}^{\pm}\bigl(zq^{-2}\bigr)^{-1}l^{\pm 1i}_{1j}\bigl(zq^{-2}\bigr).
\end{equation}Moreover,
\begin{gather}\label{13}
 \bigl[l_{11}^{\pm}(z),s^{\pm}_{ij}(z)\bigr]=0, \\
 \label{14}
 \frac{z_{\pm}-w_{\mp}}{q^{-1}z_{\pm}-qw_{\mp}}l_{11}^{\pm}(z)s^{\mp}_{ij}(w)=\frac{z_{\mp}
 -w_{\pm}}{q^{-1}z_{\mp}-qw_{\pm}}s^{\mp}_{ij}(w)l_{11}^{\pm}(z).
\end{gather}
\end{Lemma}
\begin{proof}
By the definition of quantum minors, we have
\begin{equation*}
 l^{\pm 1i}_{1j}(z)=\big\langle 1, i|\widehat{R}\bigl(q^{-2}\bigr)\mathcal{L}_1^{\pm}(z)\mathcal{L}_2^{\pm}\bigl(zq^2\bigr)|1, j\big\rangle=-l_{11}^{\pm}(z)l_{ij}^{\pm}\bigl(zq^{2}\bigr)-(-1)^{[i]}q^{-1}l_{i1}^{\pm}(z)l_{1j}^{\pm}\bigl(zq^{2}\bigr).
\end{equation*}
Then from relation \eqref{d2}, we find that
\begin{gather*}
\big\langle 1, i|\widetilde{R}(z/w)\mathcal{L}_1^{\pm}(z)\mathcal{L}_2^{\pm}(w)|1, 1\big\rangle=\big\langle 1, i|\mathcal{L}_2^{\pm}(w)\mathcal{L}_1^{\pm}(z)\widetilde{R}(z/w)|1, 1\big\rangle,
\end{gather*}
a direct calculation gives
\begin{gather*}
 -(z/w-1)l_{11}^{\pm}(z)l_{i1}^{\pm}(w)+\bigl(q-q^{-1}\bigr)z/wl_{i1}^{\pm}(z)l_{11}^{\pm}(w)=-\bigl(qz/w-q^{-1}\bigr)l_{i1}^{\pm}(w)l_{11}^{\pm}(z).
\end{gather*}
Let $z/w=q^{-2}$, we have
\begin{equation*}
 l_{11}^{\pm}\bigl(zq^{-2}\bigr)l_{i1}^{\pm}(z)=-(-1)^{[i]}q^{-1}l_{i1}^{\pm}\bigl(zq^{-2}\bigr)l_{11}^{\pm}(z),
\end{equation*}
and hence
\begin{align*}
 l^{\pm 1i}_{1j}\bigl(zq^{-2}\bigr)&=-l_{11}^{\pm}\bigl(zq^{-2}\bigr)l_{ij}^{\pm}(z)-(-1)^{[i]}q^{-1}l_{i1}^{\pm}\bigl(zq^{-2}\bigr)l_{1j}^{\pm}(z) \\
 &=-l_{11}^{\pm}\bigl(zq^{-2}\bigr)l_{ij}^{\pm}(z)-(-1)^{[i]}q^{-1}l_{i1}^{\pm}\bigl(zq^{-2}\bigr)l_{11}^{\pm}(z)\bigl(l_{11}^{\pm}(z)\bigr)^{-1}l_{1j}^{\pm}(z)\\
 &=-l_{11}^{\pm}\bigl(zq^{-2}\bigr)l_{ij}^{\pm}(z)+l_{11}^{\pm}\bigl(zq^{-2}\bigr)l_{i1}^{\pm}(z)\bigl(l_{11}^{\pm}(z)\bigr)^{-1}l_{1j}^{\pm}(z)=-l_{11}^{\pm}\bigl(zq^{-2}\bigr)s^{\pm}_{ij}(z),
\end{align*}
that is, $s^{\pm}_{ij}(z)=- l_{11}^{\pm}\bigl(zq^{-2}\bigr)^{-1}l^{\pm 1i}_{1j}\bigl(zq^{-2}\bigr)$,
which implies \eqref{12}.

For the remaining two equations, we only verify \eqref{14} since \eqref{13} is similar.
First, we have the following equation:
\begin{align*}
 &\big\langle 1, 1, i|\widetilde{R}_{01}(z_{\pm}/w_{\mp})\widetilde{R}_{02}(z_{\pm}q^{-2}/w_{\mp})\mathcal{L}_0^{\pm}(z)\widehat{R}_{12}\bigl(q^{-2}\bigr)\mathcal{L}_1^{\mp}(w)L_2^{\mp}\bigl(wq^2\bigr)|1, 1, j\big\rangle\\
 &\qquad=\big\langle 1, 1, i|\widehat{R}_{12}\bigl(q^{-2}\bigr)\mathcal{L}_1^{\mp}(w)\mathcal{L}_2^{\mp}\bigl(wq^2\bigr)\mathcal{L}_0^{\pm}(z)\widetilde{R}_{02}\bigl(z_{\mp}q^{-2}/w_{\pm}\bigr)\widetilde{R}_{01}(z_{\pm}/w_{\mp})|1, 1, j\big\rangle,
\end{align*}
which is derived by the Yang--Baxter equation and relation \eqref{d2}. After a easy calculation, we get
\begin{equation*}
 \frac{q^{-2}z_{\pm}-w_{\mp}}{q^{-3}z_{\pm}-qw_{\mp}}l_{11}^{\pm}(z)l^{\mp 1i}_{1j}(w)=\frac{q^{-2}z_{\mp}-w_{\pm}}{q^{-3}z_{\mp}-qw_{\pm}}l^{\mp 1i}_{1j}(w)l_{11}^{\pm}(z).
\end{equation*}{Therefore,
\begin{equation*}
 \frac{q^{-2}z_{\pm}-q^{-2}w_{\mp}}{q^{-3}z_{\pm}-q^{-1}w_{\mp}}l_{11}^{\pm}(z)l_{11}^{\mp}\bigl(q^{-2}w\bigr)s^{\mp }_{ij}(w)=\frac{q^{-2}z_{\mp}-q^{-2}w_{\pm}}{q^{-3}z_{\mp}-q^{-1}w_{\pm}}l_{11}^{\mp}\bigl(q^{-2}w\bigr)s^{\mp }_{ij}(w)l_{11}^{\pm}(z).
\end{equation*}This implies the equation \eqref{14}. }
\end{proof}

Consider the tensor product algebra \smash{$\operatorname{End}\bigl(\mathbb{C}^{2m+1|2n}\bigr)^{\otimes 4}\otimes U\bigl(\widetilde{R}^{(m|n)}\bigr)$} for $1 < i, j < \overline{1}$.
The following lemma can be derived through direct calculation, similar to \cite[equations (3.25)--(3.31)]{JMA}.

\begin{Lemma}\label{le2}
For $1<i,j <\overline{1}$, we have
\begin{gather*}
 \widehat{R}_{12}^{(m|n)}\bigl(q^{-2}\bigr)\widehat{R}_{34}^{(m|n)}\bigl(q^{-2 }\bigr)\widetilde{R}_{14}^{(m|n)}\bigl(zq^{-2}\bigr)\widetilde{R}_{24}^{(m|n)}(z)\widetilde{R}_{13}^{(m|n)}(z)\widetilde{R}_{23}^{(m|n)}(w)|1, i, 1, j\rangle \nonumber\\
 \qquad=C(z)\widehat{R}_{12}^{(m|n)}\bigl(q^{-2}\bigr)\widehat{R}_{34}^{(m|n)}\bigl(q^{-2}\bigr)\widetilde{R}_{24}^{(m|n-1)}(z)|1,i, 1, j\rangle,
 \end{gather*}
 and
 \begin{gather*}
 \langle1, i, 1, j|\widetilde{R}_{23}^{(m|n)}(w)\widetilde{R}_{13}^{(m|n)}(z)\widetilde{R}_{24}^{(m|n)}(z)\widetilde{R}_{14}^{(m|n)}\bigl(zq^{-2}\bigr)\widehat{R}_{12}^{(m|n)}
 \bigl(q^{-2}\bigr)\widehat{R}_{34}^{(m|n)}\bigl(q^{-2}\bigr)\nonumber\\
 \qquad=C(z)\langle1, i, 1, j|\widetilde{R}_{24}^{(m|n-1)}(z)\widehat{R}_{12}^{(m|n)}\bigl(q^{-2}\bigr)\widehat{R}_{34}^{(m|n)}\bigl(q^{-2}\bigr),
\end{gather*}
where
\[C(z,w)
=\frac{(w-1)(z-q^{2})-(q-q^{-1})^2z}{(wq^{-2}-1)(z-q^4)}.
\]
\end{Lemma}

\begin{Remark}Note that, the $R$-matrix $\widetilde{R}(z)$ has a singularity at $z=q^{2}$. However, in Lemma~\ref{le2} and the proof of Lemma~\ref{le1}, we have not considered the singularity of the $R$-matrix. Indeed, If we take into account the singularity of the $R$-matrix, similar to the non-super case discussed in \cite{JMA} and \cite{JMAB}, as well as in the context of super orthosymplectic Yangians \cite{AIM}, we have to redefine the following relation
$l^{\pm a_1a_2}_{b_1b_2}(z)=\langle a_1, a_2|\mathcal{L}_2^{\pm}\bigl(zq^{-2}\bigr)\mathcal{L}_1^{\pm}(z)\widehat{R}\bigl(q^{2}\bigr)|b_1, b_2\rangle$,
where $a_i, b_i\in\{1,\dots,2m+2n+1\}$. A similar calculation leads to the result
\smash{$s^{\pm}_{ij}(z)=l^{\pm 1i}_{1j}\bigl(zq^{2}\bigr)l_{11}^{\pm}\bigl(zq^{2}\bigr)^{-1}$},
and the relations of \eqref{13}--\eqref{14}. Moreover, in Lemma~\ref{le2}, we find that
\begin{gather*}
 C(z,w)
=\frac{(wq-q^{-1})(zq-q^{-1})}{(w-1)(z-1)},
\end{gather*}when we replace $q^{-2}$ with $q^2$.

This means that, whether it is $\widehat{R}\bigl(q^{2}\bigr)$ or $\widehat{R}\bigl(q^{-2}\bigr)$, both can be regarded as a constant coefficient of \smash{$\operatorname{End}\bigl(\mathbb{C}^{2m+1|2n}\bigr)\otimes U\bigl(\widetilde{R}^{(m|n)}\bigr)$} that satisfies the Yang--Baxter equation and the defining relation~\eqref{d2} over~$\mathbb{C}\bigl(q^{1/2}\bigr)$.
\end{Remark}

Under the aforementioned constructions, we now proceed to prove Theorem \ref{th1}.

\begin{proof}[Proof of Theorem \ref{th1}]
By utilizing the Yang--Baxter equation and the defining relations in~\smash{$U\bigl(\widetilde{R}\bigr)$}, we derive the following equality:
\begin{gather*}
 \widetilde{R}_{23}^{(m|n)}\left(\frac{zq^2}{w}\right)\widetilde{R}_{13}^{(m|n)}\left(\frac{z}{w}\right)\widetilde{R}_{24}^{(m|n)}
 \left(\frac{z}{w}\right)\widetilde{R}_{14}^{(m|n)}\left(\frac{z}{wq^2}\right)\widehat{R}_{12}^{(m|n)}\bigl(q^{-2}\bigr) \\
 \qquad\times \mathcal{L}^{\pm}_1(z)\mathcal{L}^{\pm}_2\bigl(zq^2\bigr)\widehat{R}_{34}^{(m|n)}\bigl(q^{-2}\bigr)\mathcal{L}^{\pm}_3(w)\mathcal{L}^{\pm}_4\bigl(wq^{2}\bigr) \\
\qquad\phantom{\times}{}= \widehat{R}_{34}^{(m|n)}\bigl(q^{-2}\bigr)\mathcal{L}^{\pm}_3(w)\mathcal{L}^{\pm}_4\bigl(wq^{2}\bigr)\widehat{R}_{12}^{(m|n)}\bigl(q^{-2}\bigr) \\
\qquad\phantom{\times=}{}\times \mathcal{L}^{\pm}_1(z)\mathcal{L}^{\pm}_2\bigl(zq^2\bigr)\widetilde{R}_{14}^{(m|n)}\left(\frac{z}{wq^2}\right)\widetilde{R}_{24}^{(m|n)}
 \left(\frac{z}{w}\right)
 \widetilde{R}_{13}^{(m|n)}\left(\frac{z}{w}\right)\widetilde{R}_{23}^{(m|n)}\left(\frac{zq^2}{w}\right).
\end{gather*}
Assuming that $1 < i, j, k, l < \overline{1}$, by Lemma \ref{le2}, we obtain
\begin{gather}
 \langle 1, k, 1, l |\widetilde{R}_{24}^{(m|n-1)}\left(\frac{z}{w}\right)\widehat{R}_{12}^{(m|n)}\bigl(q^{-2}\bigr)
 \mathcal{L}^{\pm}_1(z)\mathcal{L}^{\pm}_2\bigl(zq^2\bigr)\widehat{R}_{34}^{(m|n)}\bigl(q^{-2}\bigr)\mathcal{L}^{\pm}_3(w)\mathcal{L}^{\pm}_4\bigl(wq^{2}\bigr)| 1, i, 1, j\rangle \nonumber\\
\qquad = \langle 1, k, 1, l |\widehat{R}_{34}^{(m|n)}\bigl(q^{-2}\bigr)\mathcal{L}^{\pm}_3(w)\mathcal{L}^{\pm}_4\bigl(wq^{2}\bigr)\widehat{R}_{12}^{(m|n)}\bigl(q^{-2}\bigr)\nonumber\\
\phantom{\qquad =}{}\times
 \mathcal{L}^{\pm}_1(z)\mathcal{L}^{\pm}_2\bigl(zq^2\bigr)\widetilde{R}_{24}^{(m|n-1)}\left(\frac{z}{w}\right)| 1, i, 1, j\rangle.\label{qq2}
\end{gather}
Set
\begin{gather*}
 \mathfrak{L}^{\pm}(z)=\sum\limits_{i,j\neq1,\bar{1}}E^i_j\otimes l^{\pm1i}_{1j}(z)\in \operatorname{End}\mathbb{C}^{2m+1|2n}\otimes U\bigl(\widetilde{R}\bigr),
\end{gather*}
then \eqref{qq2} is equivalent to
\begin{gather}\label{22}
 \widetilde{R}_{24}^{(m|n-1)}\left(\frac{z}{w}\right) \mathfrak{L}_2^{\pm}(z)\mathfrak{L}_4^{\pm}(w)=\mathfrak{L}_4^{\pm}(w)\mathfrak{L}_2^{\pm}(z)\widetilde{R}_{24}^{(m|n-1)}\left(\frac{z}{w}\right).
\end{gather}
Let
\smash{$ S^{\pm}(z)=\sum_{i,j\neq1,\bar{1}}E^i_j\otimes s_{ij}^{\pm}(z)$}.
By Lemma \ref{le1},
$S^{\pm}(z)=-l_{11}^{\pm}\bigl(zq^{-2}\bigr)^{-1}\mathfrak{L}^{\pm}\bigl(zq^{-2}\bigr)$.
Hence, \eqref{22} implies that
\begin{gather}\label{23}
 \widetilde{R}^{(m|n-1)}\left(\frac{z}{w}\right) S_1^{\pm}(z)S_2^{\pm}(w)=S_2^{\pm}(w)S_1^{\pm}(z) \widetilde{R}^{(m|n-1)}\left(\frac{z}{w}\right).
\end{gather}
Similarly, we also have
\begin{gather}\label{24}
 \widetilde{R}^{(m|n-1)}\left(\frac{zq^{\pm c}}{w}\right) S_1^{\pm}(z)S_2^{\mp}(w)=S_2^{\mp}(w)S_1^{\pm}(z) \widetilde{R}^{(m|n-1)}\left(\frac{zq^{\mp c}}{w}\right).
\end{gather}
Now, \eqref{23} and \eqref{24} provide the proof of Theorem \ref{th1}.
\end{proof}

We highlight certain consequences of Theorem \ref{th1}, which can be verified using similar methods as in the non-super case (see \cite[Theorem 3.9]{JMA}, and \cite[Theorem 3.7]{JMAB}). Similar results also apply to Yangians under the orthosymplectic superalgebras (see \cite[Corollary 3.2]{AIM}). Moreover, these consequences follow from the Sylvester theorem for quasideterminants (see \cite{IMG,DKBM}).

\begin{Theorem}\label{th2}
The mapping
\begin{align}\label{25}
\psi_l\colon\ l_{ij}^{\pm}(z)\mapsto
\left|
 \begin{matrix}
 l_{11}^{\pm}(z) & \ldots & l_{1l}^{\pm}(z) & l_{1j}^{\pm}(z) \\
 \ldots & \ldots & \ldots & \ldots \\
 l_{l1}^{\pm}(z) & \ldots & l_{ll}^{\pm}(z) & l_{lj}^{\pm}(z) \\
 l_{i1}^{\pm}(z) & \ldots & l_{il}^{\pm}(z) & \boxed{l_{ij}^{\pm}(z)} \\
 \end{matrix}
\right|,
\end{align}defines a homomorphism
\begin{equation*}
 U\bigl(\widetilde{R}^{(m|n-l)}\bigr)\rightarrow U\bigl(\widetilde{R}^{(m|n)}\bigr)\qquad \text{for} \quad l<n,
\end{equation*}and another homomorphism
\begin{equation*}
 U\bigl(\widetilde{R}^{(m+n-l|0)}\bigr)\rightarrow U\bigl(\widetilde{R}^{(m|n)}\bigr)\qquad \text{for} \quad l\geq n,
\end{equation*}
where the generators $l_{ij}^{\pm}(z)$ of the superalgebras \smash{$U\bigl(\widetilde{R}^{(m|n-l)}\bigr)$} or \smash{$U\bigl(\widetilde{R}^{(m+n-l|0)}\bigr)$} are indexed by \smash{$l+1\leq i,j \leq\overline{(l+1)}$}.
\end{Theorem}

\begin{Remark}\label{RM3}
In the case of $n=0$, there exists a homomorphism theorem related to the non-super $R$-matrix algebra for type B, as documented in \cite[Theorem 3.7]{JMAB}. Hence, our result is encompassed therein.
\end{Remark}

The mapping \eqref{25} possesses the following consistency property, see \cite{IGA}. Denote $\psi^{2m+2n+1}_l$ as the map $\psi_l$ in Theorem \ref{th2}, where we establish the equality
\begin{equation*}
 \psi^{2m+2n+1}_k\circ\psi^{2m+2n+1-2k}_l=\psi^{2m+2n+1}_{k+l}.
\end{equation*}

\begin{Corollary}\label{co1}
Under the assumptions of Theorem {\rm\ref{th2}}, the following relations hold:
\begin{gather*}
 \bigl[l_{ab}^{\pm}(z), \psi_l\bigl( l_{ij}^{\pm}(w)\bigr)\bigr] = 0, \qquad
 \frac{z_{\pm}-w_{\mp}}{q^{-1}z_{\pm}-qw_{\mp}}l_{ab}^{\pm}(z)\psi_l( l_{ij}^{\mp}(w))=\frac{z_{\mp}-w_{\pm}}{q^{-1}z_{\mp}-qw_{\pm}}\psi_l( l_{ij}^{\mp}(w))l_{ab}^{\pm}(z)
\end{gather*}
for $1\leq a,b\leq l$, and $l+1\leq i,j\leq \overline{(l+1)}$.
\end{Corollary}

\begin{proof}
The proof is exactly as \cite[Corollary 3.10]{JMA} or \cite[Corollary 3.8]{JMAB}.
\end{proof}

Assuming $l\leq m+n$, let \smash{$\widetilde{F}^{\pm (l)}(z)$}, \smash{$\widetilde{E}^{\pm (l)}(z)$}, and \smash{$\widetilde{H}^{\pm (l)}(z)$} be defined as follows:
\begin{align*}
\widetilde{F}^{\pm (l)}(z) &= \left(
 \begin{matrix}
 1 & 0 & \cdots & 0 \\
 \mathfrak{f}_{l+2,l+1}^{\pm}(z) & 1 & \cdots & 0 \\
 \vdots & \ddots & \ddots & \vdots \\
 \mathfrak{f}_{(\overline{l+1}),l+1}^{\pm}(z) & \cdots & \mathfrak{f}_{(\overline{l+1}),(\overline{l+2})}^{\pm}(z) & 1
 \end{matrix}\right),\\
 \widetilde{E}^{\pm (l)}(z) &= \left(\begin{matrix}
 1 & \mathfrak{e}_{l+1,l+2}^{\pm}(z) & \cdots & \mathfrak{e}_{l+1,(\overline{l+1})}^{\pm}(z) \\
 0 & 1 & \ddots& \vdots \\
 \vdots & \vdots & \ddots &\mathfrak{e}_{(\overline{l+2}),(\overline{l+1})}^{\pm}(z) \\
 0& 0 & \cdots & 1
 \end{matrix}\right),
\end{align*}
and \smash{$\widetilde{H}^{\pm (l)}(z) = \operatorname{diag}\bigl(\mathfrak{h}_{l+1}^{\pm}(z),\ldots,\mathfrak{h}_{(l+1)'}^{\pm}(z)\bigr)$}. Define the product of these matrices as
\begin{equation*}
 \mathcal{L}^{\pm (l)}(z)=\widetilde{F}^{\pm (l)}(z)\widetilde{H}^{\pm (l)}(z)\widetilde{E}^{\pm (l)}(z).
\end{equation*}
Note that $ \mathcal{L}^{\pm (0)}(z)=\mathcal{L}^{\pm }(z)$.

The following properties observed in \cite[Proposition 4.2]{JMA} extend to the super case in a similar manner.
\begin{Proposition}\label{pr2}
The series \smash{$ l^{\pm (l)}_{ij}(z)$} coincides with the image of the generator series $ l^{\pm }_{ij}(z)$ of~\smash{$U\bigl(\widetilde{R}^{(m|n)}\bigr)$} for \smash{$l+1\leq i,j \leq \overline{(l+1)}$} under the homomorphism
\smash{$
 l^{\pm (l)}_{ij}(z)=\psi_l\bigl(l^{\pm }_{ij}(z)\bigr)$}.
\end{Proposition}

Therefore, we immediately derive the following corollary from Proposition \ref{pr2}.
\begin{Corollary}\label{co2}In $U\bigl(\widetilde{R}^{(m|n)}\bigr)$, we have
\begin{align*}
 &\widetilde{R}^{(m|n-l)}(z/w)\mathcal{L}_1^{\pm(l)}(z)\mathcal{L}_2^{\pm(l)}(w)=\mathcal{L}_2^{\pm(l)}(w)\mathcal{L}_1^{\pm(l)}(z) \widetilde{R}^{(m|n-l)}(z/w),\\
&\widetilde{R}^{(m|n-l)}(z_+/w_-)\mathcal{L}_1^{(l)}(z)\mathcal{L}_2^{-(l)}(w)=\mathcal{L}_2^{-(l)}(w)\mathcal{L}_1^{(l)}(z) \widetilde{R}^{(m|n-l)}(z_-/w_+)
\end{align*}for $l< n$, and
\begin{align*}
 &\widetilde{R}^{(m+n-l|0)}(z/w)\mathcal{L}_1^{\pm(l)}(z)\mathcal{L}_2^{\pm(l)}(w)=\mathcal{L}_2^{\pm(l)}(w)\mathcal{L}_1^{\pm(l)}(z) \widetilde{R}^{(m+n-l|0)}(z/w),\\
&\widetilde{R}^{(m+n-l|0)}(z_+/w_-)\mathcal{L}_1^{(l)}(z)\mathcal{L}_2^{-(l)}(w)=\mathcal{L}_2^{-(l)}(w)\mathcal{L}_1^{(l)}(z) \widetilde{R}^{(m+n-l|0)}(z_-/w_+)
\end{align*}for $l\geq n$.
\end{Corollary}

\section[Drinfeld presentations in U(R) and U(R)]{Drinfeld presentations in $\boldsymbol{U(R)}$ and $\boldsymbol{U\bigl(\widetilde{R}\bigr)}$ }\label{sec5}

Based on the definition of Gaussian generators from Section \ref{sec4.3},
 we first investigate the relations among these generators by applying the Gauss decomposition in the superalgebras $U(R)$ and~$U\bigl(\widetilde{R}\bigr)$.
Furthermore, the central elements mentioned in Proposition \ref{Apr2} are explicitly expressed in terms of forms of Gaussian generators. Finally,
utilizing these established relations, it is found that in super $R$-matrix algebras $U(R)$ and $U\bigl(\widetilde{R}\bigr)$, their Drinfeld presentation arises.

\subsection{Relations of Gaussian generators}

\begin{Proposition}\label{pr3}
Suppose that \smash{$l+1\leq j, k, s\leq \overline{(l+1)}$} and $j\neq \overline{s}$. Then in \smash{$U\bigl(\widetilde{R}^{(m|n)}\bigr)$}, we have the following relations.
If $s>j$,
\begin{align}
\bigl[\mathfrak{e}_{lj}^{\pm}(z), l_{ks}^{\mp(l)}(w)\bigr]={}&(-1)^{[j]+[l]}\left\{\frac{\bigl(q-q^{-1}\bigr)w_{\pm}}{z_{\mp}-w_{\pm}}l_{kj}^{\mp(l)}(w)\mathfrak{e}_{ls}^{\pm}(z) -\frac{\bigl(q-q^{-1}\bigr)z_{\mp}}{z_{\mp}-w_{\pm}}l_{kj}^{\mp(l)}(w)\mathfrak{e}_{ls}^{\mp}(w)\right\},\nonumber\\
\bigl[\mathfrak{e}_{lj}^{\pm}(z), l_{ks}^{\pm(l)}(w)\bigr]={}&(-1)^{[j]+[l]}\nonumber\\
&\times\left\{\frac{\bigl(q-q^{-1}\bigr)w}{z-w}l_{kj}^{\pm(l)}(w)\mathfrak{e}_{ls}^{\pm}(z) -\frac{\bigl(q-q^{-1}\bigr)z}{z-w}l_{kj}^{\pm(l)}(w)\mathfrak{e}_{ls}^{\pm}(w)\right\}.\label{GS9}
\end{align}
If $s<j$,
\begin{align*}
 &\bigl[\mathfrak{e}_{lj}^{\pm}(z), l_{ks}^{\mp(l)}(w)\bigr]=(-1)^{[j]+[l]}\frac{\bigl(q-q^{-1}\bigr)z_{\mp}}{z_{\mp}-w_{\pm}}\bigl\{l_{kj}^{\mp(l)}(w)\mathfrak{e}_{ls}^{\pm}(z) -l_{kj}^{\mp(l)}(w)\mathfrak{e}_{ls}^{\mp}(w)\bigr\},\\
 &\bigl[\mathfrak{e}_{lj}^{\pm}(z), l_{ks}^{\pm(l)}(w)\bigr]=(-1)^{[j]+[l]}\frac{\bigl(q-q^{-1}\bigr)z}{z-w}\bigl\{l_{kj}^{\pm(l)}(w)\mathfrak{e}_{ls}^{\pm}(z) -l_{kj}^{\pm(l)}(w)\mathfrak{e}_{ls}^{\pm}(w)\bigr\}.
\end{align*}
If $s=j$ and $[j]=0$,
\begin{align}
 &\mathfrak{e}_{lj}^{\pm}(z)l_{kj}^{\mp(l)}(w)=\frac{q^{-1}z_{\mp}-qw_{\pm}}{z_{\mp}-w_{\pm}}l_{kj}^{\mp(l)}(w)\mathfrak{e}_{lj}^{\pm}(z) -\frac{\bigl(q-q^{-1}\bigr)z_{\mp}}{z_{\mp}-w_{\pm}}l_{kj}^{\mp(l)}(w)\mathfrak{e}_{lj}^{\mp}(w),\nonumber\\
 &\mathfrak{e}_{lj}^{\pm}(z)l_{kj}^{\pm(l)}(w)=\frac{q^{-1}z-qw}{z-w}l_{kj}^{\pm(l)}(w)\mathfrak{e}_{lj}^{\pm}(z) -\frac{\bigl(q-q^{-1}\bigr)z}{z-w}l_{kj}^{\pm(l)}(w)\mathfrak{e}_{1j}^{\pm}(w).\label{GS11}
\end{align}
If $s=j$ and $[j]=1$,
\begin{align*}
 &\mathfrak{e}_{lj}^{\pm}(z)l_{kj}^{\mp(l)}(w)=\frac{\bigl(q-q^{-1}\bigr)z_{\mp}}{z_{\mp}-w_{\pm}}\bigl\{l_{kj}^{\mp(l)}(w)\mathfrak{e}_{lj}^{\pm}(z) -l_{kj}^{\mp(l)}(w)\mathfrak{e}_{lj}^{\mp}(w)\bigr\},\\
 &\mathfrak{e}_{lj}^{\pm}(z)l_{kj}^{\pm(l)}(w)=\frac{\bigl(q-q^{-1}\bigr)z}{z-w}\bigl\{l_{kj}^{\mp(l)}(w)\mathfrak{e}_{lj}^{\pm}(z) -l_{kj}^{\pm(l)}(w)\mathfrak{e}_{lj}^{\pm}(w)\bigr\}.
\end{align*}
\end{Proposition}

\begin{proof}

For convenience, we denote $C^+\bigl(E^i_j\otimes E^k_s\bigr)$ and $C^-\bigl(E^i_j\otimes E^k_s\bigr)$ associated with the parameter $z_+/w_-$ and $z_-/w_+$, respectively, to be the coefficients of the position \smash{$E^i_j\otimes E^k_s$} in \smash{$\widetilde{R}(z)$}.

Let $l=1$, if $j\neq s$, then by the defining relations, we have
\begin{gather}
 C^{\pm}\bigl(E^1_1\otimes E^k_k\bigr)l_{1j}^{\pm}(z)l_{ks}^{\mp}(w)+ C^{\pm}\bigl(E^1_k\otimes E^k_1\bigr)l_{kj}^{\pm}(z)l_{1s}^{\mp}(w) \nonumber\\ \qquad =l_{ks}^{\mp}(w)l_{1j}^{\pm}(z)C^{\mp}\bigl(E^j_j\otimes E^s_s\bigr)+l_{kj}^{\mp}(w)l_{1s}^{\pm}(z)C^{\mp}\bigl(E^s_j\otimes E^j_s\bigr).\label{GS1}
\end{gather}
Since \smash{$l_{ks}^{\mp}(w)=l_{ks}^{\mp(l)}(w)+\mathfrak{f}_{k1}^{\mp}(w)\mathfrak{h}_{1}^{\mp}(w)\mathfrak{e}_{1s}^{\mp}(w)$}, the left-hand side of \eqref{GS1} can be written as
\begin{gather*}
 C^{\pm}\bigl(E^1_1\otimes E^k_k\bigr)l_{1j}^{\pm}(z)l_{ks}^{\mp(1)}(w)+ C^{\pm}\bigl(E^1_1\otimes E^k_k\bigr)l_{1j}^{\pm}(z)\mathfrak{f}_{k1}^{\mp}(w)\mathfrak{h}_{1}^{\mp}(w)\mathfrak{e}_{1s}^{\mp}(w)\nonumber\\
 \qquad{}+ C^{\pm}\bigl(E^1_k\otimes E^k_1\bigr)l_{kj}^{\pm}(z)l_{1s}^{\mp}(w).
\end{gather*}On the other hand,
\begin{gather*}
 C^{\pm}\bigl(E^1_1\otimes E^k_k\bigr)l_{1j}^{\pm}(z)l_{k1}^{\mp}(w)+ C^{\pm}\bigl(E^1_k\otimes E^k_1\bigr)l_{kj}^{\pm}(z)l_{11}^{\mp}(w)\nonumber\\
\qquad = l_{k1}^{\mp}(w)l_{1j}^{\pm}(z)C^{\mp}(E^j_j\otimes E^1_1)+l_{kj}^{\mp}(w)l_{11}^{\pm}(z)C^{\mp}\bigl(E^1_j\otimes E^j_1\bigr).
\end{gather*}
Thus, the left-hand side of \eqref{GS1} is equal to
\begin{gather*}
 C^{\pm}\bigl(E^1_1\otimes E^k_k\bigr)l_{1j}^{\pm}(z)l_{ks}^{\mp(1)}(w)+ C^{\pm}\bigl(E^j_j\otimes E^1_1\bigr)f_{k1}^{\mp}(w)l_{11}^{\mp}(w)l_{1j}^{\pm}(z)\mathfrak{e}_{1s}^{\mp}(w)\\
 \qquad{}+
 C^{\mp}\bigl(E^1_j\otimes E^j_1\bigr)l_{kj}^{\mp}(w)l_{11}^{\pm}(z)\mathfrak{e}_{1s}^{\mp}(w).
\end{gather*}Note that
\begin{gather*}
 C^{\pm}\bigl(E^1_1\otimes E^1_1\bigr)l_{1j}^{\pm}(z)l_{11}^{\mp}(w)= l_{11}^{\mp}(w)l_{1j}^{\pm}(z)C^{\mp}\bigl(E^j_j\otimes E^1_1\bigr)+l_{1j}^{\mp}(w)l_{11}^{\pm}(z)C^{\mp}\bigl(E^1_j\otimes E^j_1\bigr),
\end{gather*}
it can transform the left-hand side of \eqref{GS1} as
\begin{gather*}
 C^{\pm}\bigl(E^1_1\otimes E^k_k\bigr)l_{1j}^{\pm}(z)l_{ks}^{\mp(1)}(w)+ C^{\pm}\bigl(E^1_1\otimes E^1_1\bigr)f_{k1}^{\mp}(w)l_{1j}^{\pm}(z)\mathfrak{e}_{1s}^{\mp}(w)\\
 \qquad{}+
 C^{\mp}\bigl(E^1_j\otimes E^j_1\bigr)l_{kj}^{\mp(1)}(w)l_{11}^{\pm}(z)\mathfrak{e}_{1s}^{\mp}(w).
\end{gather*}
Furthermore, we find that
\begin{gather*}
 C^{\pm}\bigl(E^1_1\otimes E^1_1\bigr)l_{1j}^{\pm}(z)l_{1s}^{\mp}(w)= l_{1s}^{\mp}(w)l_{1j}^{\pm}(z)C^{\mp}\bigl(E^j_j\otimes E^s_s\bigr)+l_{1j}^{\mp}(w)l_{1s}^{\pm}(z)C^{\mp}\bigl(E^s_j\otimes E^j_s\bigr).
\end{gather*}
Bring it to the left-hand side of \eqref{GS1} and rearranging this equation, we obtain
\begin{gather}
 C^{\pm}\bigl(E^1_1\otimes E^k_k\bigr)l_{1j}^{\pm}(z)l_{ks}^{\mp(1)}(w)- C^{\mp}\bigl(E^j_j\otimes E^s_s\bigr)l_{ks}^{\mp(1)}(w)l_{1j}^{\pm}(z) \nonumber\\ \qquad=C^{\mp}\bigl(E^s_j\otimes E^j_s\bigr)l_{kj}^{\mp(1)}(w)l_{11}^{\pm}(z)\mathfrak{e}_{1s}^{\pm}(z)-C^{\mp}\bigl(E^1_j\otimes E^j_1\bigr)l_{kj}^{\mp(1)}(w)l_{11}^{\pm}(z)\mathfrak{e}_{1s}^{\mp}(w).\label{GS2}
\end{gather}
Similarly,
\begin{gather*}
 C^{\pm}\bigl(E^1_1\otimes E^k_k\bigr)l_{1j}^{\pm}(z)l_{ks}^{\pm(1)}(w)- C^{\pm}\bigl(E^j_j\otimes E^s_s\bigr)l_{ks}^{\pm(1)}(w)l_{1j}^{\pm}(z) \\ \qquad =C^{\mp}\bigl(E^s_j\otimes E^j_s\bigr)l_{kj}^{\pm(1)}(w)l_{11}^{\pm}(z)\mathfrak{e}_{1s}^{\pm}(z)-C^{\mp}\bigl(E^1_j\otimes E^j_1\bigr)l_{kj}^{\pm(1)}(w)l_{11}^{\pm}(z)\mathfrak{e}_{1s}^{\pm}(w).
\end{gather*}

If $j=s$, the same argument gives that
\begin{gather*} 
 C^{\pm}\bigl(E^1_1\otimes E^k_k\bigr)l_{1j}^{\pm}(z)l_{kj}^{\mp(1)}(w) \\
 \qquad=C^{\mp}\bigl(E^j_j\otimes E^j_j\bigr)l_{kj}^{\mp(1)}(w)l_{11}^{\pm}(z)\mathfrak{e}_{1j}^{\pm}(z)-C^{\mp}\bigl(E^1_j\otimes E^j_1\bigr)l_{kj}^{\mp(1)}(w)l_{11}^{\pm}(z)\mathfrak{e}_{1j}^{\mp}(w),
\\
 C^{\pm}\bigl(E^1_1\otimes E^k_k\bigr)l_{1j}^{\pm}(z)l_{kj}^{\pm(1)}(w) \\
\qquad =C^{\mp}\bigl(E^j_j\otimes E^j_j\bigr)l_{kj}^{\mp(1)}(w)l_{11}^{\pm}(z)\mathfrak{e}_{1j}^{\pm}(z)-C^{\mp}\bigl(E^1_j\otimes E^j_1\bigr)l_{kj}^{\mp(1)}(w)l_{11}^{\pm}(z)\mathfrak{e}_{1j}^{\pm}(w).
\end{gather*}
Moreover, Corollary \ref{co1} implies that
\begin{gather}
 \frac{z_{\pm}-w_{\mp}}{q^{-1}z_{\pm}-qw_{\mp}}l_{11}^{\pm}(z)l_{kj}^{\mp(1)}(w)=
 \frac{z_{\mp}-w_{\pm}}{q^{-1}z_{\mp}-qw_{\pm}}l_{kj}^{\mp(1)}(w)l_{11}^{\pm}(z),\nonumber\\
 \frac{z_{\pm}-w_{\mp}}{q^{-1}z_{\pm}-qw_{\mp}}l_{11}^{\pm}(z)l_{ks}^{\mp(1)}(w)=
 \frac{z_{\mp}-w_{\pm}}{q^{-1}z_{\mp}-qw_{\pm}}l_{ks}^{\mp(1)}(w)l_{11}^{\pm}(z).\label{GS5}
\end{gather}
From the $R$-matrix $\widetilde{R}(z)$, we list the coefficients as below
\begin{gather*}
 C^{\pm}\bigl(E^1_1\otimes E^k_k\bigr)=\frac{z_{\pm}-w_{\mp}}{q^{-1}z_{\pm}-qw_{\mp}}, \qquad
 C^{\pm}\bigl(E^1_j\otimes E^j_1\bigr)=(-1)^{[j]}\frac{\bigl(q^{-1}-q\bigr)z_{\pm}}{q^{-1}z_{\pm}-qw_{\mp}},
\\
 C^{\pm}\bigl(E^s_j\otimes E^j_s\bigr)=
 \begin{cases}
\displaystyle (-1)^{[j]}\frac{(q^{-1}-q)w_{\mp}}{q^{-1}z_{\pm}-qw_{\mp}},&s>j,\vspace{1mm}\\
 \displaystyle(-1)^{[j]}\frac{(q^{-1}-q)z_{\pm}}{q^{-1}z_{\pm}-qw_{\mp}},&s<j, \end{cases}
\\
 C^{\pm}\bigl(E^j_j\otimes E^s_s\bigr)=
 \begin{cases}
 \displaystyle \frac{z_{\pm}-w_{\mp}}{q^{-1}z_{\pm}-qw_{\mp}},&j\neq s,\\
 1,&j=s\quad \text{and} \quad[j]=0,\\
 \displaystyle \frac{qz_{\pm}-q^{-1}w_{\mp}}{q^{-1}z_{\pm}-qw_{\mp}},&j=s\quad \text{and} \quad[j]=1. \end{cases}
\end{gather*}
Therefore, after checking the coefficients of \eqref{GS2}--\eqref{GS5}, we can derive all the relations for the case $l=1$. The general case immediately follows from Proposition \ref{pr2}.
\end{proof}

As a consequence, we have the following result.
\begin{Proposition}\label{pr4}
Suppose that \smash{$l+1\leq j, k, s\leq \overline{(l+1)}$} and $j\neq \overline{k}$. Then in $U\bigl(\widetilde{R}^{(m|n)}\bigr)$, we have the following relations.
If $k>j$,
\begin{gather*}
 \bigl[\mathfrak{f}_{jl}^{\pm}(z), l_{ks}^{\mp(l)}(w)\bigr]=(-1)^{[j]+[l]}\left\{\frac{\bigl(q-q^{-1}\bigr)w_{\pm}}{z_{\mp}-w_{\pm}}\mathfrak{f}_{kl}^{\mp}(w)l_{js}^{\mp(l)}(w) -\frac{\bigl(q-q^{-1}\bigr)z_{\mp}}{z_{\mp}-w_{\pm}}\mathfrak{f}_{kl}^{\pm}(z)l_{js}^{\mp(l)}(w)\right\},\\
 \bigl[\mathfrak{f}_{jl}^{\pm}(z), l_{ks}^{\pm(l)}(w)\bigr]=(-1)^{[j]+[l]}\left\{\frac{\bigl(q-q^{-1}\bigr)w}{z-w}\mathfrak{f}_{kl}^{\pm}(w)l_{js}^{\pm(l)}(w) -\frac{\bigl(q-q^{-1}\bigr)z}{z-w}\mathfrak{f}_{kl}^{\pm}(z)l_{js}^{\pm(l)}(w)\right\}.
\end{gather*}
If $k<j$,
\begin{gather*}
 \bigl[\mathfrak{f}_{jl}^{\pm}(z), l_{ks}^{\mp(l)}(w)\bigr]=(-1)^{[j]+[l]}\frac{\bigl(q-q^{-1}\bigr)w_{\mp}}{z_{\mp}-w_{\pm}}\bigl\{\mathfrak{f}_{kl}^{\mp}(w)l_{js}^{\mp(l)}(w) -\mathfrak{f}_{kl}^{\pm}(z)l_{js}^{\mp(l)}(w)\bigr\},\\
 \bigl[\mathfrak{f}_{jl}^{\pm}(z), l_{ks}^{\pm(l)}(w)\bigr]=(-1)^{[j]+[l]}\frac{\bigl(q-q^{-1}\bigr)w}{z-w}\bigl\{\mathfrak{f}_{kl}^{\pm}(w)l_{js}^{\pm(l)}(w) -\mathfrak{f}_{kl}^{\pm}(z)l_{js}^{\pm(l)}(w)\bigr\}.
\end{gather*}
If $k=j$ and $[j]=0$,
\begin{gather*}
 \mathfrak{f}_{jl}^{\pm}(z)l_{js}^{\mp(l)}(w)=\frac{z_{\pm}-w_{\mp}}{q^{-1}z_{\pm}-qw_{\mp}}l_{js}^{\mp(l)}(w)\mathfrak{f}_{jl}^{\pm}(z) +\frac{\bigl(q-q^{-1}\bigr)w_{\mp}}{q^{-1}z_{\pm}-qw_{\mp}}\mathfrak{f}_{jl}^{\mp}(w)l_{js}^{\mp(l)}(w),\\
 \mathfrak{f}_{jl}^{\pm}(z)l_{js}^{\pm(l)}(w)=\frac{z-w}{q^{-1}z-qw}l_{js}^{\pm(l)}(w)\mathfrak{f}_{jl}^{\pm}(z) +\frac{\bigl(q-q^{-1}\bigr)w}{q^{-1}z-qw}\mathfrak{f}_{jl}^{\pm}(w)l_{js}^{\pm(l)}(w).
\end{gather*}
If $k=j$ and $[j]=1$,
\begin{gather*}
 \mathfrak{f}_{jl}^{\pm}(z)l_{js}^{\mp(l)}(w)=\frac{\bigl(q-q^{-1}\bigr)w_{\mp}}{q^{-1}z_{\pm}-qw_{\mp}}\bigl\{l_{js}^{\mp(l)}(w)\mathfrak{f}_{jl}^{\pm}(z) +\mathfrak{f}_{jl}^{\mp}(w)l_{js}^{\mp(l)}(w)\bigr\},\\
 \mathfrak{f}_{jl}^{\pm}(z)l_{js}^{\pm(l)}(w)=\frac{\bigl(q-q^{-1}\bigr)w}{q^{-1}z-qw}\bigl\{l_{js}^{\pm(l)}(w)\mathfrak{f}_{jl}^{\pm}(z) +\mathfrak{f}_{jl}^{\pm}(w)l_{js}^{\pm(l)}(w)\bigr\}.
\end{gather*}
\end{Proposition}

In the following, we consistently define
\begin{gather*}
 \mathfrak{e}_{i}^{\pm}(z)\doteq\mathfrak{e}_{i,i+1}^{\pm}(z), \qquad \mathfrak{f}_{i}^{\pm}(z)\doteq\mathfrak{f}_{i+1,i}^{\pm}(z).
\end{gather*}

Let \smash{$\widetilde{c}^{\pm(m|n-l)}(z)$} \big(resp.\ \smash{$\widetilde{c}^{\pm(m+n-l|0)}(z)$}\big) denote the central elements in \smash{$U^{\pm}\bigl(\widetilde{R}^{(m|n-l)}\bigr)$} if $l<n$ \big(resp.\ \smash{$U^{\pm}\bigl(\widetilde{R}^{(m+n-l|0)}\bigr)$} if $l\geq n$\big). Note that \smash{$\widetilde{c}^{\pm}(z)=\widetilde{c}^{\pm(m|n)}(z)$}.
By Proposition \ref{Apr2}, we find that
\begin{gather*}
 D\mathcal{L}^{\pm}(z\zeta)^tD^{-1}={\mathcal{L}^{\pm}(z)}^{-1}\widetilde{c}^{\pm(m|n)}(z).
\end{gather*}
Taking the $(2m+2n+1, 2m+2n+1)$-entry on both sides of the above equation and using the Gauss decomposition, we obtain
\begin{gather}\label{For2}
 \mathfrak{h}_1^{\pm}(z\zeta)={\mathfrak{h}_{\overline{1}}^{\pm}(z)}^{-1}\widetilde{c}^{\pm}(z).
\end{gather}

\begin{Lemma}\label{le3}
In $U\bigl(\widetilde{R}^{(m|n)}\bigr)$, the following equations hold:
\begin{align}
&\mathfrak{h}_i^{\pm}(z)\mathfrak{e}_i^{\pm}(z)=q^{-1}\mathfrak{e}_i^{\pm}\bigl(q^2z\bigr)\mathfrak{h}_i^{\pm}(z),\qquad 1\leq i\leq n,\label{For3} \\
&\mathfrak{f}_i^{\pm}(z)\mathfrak{h}_i^{\pm}(z)=q\mathfrak{h}_i^{\pm}(z)\mathfrak{f}_i^{\pm}\bigl(q^2z\bigr),\qquad 1\leq i\leq n, \nonumber\\
&\mathfrak{h}_{j}^{\pm}(z)\mathfrak{e}_{j}^{\pm}(z)=q\mathfrak{e}_{j}^{\pm}\bigl(q^{-2}z\bigr)\mathfrak{h}_{n+j}^{\pm}(z),\qquad n+1\leq j\leq m+n, \label{For4}\\
&\mathfrak{f}_{j}^{\pm}(z)\mathfrak{h}_{j}^{\pm}(z)=q^{-1}\mathfrak{h}_{j}^{\pm}(z) \mathfrak{f}_{j}^{\pm}\bigl(q^{-2}z\bigr),
\qquad n+1\leq j\leq m+n.\nonumber
\end{align}
\end{Lemma}

\begin{proof}
By Theorem \ref{th2} and Proposition \ref{pr2}, we consider the $R$-matrices $\widetilde{R}^{(m|n-i+1)}$ with $1\leq i\leq n$ and $\widetilde{R}^{(m+n-j+1|0)}$ with $n+1\leq j\leq m+n$, corresponding to the superalgebras $U\bigl(\widetilde{R}^{(m|n-i+1)}\bigr)$ (for~${1\leq i\leq n}$) and \smash{$U\bigl(\widetilde{R}^{(m+n-j+1|0)}\bigr)$} (for $n+1\leq j\leq m+n$), respectively. Let $s=i$ or $j$, then $s+1\neq \overline{s}$. Applying the Gauss decomposition and defining relation \eqref{d2}, we have
\begin{align*}
 & C(E^s_s\otimes E^s_s) \mathfrak{h}_s^{\pm}(z)\mathfrak{h}_s^{\pm}(w)\mathfrak{e}_s^{\pm}(w)\\
 &\qquad =\mathfrak{h}_s^{\pm}(w)
 \mathfrak{h}_s^{\pm}(z)e_s^{\pm}(z)C\bigl(E^{s+1}_s\otimes E^s_{s+1}\bigr) + \mathfrak{h}_s^{\pm}(w)\mathfrak{e}_s^{\pm}(w)\mathfrak{h}_s^{\pm}(z)C\bigl(E^s_s\otimes E^{s+1}_{s+1}\bigr).
\end{align*}
By the super case of \cite[relations (33)--(34)]{JSKW}, \cite[Definition~2]{HBKD} for type A, and the non-super case for type B \cite[Theorem 4.29]{JMAB}, we have the commutation relations between $\mathfrak{h}_s^{\pm}(z)$ and $\mathfrak{h}_s^{\pm}(w)$ (also see Section~\ref{sec5.2}). Therefore, let $w=q^{2}z$ for $s=i$ or $w=q^{-2}z$ for $s=j$. By examining the coefficients of the above equation, we can obtain the relations involving the generators $\mathfrak{h}_s^{\pm}(z)$ and $\mathfrak{e}_s^{\pm}(z)$, and the others are similar.
\end{proof}

\begin{Lemma}\label{le4}In
$U\bigl(\widetilde{R}^{(m|n)}\bigr)$, we have
\begin{gather}\label{For5}
 \mathfrak{e}^{\pm}_{\overline{(i+1)}}(z)=-\mathfrak{e}^{\pm}_{i}\bigl(q^{2i}z\zeta\bigr), \qquad 1\leq i\leq n,\\
 \label{For6}
 \mathfrak{e}^{\pm}_{\overline{(j+1)}}(z)=-\mathfrak{e}^{\pm}_{j}\bigl(q^{2m+2n-2j-1}z\bigr), \qquad n+1\leq j\leq m+n-1,
\end{gather}
and
\begin{gather}\label{For7}
 \mathfrak{f}^{\pm}_{\overline{(i+1)}}(z)=-\mathfrak{f}^{\pm}_{i}\bigl(q^{2i}z\zeta\bigr), \qquad 1\leq i\leq n,\\
 \label{For8}
 \mathfrak{f}^{\pm}_{\overline{(j+1)}}(z)=-\mathfrak{f}^{\pm}_{j}\bigl(q^{2m+2n-2j-1}z\bigr), \qquad n+1\leq j\leq m+n-1.
\end{gather}
\end{Lemma}
\begin{proof}
By Propositions \ref{Apr2} and \ref{pr2}, for any $1\leq i\leq n$ and $n+1\leq j\leq m+n-1$, we have
\begin{gather}\label{For9}
 {\mathcal{L}^{\pm (i-1)}(z)}^{-1}\widetilde{c}^{\pm (m|n-i+1)}(z)=D^{(m|n-i+1)}{\mathcal{L}^{\pm(i)}\bigl(q^{2i-2}z\zeta\bigr)}^{t}{\bigl(D^{(m|n-i+1)}\bigr)}^{-1},\\
 \label{For10}
 {\mathcal{L}^{\pm (j-1)}(z)}^{-1}\widetilde{c}^{\pm (m+n-j+1|0)}(z)\nonumber\\
 \qquad=D^{(m+n-j+1|0)}{\mathcal{L}^{\pm(j-1)}\bigl(q^{-2j+2}zq^{2m+2n-1}\bigr)}^{t}{\bigl(D^{(m+n-j+1|0)}\bigr)}^{-1},
\end{gather}
where
\begin{gather*}
 D^{(m|n-i+1)}=\operatorname{diag}[q^{a_{i}},\dots,q^{a_{\overline{i}}}],\qquad
 D^{(m+n-j+1|0)}=\operatorname{diag}[q^{a_{j}},\dots,q^{a_{\overline{j}}}].
\end{gather*}Let $s=i$ or $j$ and consider the $(\overline{s},\overline{s})$ and $(\overline{s+1},\overline{s})$-entries on both sides of \eqref{For9} and \eqref{For10}, respectively, we find that
\begin{gather*}
 \mathfrak{h}_i^{\pm}\bigl(q^{2i-2}z\zeta\bigr)={\mathfrak{h}_{\overline{i}}^{\pm}(z)}^{-1}\widetilde{c}^{\pm(m|n-i+1)}(z),\\
 -\mathfrak{e}_{\overline{i+1}}^{\pm}(z){h_{\overline{i}}^{\pm}(z)}^{-1}\widetilde{c}^{\pm(m|n-i+1)}(z)
 =q\mathfrak{h}_i^{\pm}\bigl(q^{2i-2}z\zeta\bigr)\mathfrak{e}_i^{\pm}\bigl(q^{2i-2}z\zeta\bigr),
\end{gather*}
and
\begin{gather*}
 \mathfrak{h}_{j}^{\pm}\bigl(zq^{2m+2n-2j+1}\bigr)={\mathfrak{h}_{\overline{j}}^{\pm}(z)}^{-1}\widetilde{c}^{\pm(m+n-j+1|0)}(z),\\
 -\mathfrak{e}_{\overline{j+1}}^{\pm}(z){\mathfrak{h}_{\overline{j}}^{\pm}(z)}^{-1}\widetilde{c}^{\pm(m+n-j+1|0)}(z)
 =q^{-1}\mathfrak{h}_{j}^{\pm}\bigl(zq^{2m+2n-2j+1}\bigr)\mathfrak{e}_{j}^{\pm}\bigl(zq^{2m+2n-2j+1}\bigr).
\end{gather*}
As a consequence, we obtain
 \begin{gather*}
 -\mathfrak{e}_{\overline{i+1}}^{\pm}(z)\mathfrak{h}_i^{\pm}\bigl(q^{2i-2}z\zeta\bigr)
 =q\mathfrak{h}_i^{\pm}\bigl(q^{2i-2}z\zeta\bigr)\mathfrak{e}_i^{\pm}\bigl(q^{2i-2}z\zeta\bigr),
\end{gather*}and
 \begin{gather*}
 -\mathfrak{e}_{\overline{j+1}}^{\pm}(z)\mathfrak{h}_{j}^{\pm}\bigl(zq^{2m+2n-2j+1}\bigr)
 =q^{-1}\mathfrak{h}_i^{\pm}\bigl(zq^{2m+2n-2j+1}\bigr)\mathfrak{e}_{j}^{\pm}\bigl(zq^{2m+2n-2j+1}\bigr).
\end{gather*}
Now \eqref{For3} and \eqref{For4} imply our claims of \eqref{For5} and \eqref{For6}, while \eqref{For7} and \eqref{For8} are similar.
\end{proof}

\begin{Proposition}\label{pr5}
In the superalgebra $U\bigl(R^{(m|n)}\bigr)$ and $U\bigl(\widetilde{R}^{(m|n)}\bigr)$, we have
\begin{gather*}
c^{\pm(m|n)}(z)=\prod\limits_{i=1}^n\frac{h_i^{\pm}\bigl(z\zeta q^{2i-2}\bigr)}{h_i^{\pm}\bigl(z\zeta q^{2i}\bigr)}\prod\limits_{i=1}^m\frac{h_{n+i}^{\pm}\bigl(z q^{2m-2i+1}\bigr)}{h_{n+i}^{\pm}\bigl(zq^{2m-2i-1}\bigr)}h_{m+n+1}^{\pm}(z)h_{m+n+1}^{\pm}\bigl(q^{-1}z\bigr),\\
\widetilde{c}^{\pm(m|n)}(z)=\prod\limits_{i=1}^n\frac{\mathfrak{h}_i^{\pm}\bigl(z\zeta q^{2i-2}\bigr)}{\mathfrak{h}_i^{\pm}\bigl(z\zeta q^{2i}\bigr)}\prod\limits_{i=1}^m\frac{\mathfrak{h}_{n+i}^{\pm}\bigl(z q^{2m-2i+1}\bigr)}{\mathfrak{h}_{n+i}^{\pm}\bigl(zq^{2m-2i-1}\bigr)}\mathfrak{h}_{m+n+1}^{\pm}(z)\mathfrak{h}_{m+n+1}^{\pm}\bigl(q^{-1}z\bigr).
\end{gather*}
\end{Proposition}
\begin{proof}
Considering the matrix $ \mathcal{L}^{\pm (0)}(z)=\mathcal{L}^{\pm }(z)$ and taking the $(\overline{2},\overline{2})$-entry of \eqref{For9}, by the Gauss decomposition we find that
\begin{gather*}
 \mathfrak{h}_2^{\pm}(z\zeta)+\mathfrak{f}_1^{\pm}(z\zeta)\mathfrak{h}_1^{\pm}(z\zeta)\mathfrak{e}_1^{\pm}(z\zeta)=
 \bigl\{{\mathfrak{h}_{\overline{2}}^{\pm}(z)}^{-1}+\mathfrak{e}_{\overline{1}}^{\pm}(z){\mathfrak{h}_{\overline{1}}^{\pm}(z)}^{-1}
 {\mathfrak{f}_{\overline{1}}^{\pm}(z)}^{-1}\bigr\}\widetilde{c}^{\pm(m|n)}(z).
\end{gather*}Proposition \ref{Apr2} and Lemma \ref{le3} together with \eqref{For2} imply
\begin{gather*}
 {\mathfrak{h}_{\overline{2}}^{\pm}(z)}^{-1}\widetilde{c}^{\pm(m|n)}(z)=\mathfrak{h}_2^{\pm}(z\zeta)+\mathfrak{f}_1^{\pm}(z\zeta)\mathfrak{h}_1^{\pm}(z\zeta)
 \mathfrak{e}_1^{\pm}(z\zeta)-\mathfrak{e}_1^{\pm}\bigl(q^{2}z\zeta\bigr)\mathfrak{h}_1^{\pm}(z\zeta)\mathfrak{f}_1^{\pm}(z\zeta).
\end{gather*}
As the proof of \cite{HBKD} (also see \cite{JSKW} or \cite{YZC}), we have
\begin{gather*}
\bigl[\mathfrak{e}_1^{\pm}(z), \mathfrak{f}_1^{\pm}(w)\bigr]=\frac{\bigl(q-q^{-1}\bigr)z}{z-w}\bigl(\mathfrak{h}_2^{\pm}(w){\mathfrak{h}_1^{\pm}(w)}^{-1}-\mathfrak{h}_2^{\pm}(z)
{\mathfrak{h}_1^{\pm}(z)}^{-1}\bigr),
\end{gather*}together with \eqref{For3}, we deduce that
\begin{gather*}
 {\mathfrak{h}_{\overline{2}}^{\pm}(z)}^{-1}\widetilde{c}^{\pm(m|n)}(z)=\mathfrak{h}_2^{\pm}\bigl(q^{2}z\zeta\bigr)
 {\mathfrak{h}_1^{\pm}\bigl(q^{2}z\zeta\bigr)}^{-1}\mathfrak{h}_1^{\pm}(z\zeta).
\end{gather*}
On one hand, $\widetilde{c}^{\pm(m|n-1)}(z)=\mathfrak{h}_{\overline{2}}^{\pm}(z)\mathfrak{h}_2^{\pm}\bigl(q^{2}z\zeta\bigr)$, so that
\begin{gather*}
 \widetilde{c}^{\pm(m|n)}(z)=
 {\mathfrak{h}_1^{\pm}\bigl(q^{2}z\zeta\bigr)}^{-1}\mathfrak{h}_1^{\pm}(z\zeta)\widetilde{c}^{\pm(m|n-1)}(z).
\end{gather*}
Repeat this process for $\widetilde{c}^{\pm(m|n-l)}(z)$, and when $l>n$, relations \eqref{For4} are used. Thus we only need to know the formulas for $\widetilde{c}^{\pm(1|0)}(z)$. Considering the superalgebra
 $U\bigl(\widetilde{R}^{(1|0)}\bigr)$ (it should be noted that the required relations can be obtained in \cite{JMAB} based on Remark \ref{RM2}\,(1)), then the same argument allows us to conclude that
\begin{gather*}
 {\mathfrak{h}_{m+n+1}^{\pm}(z)}^{-1}\widetilde{c}^{\pm(1|0)}(z)=\mathfrak{h}_{m+n+1}^{\pm}\bigl(q^{-1}z\bigr)
 {\mathfrak{h}_{m+n}^{\pm}\bigl(q^{-1}z\bigr)}^{-1}\mathfrak{h}_{m+n}^{\pm}(qz),
\end{gather*}which implies that the formula $\widetilde{c}^{\pm(m|n)}(z)$ holds. The formula $c^{\pm(m|n)}(z)$ follows from Proposition~\ref{Apr4}.
\end{proof}

\subsection{Relations of Drinfeld generators}\label{sec5.2}

Now we illustrate the Drinfeld generators and relations in $U\bigl(\widetilde{R}\bigr)$ and $U(R)$ by setting
\begin{gather*}
 \widetilde{X}^{-}_i(z)=\mathfrak{e}^+_{i}(z_+)-\mathfrak{e}^-_{i}(z_-),\qquad
 \widetilde{X}^{+}_i(z)=\mathfrak{f}^+_{i}(z_-)-\mathfrak{f}^-_{i}(z_+),\\
 X^{-}_i(z)=e^+_{i}(z_+)-e^-_{i}(z_-),\qquad
 X^{+}_i(z)=f^+_{i}(z_-)-f^-_{i}(z_+),
\end{gather*}
and the $\delta$-function
$ \delta(z)=\sum_{p\in\mathbb{Z}}z^p$.
Note that the $R$-matrix of \smash{$U_q\bigl(\widehat{\mathfrak{gl}(n|m)}\bigr)$} is
\begin{align*}
 \overline{R}(z)={}&\frac{q^{-1}-zq}{q-q^{-1}z}\sum\limits_{a=1 }^{n}E^a_a\otimes E^a_a+\sum\limits_{a=n+1 }^{n+m}E^a_a\otimes E^a_a-\frac{z-1}{q-q^{-1}z}\sum\limits_{a\neq b}(-1)^{[a][b]}E^a_a\otimes E^b_b\nonumber\\
 &+\frac{q-q^{-1}}{q-q^{-1}z}\sum\limits_{a>b}(-1)^{[b]}E^a_b\otimes E^b_a+\frac{\bigl(q-q^{-1}\bigr)z}{q-q^{-1}z}\sum\limits_{a<b }(-1)^{[b]}E^a_b\otimes E^b_a.
\end{align*}
Compare to the $R$-matrix $\widetilde{R}(z)$ and by Remark \ref{RM2} along with the quasideterminant formulas, in the same way as presented in \cite{JSKW,HBKD} and \cite{YZC} (where the original method was provided by \cite{JDIB} for the non-super case), we can arrive at the following proposition.
\begin{Proposition} In the superalgebra $U\bigl(\widetilde{R}^{(m|n)}\bigr)$, we have
\begin{gather*}
 \mathfrak{h}_i^{\pm}(z)\mathfrak{h}_i^{\epsilon}(w)=\mathfrak{h}_i^{\epsilon}(w)\mathfrak{h}_i^{\pm}(z),\qquad n+1\leq i\leq n+m,\\
 \mathfrak{h}_j^{\pm}(z)\mathfrak{h}_j^{\pm}(w)=\mathfrak{h}_j^{\pm}(w)\mathfrak{h}_j^{\pm}(z),\qquad 1\leq j\leq n,\\
 \frac{qz_+-q^{-1}w_-}{q^{-1}z_+-qw_-}\mathfrak{h}_j^{\pm}(z)\mathfrak{h}_j^{\mp}(w)=\frac{qz_--q^{-1}w_+}{q^{-1}z_--qw_+}\mathfrak{h}_j^{\mp}(w)\mathfrak{h}_j^{\pm}(z),\qquad 1\leq j\leq n,\\
 \frac{z_{\pm}-w_{\mp}}{q^{-1}z_{\pm}-qw_{\mp}}\mathfrak{h}_i^{\pm}(z)\mathfrak{h}_j^{\mp}(w)=\frac{z_{\mp}-w_{\pm}}{q^{-1}z_{\mp}-qw_{\pm}}\mathfrak{h}_j^{\mp}(w)\mathfrak{h}_i^{\pm}(z), \qquad \qquad 1\leq i<j\leq m+n,\\
 {\mathfrak{h}_i^{\pm}(z)}^{-1}\widetilde{X}^-_i(w)\mathfrak{h}_i^{\pm}(z)=\frac{q^{-1}z_{\mp}-qw}{z_{\mp}-w}\widetilde{X}^-_i(w),\qquad n+1\leq i\leq n+m-1,\\
 {\mathfrak{h}_{i+1}^{\pm}(z)}^{-1}\widetilde{X}^-_i(w)\mathfrak{h}_{i+1}^{\pm}(z)=\frac{qz_{\mp}-q^{-1}w}{z_{\mp}-w}\widetilde{X}^-_i(w),\qquad n+1\leq i\leq n+m-1,\\
 \mathfrak{h}_i^{\pm}(z)\widetilde{X}^+_i(w){\mathfrak{h}_i^{\pm}(z)}^{-1}=\frac{q^{-1}z_{\pm}-qw}{z_{\pm}-w}\widetilde{X}^+_i(w),\qquad n+1\leq i\leq n+m-1,\\
 \mathfrak{h}_{i+1}^{\pm}(z)\widetilde{X}^+_i(w){\mathfrak{h}_{i+1}^{\pm}(z)}^{-1}=\frac{qz_{\pm}-q^{-1}w}{z_{\pm}-w}\widetilde{X}^+_i(w),\qquad n+1\leq i\leq n+m-1,\\
 {\mathfrak{h}_j^{\pm}(z)}^{-1}\widetilde{X}^-_j(w)\mathfrak{h}_j^{\pm}(z)=\frac{qz_{\mp}-q^{-1}w}{z_{\mp}-w}\widetilde{X}^-_j(w),\qquad 1\leq j\leq n,\\
 {\mathfrak{h}_{j+1}^{\pm}(z)}^{-1}\widetilde{X}^-_j(w)\mathfrak{h}_{j+1}^{\pm}(z)=\frac{q^{-1}z_{\mp}-qw}{z_{\mp}-w}\widetilde{X}^-_j(w),\qquad 1\leq j\leq n,
\\
 \mathfrak{h}_j^{\pm}(z)\widetilde{X}^+_j(w){\mathfrak{h}_j^{\pm}(z)}^{-1}=\frac{qz_{\pm}-q^{-1}w}{z_{\pm}-w}\widetilde{X}^+_j(w),\qquad 1\leq j\leq n,\\
 \mathfrak{h}_{j+1}^{\pm}(z)\widetilde{X}^+_j(w){\mathfrak{h}_{j+1}^{\pm}(z)}^{-1}=\frac{q^{-1}z_{\pm}-qw}{z_{\pm}-w}\widetilde{X}^+_j(w),\qquad 1\leq j\leq n,\\
 \bigl(q^{\mp1}z-q^{\pm1}w\bigr)\widetilde{X}_i^{\pm}(z)\widetilde{X}_i^{\pm}(w)=\bigl(q^{\pm1}z-q^{\mp1}w\bigr)\widetilde{X}_i^{\pm}(w)\widetilde{X}_i^{\pm}(z),\qquad n+1\leq i\leq n+m-1,\\
 \bigl(q^{\mp1}z-q^{\pm1}w\bigr)\widetilde{X}_{i-1}^{\pm}(z)\widetilde{X}_{i}^{\pm}(w)=(z-w)\widetilde{X}_{i}^{\pm}(w)\widetilde{X}_{i-1}^{\pm}(z),\qquad n+1\leq i\leq n+m-1,\\
 \bigl(q^{\pm1}z-q^{\mp1}w\bigr)\widetilde{X}_j^{\pm}(z)\widetilde{X}_j^{\pm}(w)=\bigl(q^{\mp1}z-q^{\pm1}w\bigr)\widetilde{X}_j^{\pm}(w)X_j^{\pm}(z),\qquad 1\leq j\leq n-1,\\
 \bigl(q^{\pm1}z-q^{\mp1}w\bigr)\widetilde{X}_{j-1}^{\pm}(z)X_{j}^{\pm}(w)=(z-w)\widetilde{X}_{j}^{\pm}(w)\widetilde{X}_{j-1}^{\pm}(z),\qquad 1\leq j\leq n,\\
 \widetilde{X}_{n}^{\pm}(z)\widetilde{X}_{n}^{\pm}(w)=-\widetilde{X}_{n}^{\pm}(w)\widetilde{X}_{n}^{\pm}(z),
\end{gather*}
together with
\begin{align*}
 \bigl[\widetilde{X}_i^{+}(z), \widetilde{X}_j^{-}(w)\bigr]={}&\bigl(q-q^{-1}\bigr)\delta_{ij}\\
 &\times\left(\delta\left(\frac{zq^{-c}}{w}\right)\mathfrak{h}_{i+1}^+(w_+)
 \mathfrak{h}_i^+(w_+)^{-1}-\delta\left(\frac{zq^{c}}{w}\right)\mathfrak{h}_{i+1}^-(z_+)\mathfrak{h}_i^-(z_+)^{-1}\right)
\end{align*}
for $1\leq i\leq m+n-1$. The commutation relations for $\epsilon=\pm$ are as follows:
\begin{align*}
 &\widetilde{X}_i^{\pm}(z)\widetilde{X}_j^{\pm}(w)=\widetilde{X}_j^{\pm}(w)\widetilde{X}_i^{\pm}, \qquad 1\leq i,j \leq m+n-1,\qquad |i-j|>1,\\
 &\mathfrak{h}_i^{\pm}(z)\widetilde{X}_j^{\epsilon}(w)=\widetilde{X}_j^{\epsilon}(w)\mathfrak{h}_i^{\pm}(z), \qquad 1\leq i \leq m+n,\qquad 1\leq j \leq m+n-1,\qquad |i-j|>1.
\end{align*}
\end{Proposition}

Let $m=1, n=0$, then there is an $R$-matrix $R^{(1|0)}(z)$ associated with the Lie superalgebra~\smash{${\mathfrak{osp}_{3|0}}$} ($\cong \mathfrak{o}_3$). By Remark \ref{RM2}\,(1), considering the decomposition of the Gaussian generators~$\mathfrak{h}_k^{\pm}(z)$,~$\mathfrak{e}_i^{\pm}(z)$,~$\mathfrak{f}_j^{\pm}(w)$ in terms of the series $l^{\pm}_{ij}(z)$, we can perform the same calculations as in the non-super case of type~B (cf.\ \cite[Lemmas 4.8--4.11]{JMAB}). This yields the following relations directly.

\begin{Proposition}In the superalgebra $U\bigl(\widetilde{R}^{(1|0)}\bigr)$, it holds that
\begin{gather*}
 \mathfrak{h}_1^{\pm}(z)\mathfrak{h}_1^{\pm}(w)=\mathfrak{h}_1^{\pm}(w)\mathfrak{h}_1^{\pm}(z),\qquad
 \mathfrak{h}_1^{\pm}(z)\mathfrak{h}_1^{\mp}(w)=\mathfrak{h}_1^{\mp}(w)\mathfrak{h}_1^{\pm}(z),\\
 \mathfrak{h}_1^{\pm}(z)\mathfrak{h}_2^{\pm}(w)=\mathfrak{h}_2^{\pm}(w)\mathfrak{h}_1^{\pm}(z),\qquad
 \frac{z_{\pm}-w_{\mp}}{q^{-1}z_{\pm}-qw_{\mp}}
 \mathfrak{h}_1^{\pm}(z)\mathfrak{h}_2^{\mp}(w)=\frac{z_{\mp}-w_{\pm}}{q^{-1}z_{\mp}-qw_{\pm}}
 \mathfrak{h}_2^{\mp}(w)\mathfrak{h}_1^{\pm}(z),\\
 \mathfrak{h}_2^{\pm}(z)\mathfrak{h}_2^{\pm}(w)=\mathfrak{h}_2^{\pm}(w)\mathfrak{h}_2^{\pm}(z),\\
 \frac{\bigl(qz_{\pm}-q^{-1}w_{\mp}\bigr)\bigl(q^{-\frac{1}{2}}z_{\pm}-q^{\frac{1}{2}}w_{\mp}\bigr)}{\bigl(q^{-1}z_{\pm}-qw_{\mp}\bigr)
 \bigl(q^{\frac{1}{2}}z_{\pm}-q^{-\frac{1}{2}}w_{\mp}\bigr)}\mathfrak{h}_2^{\pm}(z)\mathfrak{h}_2^{\mp}(w)\\
 \qquad
 =\frac{\bigl(qz_{\mp}-q^{-1}w_{\pm}\bigr)\bigl(q^{-\frac{1}{2}}z_{\mp}-q^{\frac{1}{2}}w_{\pm}\bigr)}{\bigl(q^{-1}z_{\pm}-qw_{\pm}\bigr)
 \bigl(q^{\frac{1}{2}}z_{\mp}-q^{-\frac{1}{2}}w_{\pm}\bigr)}\mathfrak{h}_2^{\mp}(w)\mathfrak{h}_2^{\pm}(z),\\
 {\mathfrak{h}_1^{\pm}(z)}^{-1}\widetilde{X}^-_1(w)\mathfrak{h}_1^{\pm}(z)=\frac{q^{-1}z_{\mp}-qw}{z_{\mp}-w}\widetilde{X}^-_1(w),\\
 \mathfrak{h}_1^{\pm}(z)\widetilde{X}^+_1(w){\mathfrak{h}_1^{\pm}(z)}^{-1}=\frac{q^{-1}z_{\pm}-qw}{z_{\pm}-w}\widetilde{X}^+_1(w),
\\
 {\mathfrak{h}_{2}^{\pm}(z)}^{-1}\widetilde{X}_{1}^-(w)\mathfrak{h}_{2}^{\pm}(z)=\frac{\bigl(q^{-1}z_{\mp}-w\bigr)(z_{\mp}-w)}
 {\bigl(z_{\mp}-q^{-1}w\bigr)\bigl(q^{-1}z_{\mp}-qw\bigr)}\widetilde{X}_{1}^-(w),\\
 \label{dr9}
 \mathfrak{h}_{2}^{\pm}(z)\widetilde{X}_{1}^+(w){\mathfrak{h}_{2}^{\pm}(z)}^{-1}=\frac{\bigl(q^{-1}z_{\pm}-w\bigr)(z_{\pm}-w)}
 {\bigl(z_{\pm}-q^{-1}w\bigr)\bigl(q^{-1}z_{\pm}-qw\bigr)}\widetilde{X}_{1}^+(w),\\
\label{dr10}
 \bigl(q^{\mp1}z-q^{\pm1}w\bigr)\widetilde{X}_1^{\pm}(z)\widetilde{X}_1^{\pm}(w)=\bigl(q^{\pm1}z-q^{\mp1}w\bigr)\widetilde{X}_1^{\pm}(w)\widetilde{X}_1^{\pm}(z),\\
 \bigl[\widetilde{X}_1^{+}(z), \widetilde{X}_1^{-}(w)\bigr]=\bigl(q^{1/2}-q^{-1/2}\bigr)\\
 \phantom{ \bigl[\widetilde{X}_1^{+}(z), \widetilde{X}_1^{-}(w)\bigr]=}{}\times\left(\delta\left(\frac{zq^{-c}}{w}\right)h_{2}^+(w_+)h_1^+(w_+)^{-1}
 -\delta\left(\frac{zq^{c}}{w}\right)h_{2}^-(z_+)h_1^-(z_+)^{-1}\right).
\end{gather*}
\end{Proposition}

Moreover, by Corollary \ref{co1}, we have the following Propositions immediately.
\begin{Proposition}In the algebra $U\bigl(\widetilde{R}^{(m|n)}\bigr)$, it holds that
\begin{align*}
&\mathfrak{h}_i^{\pm}(z)\mathfrak{h}_{m+n+1}^{\pm}(w)=\mathfrak{h}_{m+n+1}^{\pm}(w)\mathfrak{h}_i^{\pm}(z),\qquad i\leq m+n,\\
&\frac{z_{\pm}-w_{\mp}}{q^{-1}z_{\pm}-qw_{\mp}}\mathfrak{h}_i^{\pm}(z)\mathfrak{h}_{m+n+1}^{\mp}(w)=\frac{z_{\mp}-w_{\pm}}{q^{-1}z_{\mp}-qw_{\pm}}\mathfrak{h}_{m+n+1}^{\mp}(w)\mathfrak{h}_i^{\pm}(z),\qquad i\leq m+n.
\end{align*}
\end{Proposition}
\begin{Proposition}In the superalgebra $U\bigl(\widetilde{R}^{(m|n)}\bigr)$, we have the commutation relations as follows:
\begin{align*}
&\mathfrak{e}_i^{\pm}(z)\mathfrak{h}_{m+n+1}^{\mp}(w)=\mathfrak{h}_{m+n+1}^{\mp}(w)\mathfrak{e}_i^{\pm}(z),\\ &\mathfrak{e}_i^{\pm}(z)\mathfrak{h}_{m+n+1}^{\pm}(w)=\mathfrak{h}_{m+n+1}^{\pm}(w)\mathfrak{e}_i^{\pm}(z),\qquad i\leq m+n,\\
&\mathfrak{f}_i^{\pm}(z)\mathfrak{h}_{m+n+1}^{\mp}(w)=\mathfrak{h}_{m+n+1}^{\mp}(w)\mathfrak{f}_i^{\pm}(z),\\ &\mathfrak{f}_i^{\pm}(z)\mathfrak{h}_{m+n+1}^{\pm}(w)=\mathfrak{h}_{m+n+1}^{\pm}(w)\mathfrak{f}_i^{\pm}(z),\qquad i\leq m+n,\\
&\mathfrak{e}_i^{\pm}(z)\mathfrak{f}_{m+n}^{\mp}(w)=\mathfrak{f}_{m+n}^{\mp}(w)\mathfrak{e}_i^{\pm}(z),\qquad \mathfrak{e}_i^{\pm}(z)\mathfrak{f}_{m+n}^{\pm}(w)=\mathfrak{f}_{m+n}^{\pm}(w)\mathfrak{e}_i^{\pm}(z),\qquad i\leq m+n-1,\\
&\mathfrak{e}_{m+n}^{\pm}(z)\mathfrak{f}_i^{\mp}(w)=\mathfrak{f}_i^{\mp}(w)\mathfrak{e}_{m+n}^{\pm}(z),\qquad \mathfrak{e}_{m+n}^{\pm}(z)\mathfrak{f}_i^{\pm}(w)=\mathfrak{f}_i^{\pm}(w)\mathfrak{e}_{m+n}^{\pm}(z), \qquad i\leq m+n-1,\\
&\mathfrak{e}_{m+n}^{\pm}(z)\mathfrak{e}_i^{\mp}(w)=\mathfrak{e}_i^{\mp}(w)\mathfrak{e}_{m+n}^{\pm}(z),\qquad \mathfrak{e}_{m+n}^{\pm}(z)\mathfrak{e}_i^{\pm}(w)=\mathfrak{e}_i^{\pm}(w)\mathfrak{e}_{m+n}^{\pm}(z), \qquad i\leq m+n-2,\\
&\mathfrak{f}_i^{\pm}(z)\mathfrak{f}_{m+n}^{\mp}(w)=\mathfrak{f}_{m+n}^{\mp}(w)\mathfrak{f}_i^{\pm}(z),\qquad \mathfrak{f}_i^{\pm}(z)\mathfrak{f}_{m+n}^{\pm}(w)=\mathfrak{f}_{m+n}^{\pm}(w)\mathfrak{f}_i^{\pm}(z),\qquad i\leq m+n-2.
\end{align*}
\end{Proposition}
\begin{Proposition}In the superalgebra $U\bigl(\widetilde{R}^{(m|n)}\bigr)$, we have the non-commutation relations as follows:
\begin{gather*}
 \bigl(q^{-1}z_{\mp} -qw_{\pm}\bigr)\mathfrak{e}_{m+n-1}^{\pm}(z)\mathfrak{e}_{m+n}^{\mp}(w)\\
 \qquad=(z_{\mp}-w_{\pm})\mathfrak{e}_{m+n}^{\mp}(w)\mathfrak{e}_{m+n-1}^{\pm}(z)
+\bigl(q^{-1}-q\bigr)w_{\pm} \mathfrak{e}_{m+n-1,m+n+1}^{\pm}(z)\\
\phantom{ \qquad=}{}- \bigl(q^{-1}-q\bigr)z_{\mp}\mathfrak{e}_{m+n-1}^{\mp}(w)\mathfrak{e}_{m+n}^{\mp}(w)-\bigl(q^{-1}-q\bigr)z_{\mp}\mathfrak{e}_{m+n-1,m+n+1}^{\mp}(w),
\\
 \bigl(q^{-1}z- qw\bigr)\mathfrak{e}_{m+n-1}^{\pm}(z)\mathfrak{e}_{m+n}^{\pm}(w)\\
 \qquad=(z-w)\mathfrak{e}_{m+n}^{\pm}(w)\mathfrak{e}_{m+n-1}^{\pm}(z)
+\bigl(q^{-1}-q\bigr)w\mathfrak{e}_{m+n-1,m+n+1}^{\pm}(z)
\\
\phantom{ \qquad=}{}-\bigl(q^{-1}-q\bigr)z\mathfrak{e}_{m+n-1}^{\pm}(w)\mathfrak{e}_{m+n}^{\pm}(w)-\bigl(q^{-1}-q\bigr)z\mathfrak{e}_{m+n-1,m+n+1}^{\pm}(w),
\\
 (z_{\pm}-w_{\mp} )\mathfrak{f}_{m+n-1}^{\pm}(z)\mathfrak{f}_{m+n}^{\mp}(w)\\
 \qquad=\bigl(q^{-1}z_{\pm}-qw_{\mp}\bigr)\mathfrak{f}_{m+n}^{\mp}(w)\mathfrak{f}_{m+n-1}^{\pm}(z)
+\bigl(q^{-1}-q\bigr)w_{\mp}\mathfrak{f}_{m+n+1,m+n-1}^{\pm}(w)\\
\phantom{ \qquad=}{}-\bigl(q^{-1}-q\bigr)w_{\mp}\mathfrak{f}_{m+n}^{\mp}(w)\mathfrak{f}_{m+n-1}^{\mp}(w) -\bigl(q^{-1}-q\bigr)z_{\pm}\mathfrak{f}_{m+n-1,m+n+1}^{\pm}(z),
\\
 (z-w) f_{m+n-1}^{\pm}(z)\mathfrak{f}_{m+n}^{\mp}(w)\\
\qquad =\bigl(q^{-1}z-qw\bigr)\mathfrak{f}_{m+n}^{\mp}(w)f_{m+n-1}^{\pm}(z)
+\bigl(q^{-1}-q\bigr)w\mathfrak{f}_{m+n-1,m+n+1}^{\pm}(w)
\\
 \phantom{ \qquad=}{}-\bigl(q^{-1}-q\bigr)w\mathfrak{f}_{m+n}^{\mp}(w)\mathfrak{f}_{m+n-1}^{\mp}(w)-\bigl(q^{-1}-q\bigr)z\mathfrak{f}_{m+n-1,m+n+1}^{\pm}(z).
\end{gather*}
\end{Proposition}
\begin{proof}
Indeed, by \eqref{GS9} and \eqref{GS11} we have
\begin{align*}
 &\mathfrak{e}_{m+n-1}^{\pm}(z)\mathfrak{h}_{m+n}^{\mp}(w)\mathfrak{e}_{m+n}^{\mp}(w)-\mathfrak{h}_{m+n}^{\mp}(w)\mathfrak{e}_{m+n}^{\mp}(w)\mathfrak{e}_{m+n-1}^{\pm}(z)\\
&\qquad=\frac{\bigl(q^{-1}-q\bigr)w_{\pm}}{z_{\mp}-w_{\pm}}\mathfrak{h}_{m+n}^{\mp}(w)\mathfrak{e}_{m+n-1,m+n+1}^{\pm}(z)
-\frac{\bigl(q^{-1}-q\bigr)z_{\mp}}{z_{\mp}-w_{\pm}}\mathfrak{h}_{m+n}^{\mp}(w)\mathfrak{e}_{m+n-1,m+n+1}^{\mp}(w),
\end{align*}and
\begin{align*}
 \mathfrak{e}_{m+n-1}^{\pm}(z)\mathfrak{h}_{m+n}^{\mp}(w)={}&\frac{q^{-1}z_{\mp}-qw_{\pm}}{z_{\mp}-w_{\pm}}\mathfrak{}h_{m+n}^{\mp}(w)\mathfrak{e}_{m+n-1}^{\pm}(z)\\
&-\frac{\bigl(q-q^{-1}\bigr)z_{\mp}}{z_{\mp}-w_{\pm}}\mathfrak{h}_{m+n}^{\mp}(w)\mathfrak{e}_{m+n-1}^{\mp}(w).
\end{align*}
Hence, those two equations give the claims of relations $\mathfrak{e}_{m+n-1}^{\pm}(z)\mathfrak{e}_{m+n}^{\mp}(w)$, and the others are similar.
\end{proof}

Now, from the above results and applying Theorem \ref{th2} and Proposition \ref{pr2}, we conclude the following theorem.
\begin{Theorem}\label{th3}
$(1)$ In the super $R$-matrix algebra $U\bigl(\widetilde{R}^{(m|n)}\bigr)$, it satisfies the following relations with the series $\mathfrak{h}_i(z)$ for $i=1,\dots,m+n+1$, and \smash{$\widetilde{X}_j^{\pm}(z)$} for $j=1,\dots,m+n$:
\begin{gather*}
 \mathfrak{h}_i^{\pm}(z)\mathfrak{h}_j^{\pm}(w)=\mathfrak{h}_j^{\pm}(w)\mathfrak{h}_i^{\pm}(z), \\
 \frac{qz_{\pm}-q^{-1}w_{\mp}}{q^{-1}z_{\pm}-qw_{\mp}}\mathfrak{h}_i^{\pm}(z)\mathfrak{h}_i^{\mp}(w)
 =\frac{qz_{\mp}-q^{-1}w_{\pm}}{q^{-1}z_{\mp}-qw_{\pm}}\mathfrak{h}_i^{\mp}(w)\mathfrak{h}_i^{\pm}(z)\qquad
 \text{for} \quad i\leq n,\\
 \mathfrak{h}_i^{\pm}(z)\mathfrak{h}_i^{\mp}(w)=\mathfrak{h}_i^{\mp}(w)\mathfrak{h}_i^{\pm}(z)\qquad
 \text{for} \quad n+1\leq i\leq n+m,\\
 \frac{z_{\pm}-w_{\mp}}{q^{-1}z_{\pm}-qw_{\mp}}\mathfrak{h}_i^{\pm}(z)\mathfrak{h}_j^{\mp}(w)=\frac{z_{\mp}-w_{\pm}}{q^{-1}z_{\mp}-qw_{\pm}} \mathfrak{h}_j^{\mp}(w)\mathfrak{h}_i^{\pm}(z) \qquad
 \text{for} \quad i<j,
\\
 \frac{\bigl(qz_{\pm}-q^{-1}w_{\mp}\bigr)(z_{\pm}-qw_{\mp})}{\bigl(q^{-1}z_{\pm}-qw_{\mp}\bigr)
 (qz_{\pm}-w_{\mp})} \mathfrak{h}_{m+n+1}^{\pm}(z)\mathfrak{h}_{m+n+1}^{\mp}(w)\\
 \qquad =\frac{\bigl(qz_{\mp}-q^{-1}w_{\pm}\bigr)(z_{\mp}-qw_{\pm})}{\bigl(q^{-1}z_{\pm}-qw_{\pm}\bigr)
 (qz_{\mp}-w_{\pm})}\mathfrak{h}_{m+n+1}^{\mp}(w)\mathfrak{h}_{m+n+1}^{\pm}(z).
\end{gather*}
The relations involving $\mathfrak{h}_i^{\pm}(z)$ and $\widetilde{X}_j^{\pm}(w)$ are
\begin{gather*}
 \mathfrak{h}_i^{\pm}(z)\widetilde{X}_j^-(w)=\frac{z_{\mp}-w}{q^{-(\varepsilon_i,\alpha_j)}z_{\mp}-q^{(\varepsilon_i, \alpha_j)}w}\widetilde{X}_j^-(w)\mathfrak{h}_i^{\pm}(z)\qquad
 \text{for} \quad i\neq m+n+1,\\
 \mathfrak{h}_i^{\pm}(z)\widetilde{X}_j^+(w)=\frac{q^{-(\varepsilon_i,\alpha_j)}z_{\pm}-q^{(\varepsilon_i,\alpha_j)}w}{z_{\pm}-w}\widetilde{X}_j^+(w) \mathfrak{h}_i^{\pm}(z)\qquad
 \text{for} \quad i\neq m+n+1,\\
 \mathfrak{h}_{m+n+1}^{\pm}(z)\widetilde{X}_{m+n}^-(w)=\frac{\bigl(q^{-1}z_{\mp}-w\bigr)(z_{\mp}-w)}{\bigl(z_{\mp}-q^{-1}w\bigr)\bigl(q^{-1}z_{\mp}-qw\bigr)}
\widetilde{X}_{m+n}^-(w)\mathfrak{h}_{m+n+1}^{\pm}(z),\\
 \mathfrak{h}_{m+n+1}^{\pm}(z)\widetilde{X}_{m+n}^+(w)=\frac{\bigl(z_{\pm}-q^{-1}w\bigr)\bigl(q^{-1}z_{\pm}-qw\bigr)}
 {\bigl(q^{-1}z_{\pm}-w\bigr)(z_{\pm}-w)}\widetilde{X}_{m+n}^+(w)\mathfrak{h}_{m+n+1}^{\pm}(z),
\end{gather*}
and
\begin{gather*}
\mathfrak{h}_{m+n+1}^{\pm}(z)\widetilde{X}_i^{\epsilon}(w)=\widetilde{X}_i^{\epsilon}(w)\mathfrak{h}_{m+n+1}^{\pm}(z)\qquad
 \text{for} \quad 1\leq i\leq m+n-1,\quad \epsilon=\pm .
\end{gather*}
The relations involving $\widetilde{X}_j^{\pm}(z)$ are
\begin{gather*}
 \bigl(z-wq^{\pm (\alpha_i, \alpha_j)}\bigr)\widetilde{X}_i^{\pm}\bigl(zq^i\bigr)\widetilde{X}_j^{\pm}\bigl(wq^j\bigr)=\bigl(zq^{\pm (\alpha_i, \alpha_j)}-w\bigr)\widetilde{X}_j^{\pm}\bigl(wq^j\bigr)\widetilde{X}_i^{\pm}\bigl(zq^i\bigr)
\end{gather*}
for $1\leq i,j \leq m+n$ and $(i,j)\neq(n,n)$, together with
 \[
 \widetilde{X}_{n}^{\pm}(z)\widetilde{X}_n^{\pm}(w)=-\widetilde{X}_n^{\pm}(w)\widetilde{X}_n^{\pm}(z),
\]
and
\begin{gather*}
 \bigl[\widetilde{X}_i^{+}(z), \widetilde{X}_j^{-}(w)\bigr]\\
 \qquad=\delta_{ij}\bigl(q_i-q_i^{-1}\bigr)\left(\delta\left(\frac{zq^{-c}}{w}\right)\mathfrak{h}_{i+1}^+(w_+)
 \mathfrak{h}_i^+(w_+)^{-1}-\delta\left(\frac{zq^{c}}{w}\right)\mathfrak{h}_{i+1}^-(z_+)\mathfrak{h}_i^-(z_+)^{-1}\right)
\end{gather*}
for $1\leq i,j\leq m+n$.
The Serre relations are for $\epsilon=\pm$,
\begin{gather*}
\operatorname{Sym}_{z_1,z_2}\bigl[\!\!\bigl[\widetilde{X}_i^{\pm}(z_1),\bigl[\!\!\bigl[\widetilde{X}_i^{\pm}(z_2), \widetilde{X}_i^{\pm}(w)\bigr]\!\!\bigr] \bigr]\!\!\bigr]=0\qquad\textrm{if}\quad i\neq n,m+n, \\
\operatorname{Sym}_{z_1,z_2,z_3}\bigl[\!\!\bigl[\widetilde{X}_{m+n}^{\pm}(z_1), \bigl[\!\!\bigl[\widetilde{X}_{m+n}^{\pm}(z_2),\bigl[\!\!\bigl[\widetilde{X}_{m+n}^{\pm}(z_3), \widetilde{X}_{m+n-1}^{\pm}(w)\bigr]\!\!\bigr] \bigr]\!\!\bigr] \bigr]\!\!\bigr]=0, \\
\operatorname{Sym}_{z_1,z_2}\bigl[ \bigl[\!\!\bigl[ \bigl[\!\!\bigl[\widetilde{X}_{n-1}^{\pm}(z_1), \widetilde{X}_{n}^{\pm}(w_1)\bigr]\!\!\bigr], \widetilde{X}_{n+1}^{\pm}(z_2)\bigr]\!\!\bigr], \widetilde{X}_{n}^{\pm}(w_2)\bigr]=0
\qquad\textrm{if}\quad n>1.
\end{gather*}

$(2)$ In the super $R$-matrix algebra $U\bigl(R^{(m|n)}\bigr)$, it satisfies the following relations with the series~$h_i(z)$ for $i=1,\dots,m+n+1$, and $X_j^{\pm}(z)$ for $j=1,\dots,m+n$:
\begin{gather*}
 h_i^{\pm}(z)h_j^{\pm}(w)=h_j^{\pm}(w)h_i^{\pm}(z), \\
 g\bigl((zq^c/w)^{\pm1}\bigr)\frac{qz_{\pm}-q^{-1}w_{\mp}}{q^{-1}z_{\pm}-qw_{\mp}}h_i^{\pm}(z)h_i^{\mp}(w)
 =g\bigl((zq^{-c}/w)^{\pm1}\bigr)\frac{qz_{\mp}-q^{-1}w_{\pm}}{q^{-1}z_{\mp}-qw_{\pm}}h_i^{\mp}(w)h_i^{\pm}(z)\\
 \text{for} \quad
 i\leq n,\\
 g\bigl((zq^{c}/w)^{\pm1}\bigr) h_i^{\pm}(z)h_i^{\mp}(w)=g\bigl((zq^{-c}/w)^{\pm1}\bigr)h_i^{\mp}(w)h_i^{\pm}(z), \qquad\textrm{for}\quad n+1\leq i\leq n+m,\\
 g\bigl((zq^{c}/w)^{\pm1}\bigr)\frac{z_{\pm}-w_{\mp}}{q^{-1}z_{\pm}-qw_{\mp}}h_i^{\pm}(z)h_j^{\mp}(w)=g\bigl((zq^{-c}/w)^{\pm1}\bigr)
 \frac{z_{\mp}-w_{\pm}}{q^{-1}z_{\mp}-qw_{\pm}}h_j^{\mp}(w)h_i^{\pm}(z)\\
 \text{for} \quad
 i<j,
\\
 g\bigl((zq^{c}/w)^{\pm1}\bigr)\frac{\bigl(qz_{\pm}-q^{-1}w_{\mp}\bigr)(q^{-1/2}z_{\pm}-q^{1/2}w_{\mp})}{\bigl(q^{-1}z_{\pm}-qw_{\mp}\bigr)
 (q^{1/2}z_{\pm}-q^{-1/2}w_{\mp})}h_{m+n+1}^{\pm}(z)h_{m+n+1}^{\mp}(w)\\
 \qquad=g\bigl((zq^{-c}/w)^{\pm1}\bigr)\frac{\bigl(qz_{\mp}-q^{-1}w_{\pm}\bigr)(q^{-1/2}z_{\mp}-q^{1/2}w_{\pm})}{\bigl(q^{-1}z_{\pm}-qw_{\pm}\bigr)
 (q^{1/2}z_{\mp}-q^{-1/2}w_{\pm})}h_{m+n+1}^{\mp}(w)h_{m+n+1}^{\pm}(z),
\end{gather*}
and the remaining relations as same as $U\bigl(\widetilde{R}^{(m|n)}\bigr)$ by replacing the generators $\mathfrak{h}_i(z)$, $\widetilde{X}_j^{\pm}(z)$ as~$h_i(z)$,~$X_j^{\pm}(z)$.
\end{Theorem}

\begin{proof} Here, we only need to prove the Serre relations. Since the Serre relations in both \smash{$U\bigl(\widetilde{R}^{(m|n)}\bigr)$} and $U\bigl(R^{(m|n)}\bigr) $ have the same forms, we will focus only on the relations for the superalgebra \smash{$U\bigl(\widetilde{R}^{(m|n)}\bigr)$}. Set
\begin{align*}
 &\widetilde{X}_i^{\pm}(zq^{-\nu_i})=\bigl(q_i-q_i^{-1}\bigr)\dot{x}_i^{\pm}(z)=\bigl(q_i-q_i^{-1}\bigr)\sum_{p\in\mathbb{Z}}\dot{x}_{i,p}^{\pm}z^{p},\\
 &\mathfrak{h}_{i+1}^{\pm}(zq^{-\nu_i})\mathfrak{h}_{i}^{\pm}(zq^{-\nu_i})^{-1}=\Phi^{\pm}_{i}(z)=K_i^{\pm 1}\exp\bigl(q_i-q_i^{-1}\bigr)\sum\limits_{\ell\geq 1}\dot{a}_{i,\pm \ell}z^{\pm \ell}.
\end{align*}
Therefore, the Serre relations take the forms
\begin{align*}
 &\operatorname{Sym}_{r_1,r_2}[\![\dot{x}_{i,r_1}^{\epsilon},[\![\dot{x}_{i,r_2}^{\epsilon}, \dot{x}_{i\pm 1,s}^{\epsilon}]\!] ]\!]=0\qquad\textrm{if}\quad i\neq n,m+n, \\
&\operatorname{Sym}_{r_1,r_2,r_3}[\![\dot{x}_{m+n,r_1}^{\epsilon}, [\![\dot{x}_{m+n,r_2}^{\epsilon},[\![\dot{x}_{m+n,r_3}^{\epsilon}, \dot{x}_{m+n-1,s_1}^{\epsilon}]\!] ]\!] ]\!]=0, \\
&\operatorname{Sym}_{s_1,s_2}[ [\![ [\![\dot{x}_{n-1,r_1}^{\epsilon}, \dot{x}_{n,s_1}^{\epsilon}]\!], \dot{x}_{n+1,r_2}^{\epsilon}]\!], \dot{x}_{n,s_2}^{\epsilon}]=0
\qquad\textrm{if}\quad n>1,
\end{align*}
where $r_i$ and $s_j$ are arbitrary integers. Furthermore, by the defining relations,
 we can derive that
\begin{gather*}
 \bigl[\dot{a}_{i,\pm s}, \dot{x}_{j,k}^{\pm}\bigr]=\pm\frac{[sA_{ij}]_{i}}{s}q^{\mp|s|c/2}\dot{x}_{j,s+k}^{\pm}.
\end{gather*}
Now, as the original methods of \cite{SZLO}, the Serre relations can be proved by an induction on the integers $r_i$ and $s_j$, and more details please see \cite[Section 5.1]{XWLHZ} and \cite[Section 4]{LYZ}.
\end{proof}

\section{Isomorphism theorem }\label{sec6}

In this section, the superalgebra $\mathcal{A}_q$ is defined using the Drinfeld generators obtained from $U(R)$. The statement suggests the existence of a homomorphism, denoted as $AR$, from $\mathcal{A}_q$ to $U(R)$. Additionally, the quantum affine superalgebra $\mathcal{U}_q(\mathfrak{\hat{g}})$ can be viewed as a quotient algebra of $\mathcal{A}_q$. This implies that $\mathcal{U}_q(\mathfrak{\hat{g}})$ can be embedded into $\mathcal{A}_q$. By leveraging the $L$-operators of $\mathcal{U}_q(\mathfrak{\hat{g}})$ and the vector representation from Section \ref{sec4.1}, an inverse map of the given homomorphism $AR$ can be established. This suggests that $\mathcal{A}_q$ is actually isomorphic to $U(R)$.

\subsection[The superalgebra A\_q]{The superalgebra $\boldsymbol{\mathcal{A}_q}$}

\begin{Definition}
Let $\mathcal{A}_q$ be the superalgebra generated by the generators $h_i(z)$ ($i=1,\dots,\allowbreak m+n+1$), and $X_j^{\pm}(z)$ ($j=1,\dots,m+n$) with the same relations in $U\bigl(R^{(m|n)}\bigr) $.
\end{Definition}

Combining the generators $x_{i,p}^{\pm}$ in $\mathcal{U}_q(\mathfrak{\hat{g}})$ with the formal power series
\begin{gather*}
 x_i^{\pm}(z)=\sum_{p\in\mathbb{Z}}x_{i,p}^{\pm}z^{p},
\end{gather*}
and defining
\begin{gather}\label{homo2}
	\Phi^{\pm}_i(z)=\sum\limits_{m=0}\Phi_{i,\pm r}^{\pm}z^{\pm r}=K_i^{\pm 1} \exp\biggl(\bigl(q_i-q_i^{-1}\bigr)\sum\limits_{\ell\geq 1}a_{i,\pm \ell}z^{\pm \ell}\biggr).
\end{gather}

\begin{Proposition}\label{Apr0} The coefficients of
 $c^{\pm}(z)$ as the form in Proposition {\rm\ref{pr5}} are central elements of algebra $\mathcal{A}_q$. Moreover,
 the map such that
\begin{align*}
 &q^{c/2}\mapsto q^{c/2},\qquad
 (-1)^{[\alpha_i]} x_i^{\pm}(z)\mapsto\bigl(q_i-q_i^{-1}\bigr)^{-1}X_i^{\pm}(zq^{-\nu_i}),\qquad 1\leq i\leq m+n,\\
 &\Phi^{\pm}_{i}(z)\mapsto h_{i+1}^{\pm}(zq^{-\nu_i})h_{i}^{\pm}(zq^{-\nu_i})^{-1},\qquad 1\leq i\leq m+n,
\end{align*}
define an embedding $\tau\colon \mathcal{U}_q(\mathfrak{\hat{g}})\rightarrow\mathcal{A}_q$.
\end{Proposition}

\begin{proof}
 The central elements and the homomorphism are obvious. To show the injectivity, we will construct a homomorphism $\rho\colon \mathcal{A}_q\mapsto \mathcal{U}_q(\mathfrak{\hat{g}})$ such that $\rho\circ\tau$ is the identity homomorphism on $\mathcal{U}_q(\mathfrak{\hat{g}})$.

We first construct a map $\rho_1\colon \mathcal{A}_q\mapsto \mathcal{U}_q(\mathfrak{\hat{g}})$ such that
\begin{align*}
 & X_i^{\pm}(z)\mapsto (-1)^{[\alpha_{i}]}\bigl(q_i-q_i^{-1}\bigr)x_i^{\pm}(zq^{\nu_i}),\qquad 1\leq i\leq m+n,\qquad h_i^{\pm}(z)\mapsto \Gamma_i^{\pm}(z).
\end{align*}
To explain the element $\Gamma_i^{\pm}(z)$, there exist power series $\kappa^{\pm}(z)$ with coefficients in the center of~$\mathcal{A}_q$ such that $\kappa^{\pm}(z)\kappa^{\pm}(z\zeta)=c^{\pm}(z)$, where
\begin{gather}\label{kappa}
 \kappa^{\pm}(z)=\prod\limits_{p=0}^{\infty}c^{\pm}\bigl(z\zeta^{-2p-1}\bigr)c^{\pm}\bigl(z\zeta^{-2p-2}\bigr)^{-1}.
\end{gather}
Then we have the endomorphism $\rho_2\colon \mathcal{A}_q\mapsto\mathcal{A}_q$ such that $X_i^{\pm}(z)\mapsto X_i^{\pm}(z)$, $h_i^{\pm}(z)\mapsto \kappa^{\pm}(z)h_i^{\pm}(z)$. So that
$h_i^{\pm}(z)\kappa^{\pm}(z)h_i^{\pm}(z\zeta)\kappa^{\pm}(z\zeta)=h_i^{\pm}(z)h_i^{\pm}(z\zeta)c^{\pm}(z)$.
Hence, denote that
\begin{gather*}
 \Gamma_i^{\pm}(z)\Gamma_i^{\pm}(z\zeta)=\prod\limits_{k=1}^{n+m}\Phi^{\pm}_{k}(z\zeta q^{-\nu_k})^{-1}\prod\limits_{k=1}^{i-1}\Phi^{\pm}_{k}(z\zeta q^{\nu_k})\prod\limits_{k=i}^{n+m}\Phi^{\pm}_{k}(z q^{\nu_k})^{-1},\qquad i=1,\dots,m+n,
\end{gather*}
and
\begin{gather*}
 \Gamma_{m+n+1}^{\pm}(z)\Gamma_{m+n+1}^{\pm}(z\zeta)=\prod\limits_{k=1}^{n+m}\Phi^{\pm}_{k}(z\zeta q^{-\nu_k})^{-1}
 \times\prod\limits_{k=1}^{n+m}\Phi^{\pm}_{k}(z\zeta q^{\nu_k}).
\end{gather*}
Set $\widetilde{\Phi}^{\pm}_{j}(z)=k^{\mp}_j\Phi^{\pm}_{j}(z)$, then
\begin{align*}
 \Gamma_i^{\pm}(z)={}&\prod\limits_{p=0}^{\infty}\prod\limits_{k=1}^{n+m}\widetilde{\Phi}^{\pm}_{k}\bigl(z\zeta^{-2p} q^{-\nu_k}\bigr)^{-1}\widetilde{\Phi}^{\pm}_{k}\bigl(z\zeta^{-2p-1} q^{-\nu_k}\bigr)\widetilde{\Phi}^{\pm}_{k}\bigl(z\zeta^{-2p-1} q^{\nu_k}\bigr)^{-1}\widetilde{\Phi}^{\pm}_{k}\bigl(z\zeta^{-2p-2} q^{\nu_k}\bigr)\\
 &\times\prod\limits_{k=1}^{i-1}\widetilde{\Phi}^{\pm}_{k}(z q^{\nu_k})\prod\limits_{k=i}^{n}k_i
\end{align*}
for $i=1,\dots,m+n$, and
\begin{align*}
 \Gamma_{m+n+1}^{\pm}(z)={}&\prod\limits_{p=0}^{\infty}\prod\limits_{k=1}^{n+m}\widetilde{\Phi}^{\pm}_{k}\bigl(z\zeta^{-2p} q^{-\nu_k}\bigr)^{-1}\widetilde{\Phi}^{\pm}_{k}\bigl(z\zeta^{-2p-1} q^{-\nu_k}\bigr)\widetilde{\Phi}^{\pm}_{k}\bigl(z\zeta^{-2p-1} q^{\nu_k}\bigr)^{-1}\\
 &\times\widetilde{\Phi}^{\pm}_{k}\bigl(z\zeta^{-2p-2} q^{\nu_k}\bigr)
 \prod\limits_{k=1}^{n+m}\widetilde{\Phi}^{\pm}_{k}(z q^{\nu_k}).
\end{align*}

Under the constructions above, it is straightforward to show that $\rho_1$ is a homomorphism. In fact,
we demonstrate that the relations between the generators $h_i^{\pm}(z)$ and $X_i^{\pm}(w)$ are preserved under the map $\rho_1$, as the remaining relations among all the generators can be verified in a similar manner.
In algebra $\mathcal{U}_q(\mathfrak{\hat{g}})$, we have
\begin{align*}
 &\Phi^{+}_i(z)x_j^{\pm}(w)=\left(\frac{z_{\pm}q^{(\alpha_i,\alpha_j)}-w}{z_{\pm}-wq^{(\alpha_i,\alpha_j)}}\right)^{\pm1}x_j^{\pm}(w)\Phi^{+}_i(z),\\
 & \Phi^{-}_i(z)x_j^{\pm}(w)=\left(\frac{w_{\pm}q^{(\alpha_i,\alpha_j)}-z}{w_{\pm}-zq^{(\alpha_i,\alpha_j)}}\right)^{\mp1}x_j^{\pm}(w)\Phi^{-}_i(z).
\end{align*}A direct calculation gives that
\begin{align*}
 &\Gamma_i^{\pm}(z)x_j^{+}(w)\bigl(\Gamma_i^{\mp}(z)\bigr)^{-1}=\frac{z_{\pm}q^{-(\varepsilon_i,\alpha_j)}-wq^{(\varepsilon_i,\alpha_j)}}{z_{\pm}-w}
 x_j^{+}(w),\\
 &\Gamma_i^{\pm}(z)x_j^{-}(w)\bigl(\Gamma_i^{\mp}(z)\bigr)^{-1}=\frac{z_{\mp}-w}{z_{\mp}q^{-(\varepsilon_i,\alpha_j)}-wq^{(\varepsilon_i,\alpha_j)}}
 x_j^{-}(w)
\end{align*}for $i\neq n+m+1$, and
\begin{align*}
 &\Gamma_{m+n+1}^{\pm}(z)x_{m+n}^+(w)\bigl(\Gamma_{m+n+1}^{\mp}(z)\bigr)^{-1}=\frac{\bigl(z_{\pm}-q^{-1}w\bigr)\bigl(q^{-1}z_{\pm}-qw\bigr)}{\bigl(q^{-1}z_{\pm}-w\bigr)(z_{\pm}-w)}x_{m+n}^+(w),\\
 &\Gamma_{m+n+1}^{\pm}(z)x_{m+n}^-(w)\bigl(\Gamma_{m+n+1}^{\mp}(z)\bigr)^{-1}=\frac{\bigl(q^{-1}z_{\mp}-w\bigr)(z_{\mp}-w)}{\bigl(z_{\mp}-q^{-1}w\bigr)\bigl(q^{-1}z_{\mp}-qw\bigr)}x_{m+n}^-(w).
\end{align*}This implies that
\begin{gather*}
 \rho_1\bigl(h_i^{\pm}(z)X_i^{\epsilon}(w)\bigl(h_i^{\pm}(z)\bigr)^{-1}\bigr)=\rho_1\bigl(h_i^{\pm}(z)\bigr)\cdot\rho_1(X_i^{\epsilon}(w))\cdot\rho_1\bigl(\bigl(h_i^{\pm}(z)\bigr)^{-1}\bigr),
\end{gather*}where $\epsilon=\pm1$.

Set $\rho=\rho_1\circ\rho_2$, it is easy to see that the map $\rho\circ\tau$ is the identity map on $\mathcal{U}_q(\mathfrak{\hat{g}})$ by the formulas $c^{\pm}(z)$ in Proposition \ref{pr5}.
\end{proof}

\begin{Proposition}\label{Apr1}Between the algebras $\mathcal{U}_q(\mathfrak{\hat{g}})$ and $\mathcal{A}_q$, we have the tensor product decomposition
\[
 \mathcal{U}_q(\mathfrak{\hat{g}})\otimes_{\mathbb{C}(q^{1/2})}\mathfrak{C}=\mathcal{A}_q,
\]
where $\mathfrak{C}$ is the subalgebra of $\mathcal{A}_q$ generated by the coefficients of the series $c^{\pm}(z)$.
\end{Proposition}

\begin{proof}In the proof of Proposition \ref{Apr0}, we can define a surjective homomorphism $\tilde{\rho}\colon\mathcal{A}_q\rightarrow \mathcal{U}_q(\mathfrak{\hat{g}})\otimes_{\mathbb{C}\bigl(q^{1/2}\bigr)}\mathfrak{C}$ such that
\begin{align*}
 & X_i^{\pm}(z)\mapsto \rho(X_i^{\pm}(z))\otimes1,\qquad 1\leq i\leq m+n,\\
 &h_i^{\pm}(z)\mapsto \rho\bigl(h_i^{\pm}(z)\bigr)\otimes\kappa^{\pm}(z),\qquad 1\leq i\leq m+n.
\end{align*}
On the other hand, there is a map from $\mathcal{U}_q(\mathfrak{\hat{g}})\otimes_{\mathbb{C}(q^{1/2})}\mathfrak{C}$ to
$\mathcal{A}_q$ by
\begin{align*}
 & a\otimes1\mapsto\tau(a),\qquad a\in \mathcal{U}_q(\mathfrak{\hat{g}}),\qquad 1\otimes\kappa^{\pm}(z)\mapsto\kappa^{\pm}(z).
\end{align*}
It is easy to see that the above map is a inverse homomorphism of $\tilde{\rho}$.
\end{proof}

\subsection[Decomposition of universal R-matrix and inverse map]{Decomposition of universal $\boldsymbol{R}$-matrix and inverse map}

First, we recalled some properties of the universal $R$-matrix. Using the same notations of~\cite{SMVN2}, we set
\begin{gather*}
 \exp_q(x)\doteq1+x+\frac{x^2}{(2)_q!}+\dots+\frac{x^n}{(n)_q!}+\dots=\sum_{n\geq0}\frac{x^n}{(n)_q!},\qquad
 (a)_q\doteq\frac{q^a-1}{q-1}.
\end{gather*}
 Consider the $\hbar$-adic settings and let $q=\exp(\hbar)\in\mathbb{C}[[\hbar]]$.
Introduce elements $h_1,\dots,h_{n+m}$ by defining $K_i=\exp(\hbar h_i)$.
Let $\widehat{\Delta}$, $\widehat{\Delta}_+$ and $\widehat{Q}$ be the affine root system, affine positive root system and affine root lattice of $\mathfrak{\hat{g}}$, respectively.
There exists a bilinear map $Q\times \mathbb{Z}\rightarrow\widehat{Q}$ such that~${(\alpha_i,0)\mapsto \alpha_i}$ for $i=1,\ldots,m+n$ and $(-\theta,1)\mapsto \alpha_0$. Let $\underline{\Delta}_+$ denote the reduced root system obtained from the positive root system $\Delta_+$ of $\mathfrak{g}$ by excluding roots $\alpha$ such that $\alpha/2$ are odd roots.
Then the reduced positive root system with multiplicity of $\mathfrak{\hat{g}}$, denoted by $\widetilde{\underline{\Delta}}_+$, is given by
\begin{gather*}
\widetilde{\underline{\Delta}}_+=\widetilde{\underline{\Delta}}^{\text{re}}_{>}\cup \widetilde{\Delta}^{\text{im}}\cup \widetilde{\underline{\Delta}}^{\text{re}}_{<},
\end{gather*}
where \smash{$\widetilde{\underline{\Delta}}^{\text{re}}_{>}=\{ (\alpha,k) \mid \alpha\in \underline{\Delta}_+, k\geqslant 0 \}$}, \smash{$\widetilde{\Delta}^{\text{im}}=\{ (0,k) \mid k>0 \}\times I$}, $\smash{\widetilde{\Delta}^{\text{re}}_{<}=\{ (-\alpha,k)} \mid \alpha\in \underline{\Delta}_+, \allowbreak {k\geqslant 1} \}$.
Establish a fixed ordering on \smash{$\widetilde{\underline{\Delta}}_+$} (see \cite[Section 3]{SMVN} and
\cite {XWLHZ}), and for any \smash{$\alpha\in\widetilde{\underline{\Delta}}_+$}, set \smash{$\hat{q}_{\alpha}=(-1)^{[\alpha]}q^{(\alpha,\alpha)}$}.
\begin{Proposition}[{\cite[Theorem 4.1]{SMVN}}]\label{URD}
The universal $R$-matrix $\mathfrak{R}$ $($up to a multiplicative con\-stant$)$ of $\widetilde{U}_q(\mathfrak{\hat{g}})$ has a unique solution, and takes the form
\begin{gather*}
 \mathfrak{R}=\prod\limits_{\alpha\in\widetilde{\underline{\Delta}}_+ }^{\rightarrow}\mathfrak{R}_{\alpha}\cdot \mathcal{K}=\mathfrak{R}^{>0}\mathfrak{R}^{0}\mathfrak{R}^{<0}\cdot \mathcal{K},
\end{gather*}
where
\begin{gather*}
 \mathfrak{R}^{>0}=\prod\limits_{\alpha\in\widetilde{\underline{\Delta}}_+\setminus\widetilde{\Delta}^{\text{im}} }\exp_{\hat{q}_{\alpha}^{-1}}\bigl((-1)^{[\alpha]}\bigl(q-q^{-1}\bigr)c(\alpha)^{-1}\mathfrak{E}_{\alpha}\otimes \mathfrak{F}_{\alpha}\bigr), \\
 \mathfrak{R}^{<0}=\prod\limits_{\alpha\in\widetilde{\underline{\Delta}}_+\setminus\widetilde{\Delta}^{\text{im}} }\exp_{\hat{q}_{\alpha}^{-1}}\bigl((-1)^{[\alpha]}\bigl(q-q^{-1}\bigr)c(\alpha)^{-1}\mathfrak{E}_{-\alpha}\otimes \mathfrak{F}_{-\alpha}\bigr),\\
 \mathfrak{R}^{0}=\exp\Biggl(\sum\limits_{k>0}\sum\limits_{i,j=1}^{m+n}(-1)^{[k\delta]}\bigl(q-q^{-1}\bigr)c_{ij}(k)\mathfrak{E}_{k\delta^{(i)}}\otimes \mathfrak{F}_{k\delta^{(j)}}\Biggr),
\end{gather*}
where $\mathcal{K}=\mathcal{T}q^{-c\otimes d-d\otimes c}$, $\mathcal{T}=\exp\bigl(\hbar {A_{ij}^{\operatorname{sym}}}^{-1}h_i\otimes h_j\bigr)$ for $i,j\in\{1,\dots,n+m\}$, and $\bigl(A_{ij}^{\operatorname{sym}}\bigr)_{i,j=1}^{n+m}=A^{\operatorname{sym}}=CA$ is the symmetric Cartan matrix and $C=\operatorname{diag}(1, 1,\dots,1/2)$.
 $c(\alpha)$ be the coefficients determined by \smash{$[\mathfrak{E}_{\alpha}, \mathfrak{F}_{\alpha}]=c(\alpha)\frac{K_{\delta}^k-K_{\delta}^{-k}}{q-q^{-1}}$}, where
$(c_{ij}(k))$ is an inverse to the matrix $(\tilde{c}_{ij}(k))$ with the elements determined by
\begin{gather*}
[\mathfrak{E}_{k\delta^{(i)}}, \mathfrak{F}_{k\delta^{(j)}}]=\tilde{c}_{ij}(k)\frac{K_{\delta}^k-K_{\delta}^{-k}}{q-q^{-1}}.
\end{gather*}
\end{Proposition}

From Section \ref{sec3}, we have
$\mathfrak{R}(z)=\mathfrak{R}^{>0}(z)\mathfrak{R}^{0}(z)\mathfrak{R}^{<0}(z)$,
where
\begin{gather*}
 \mathfrak{R}^{>0}(z)=\prod\limits_{\alpha\in\underline{\Delta}_{+} }\prod\limits_{k\geq0 }\exp_{\hat{q}_{\alpha}^{-1}}\bigl((-1)^{[\alpha]}(q_{\alpha}-q_{\alpha}^{-1})z^kc(\alpha+k\delta)^{-1}\mathfrak{E}_{\alpha+k\delta}\otimes \mathfrak{F}_{\alpha+k\delta}\bigr), \\
 \mathfrak{R}^{<0}(z)=\mathcal{T}^{-1}\prod\limits_{\alpha\in\underline{\Delta}_{+} }\prod\limits_{k>0 }\exp_{\hat{q}_{\alpha}^{-1}}\bigl((-1)^{[\alpha]}(q_{\alpha}-q_{\alpha}^{-1})z^kc(-\alpha+k\delta)^{-1}\mathfrak{E}_{-\alpha+k\delta}\otimes \mathfrak{F}_{-\alpha+k\delta}\bigr)\mathcal{T},
\\
 \mathfrak{R}^{0}(z)=\exp\Biggl(\sum\limits_{k>0}\sum\limits_{i,j=1}^{n+m}\frac{\bigl(q_i-q_i^{-1}\bigr)\bigl(q_j-q_j^{-1}\bigr)}{q-q^{-1}}
 \frac{k}{[k]_q}\bigl(A_{ij}^{\operatorname{sym}}\bigl(q^k\bigr)\bigr)^{-1}
 z^kq^{kc/2}
 a_{i,k}\otimes a_{j,-k}q^{-kc/2}\Biggr)\mathcal{T}.
\end{gather*}
Here \smash{$A^{\operatorname{sym}}\bigl(q^k\bigr)=\bigl(A_{ij}^{\operatorname{sym}}\bigl(q^k\bigr)\bigr)_{i,j=1}^{n+m}=\bigl(\bigl[A_{ij}^{\operatorname{sym}}\bigr]_{q^k}\bigr)_{i,j=1}^{n+m}$} be the $q$-deformed matrix of the symmetric Cartan matrix.

Furthermore, the inverses of $A^{\operatorname{sym}}$ and $A^{\operatorname{sym}}(q)$ for $\mathfrak{g}$ are ($i\geq j$)
\begin{equation*}
 \bigl(A_{ij}^{\operatorname{sym}}\bigr)^{-1}=
 \begin{cases}
 -j,& 1\leq j\leq n,\\
 j-2n,& n< j\leq m+n.
 \end{cases}\qquad \bigl(A_{ij}^{\operatorname{sym}}(q)\bigr)^{-1}=\frac{A_{ij}^*(q)}{\operatorname{det}\bigl(A^{\operatorname{sym}}(q)\bigr)},
\end{equation*}
where
\begin{equation*}
 \operatorname{det}\bigl(A^{\operatorname{sym}}(q)\bigr)=(-1)^n([m-n]_q-[m-n-1]_q),
\end{equation*}
and
\begin{equation*}
 A_{ij}^*(q)=
 \begin{cases}
 (-1)^{n+1}[j]_q, & 1\leq j<n,\quad i=m+n,\\
 (-1)^n[j-2n]_q, & n\leq j\leq m+n,\quad i=m+n,\\
 (-1)^{n+1}[j]_q([m-n+i]_q-[m-n+i-1]_q), & 1\leq j\leq i<n,\\
 (-1)^n[j-2n]_q([m+n-i]_q-[m+n-i-1]_q), & n\leq j\leq i< m+n,\\
 (-1)^{n+1}[j]_q([m+n-i]_q-[m+n-i-1]_q), & 1\leq j<n\leq i< m+n.
 \end{cases}
\end{equation*}

For the $L$-operators $\mathfrak{L}^{\pm}(z)$ as defined in equation~\eqref{L-operators}, let us establish the following notations:
\begin{align*}
 &\dot{L}^+(z)=\mathfrak{L}^+(z)\kappa^{+}(z), \qquad
 \dot{L}^-(z)=\mathfrak{L}^-(z)\kappa^{-}(z).
\end{align*}By the defining relation of $\kappa^{\pm}(z)$ (see \eqref{kappa}), the coefficients of the series in $z^{\pm1}$ belong to $\mathcal{A}_q$. Therefore, by Proposition \ref{U(ENDV)}, we have
\begin{align*}
 &R(z/w)\dot{L}_1^{\pm}(z)\dot{L}_2^{\pm}(w)=\dot{L}_2^{\pm}(w)\dot{L}_1^{\pm}(z) R(z/w),\\
 &R(z_+/w_-)\dot{L}_1^{+}(z)\dot{L}_2^{-}(w)=\dot{L}_2^{-}(w)\dot{L}_1^{+}(z) R(z_-/w_+).
\end{align*}

\begin{Proposition} The map defined by
\smash{$ RA\colon L^{\pm}(z) \mapsto \dot{L}^{\pm}(z)
$}
establishes a homomorphism from the superalgebra $U(R)$ to $\mathcal{A}_q$.
\end{Proposition}
\begin{proof}
This is straightforward.
\end{proof}

Denote the matrices
\begin{align*}
 &\dot{F}^+(z)=(1\otimes\pi)\mathfrak{R}^{>0}\bigl(zq^{-c/2}\bigr),\qquad \dot{E}^+(z)=(1\otimes\pi)\mathfrak{R}^{<0}\bigl(zq^{-c/2}\bigr), \\
 &\dot{F}^-(z)=(1\otimes\pi)\mathfrak{R}^{>0}_{21}\bigl(\bigl(zq^{-c/2}\bigr)^{-1}\bigr)^{-1},\qquad \dot{E}^-(z)=(1\otimes\pi)\mathfrak{R}^{<0}_{21}\bigl(\bigl(zq^{-c/2}\bigr)^{-1}\bigr)^{-1}, \\
 & \dot{H}^+(z)=(1\otimes\pi)\mathfrak{R}^{0}\bigl(z_q^{-c/2}\bigr)\kappa^{+}(z),\qquad
 \dot{H}^-(z)=(1\otimes\pi)\mathfrak{R}_{21}^{0}\bigl(\bigl(z_q^{-c/2}\bigr)^{-1}\bigr)^{-1}\kappa^{-}(z)^{-1}.
\end{align*}
For the Drinfeld generators $x_{i,k}^{\pm}$ of $\mathcal{U}_q(\mathfrak{\hat{g}})$, let
\begin{align*}
 & x_i^-(z)^{\geq0} =\sum\limits_{k\geq0}x_{i,k}^{-}z^k,\qquad x_i^+(z)^{>0} =\sum\limits_{k>0}x_{i,k}^{+}z^k,\qquad x_i^-(z)^{<0} =\sum\limits_{k>0}x_{i,-k}^{-}z^{-k},\\
 & x_i^+(z)^{\leq0} =\sum\limits_{k\geq0}x_{i,-k}^{+}z^{-k}.
\end{align*}
Then set for $1\leq i\leq m+n-1$,
\begin{align*}
 & f_i^+(z)=(-1)^{[\alpha_{i}]}\bigl(q_i-q_i^{-1}\bigr)x_i^+(z_-q^{\nu_i})^{>0},\qquad e_i^+(z)=(-1)^{[\alpha_{i}]}\bigl(q_i-q_i^{-1}\bigr) x_i^-(z_+q^{\nu_i})^{\geq0},\\
 & f_i^-(z)=(-1)^{[\alpha_{i}])}\bigl(q_i^{-1}-q_i\bigr)x_i^+(z_+q^{\nu_i})^{\leq0},\qquad e_i^-(z)=(-1)^{[\alpha_{i}]}\bigl(q_i^{-1}-q_i\bigr) x_i^-(z_-q^{\nu_i})^{>0},
\end{align*}
and
\begin{align*}
 & f_{n+m}^+(z)=(-1)^{[\alpha_{n+m}]}\bigl(q_{n+m}-q_{n+m}^{-1}\bigr)[2]^{1/2}_{q_{n+m}}x_{n+m}^+(z_-q^{m-n})^{>0},\\ & f_{n+m}^-(z)=(-1)^{[\alpha_{n+m}]}\bigl(q_{n+m}^{-1}-q_{n+m}\bigr)[2]^{1/2}_{q_{n+m}}x_{n+m}^+(z_+q^{m-n})^{\leq0},\\
 &e_{n+j}^+(z)=(-1)^{[\alpha_{n+m}]}\bigl(q_{n+m}-q_{n+m}^{-1}\bigr)[2]^{1/2}_{q_{n+m}} x_{n+m}^-(z_+q^{m-n})^{\geq0},\\
 &e_{n+m}^-(z)=(-1)^{[\alpha_{n+m}]}\bigl(q_{n+m}^{-1}-q_{n+m}\bigr)[2]^{1/2}_{q_{n+m}} x_{n+m}^-(z_-q^{m-n})^{>0}.
\end{align*}

\begin{Proposition}\label{FHE}
Under the above construction, we have the following decomposition:
\begin{align*}
 \dot{L}^{\pm}(z)={}&\dot{F}^{\pm}(z)\dot{H}^{\pm}(z)\dot{E}^{\pm}(z)=\left(
 \begin{matrix}
 1 & & & & & \\
 f_1^{\pm}(z) & 1 & & & & \\
 & \ddots & \ddots & & & \\
 & & f_{n+m}^{\pm}(z) & 1 & & \\
 & & & -f_{n+m-1}^{\pm}(zq) &1 & \\
 & & * & \ddots& \ddots & \\
 & & & & -f_{1}^{\pm}\bigl(z\zeta q^2\bigr) & 1 \\
 \end{matrix}
 \right)\\
 & \times \dot{H}^{\pm}(z)\times
 \left(
 \begin{matrix}
 1 & e_1^{\pm}(z) & & & & \\
 & \ddots & \ddots & & * & \\
 & & 1 & e_{n+m}^{\pm}(z) & & \\
 & & & 1 & -e_{n+m-1}^{\pm}(zq) & \\
 & & & \ddots & \ddots & \\
 & & & & 1 & -e_{1}^{\pm}\bigl(z\zeta q^2\bigr) \\
 & & & & & 1\\
 \end{matrix}
 \right),
\end{align*}
where
 \begin{align*}
\dot{H}^{\pm}(z)={}&\operatorname{diag}\bigl(h_1^{\pm}(z), \dots, h_{n+m}^{\pm}(z), h_{n+m+1}^{\pm}(z), c^{\pm(1)}(z)h_{n+m}^{\pm}(zq)^{-1} ,\dots,\\
& c^{\pm ( m-1)}h_{n+1}^{\pm}\bigl(zq^{2m-1}\bigr)^{-1},
c^{\pm( m)}h_{n}^{\pm}\bigl(zq^{2m-3}\bigr)^{-1},\dots,
c^{\pm(n+m)}(z)h_{1}^{\pm}(z\zeta)^{-1}\bigr).
\end{align*}
\end{Proposition}
\begin{proof} We only consider the decomposition of $\dot{L}^+(z)$ since $\dot{L}^-(z)$ is similar.
By the isomorphism relations in Theorem \ref{DR-DRJ}, for simple roots $\alpha_i$ with $i=1,\dots,m+n$, we can write the product
\begin{align*}
\prod\limits_{k\geq0 }\exp_{(-1)^{[\alpha_i]}q_{i}^{-1}}\bigl((-1)^{[\alpha_i]}\bigl(q_{i}-q_{i}^{-1}\bigr)\bigl(zq^{-c/2}\bigr)^k\mathfrak{E}_{\alpha_i+k\delta}\otimes \mathfrak{F}_{\alpha_i+k\delta}\bigr)
\end{align*}as
\begin{align*}
\prod\limits_{k\geq0 }\exp_{(-1)^{[\alpha_i]}q_{i}^{-1}}\bigl((-1)^{[\alpha_i]}\bigl(q_{i}-q_{i}^{-1}\bigr)\bigl(zq^{-c/2}\bigr)^k x^{+}_{i,k}\otimes x^{-}_{i,-k}\bigr).
\end{align*}
Suppose that $i\leq n$, then by the representation $\pi_{V}$ presented in Proposition \ref{Level-0} for $V(1)=V$, we get
\begin{align*}
&(1\otimes\pi_V) \prod\limits_{k\geq0 }\exp_{(-1)^{[\alpha_i]}q_{i}^{-1}}\bigl((-1)^{[\alpha_i]}\bigl(q_{i}-q_{i}^{-1}\bigr)\bigl(zq^{-c/2}\bigr)^k x^{+}_{i,k}\otimes x^{-}_{i,-k}\bigr)\\
&\qquad=\prod\limits_{k\geq0 }\exp_{(-1)^{[\alpha_i]}q_{i}^{-1}}\bigl((-1)^{[\alpha_i]}\bigl(q_{i}-q_{i}^{-1}\bigr)\bigl(z_-q^{i}\bigr)^k x^{+}_{i,k}\otimes E^{i+1}_i-(-1)^{[\alpha_i]}\bigl(q_{i}-q_{i}^{-1}\bigr)\\
&\phantom{\qquad=}{}\times\bigl(z_-q^{2n-2m-i+1}\bigr)^k x^{+}_{i,k}\otimes E^{\overline{i}}_{\overline{i+1}}\bigr).
\end{align*}
Expanding the $q$-exponent and using the definition of $f_i^+(z)$ and $x_i^+(z)^{>0}$, we deduce that \smash{$1+f_i^+(z)\otimes E^{i+1}_i-f_i^+\bigl(z\zeta q^{2i}\bigr)\otimes E^{\overline{i}}_{\overline{i+1}}$} for $i\leq n$ as required. A similar calculation shows that this holds for $n<i$, thereby giving us the expression of $\dot{F}^+(z)$.
For $\dot{E}^+(z)$, first from the definition elements of $k_i\tilde{\mapsto} K_i=\exp(\hbar h_i)$, we have
\begin{align*}
(1\otimes\pi_V)(\mathcal{T}_{12})={}&\exp\Biggl(\hbar\sum\limits_{i=1}^{n+m}\sum\limits_{j=1}^{n+m}\bigl(A_{ij}^{\operatorname{sym}}\bigr)^{-1}h_i\otimes\pi_V(h_j)\Biggr)\\
={}&\exp\Biggl\{\hbar\sum\limits_{i=1}^{n+m}\Biggl(\sum\limits_{j=1}^{n-1}\bigl(A_{ij}^{\operatorname{sym}}\bigr)^{-1}h_i\otimes
\bigl(E^{j+1}_{j+1}-E^{j}_{j}-E^{\overline{j+1}}_{\overline{j+1}}+E^{\overline{j}}_{\overline{j}}\bigr)\\
&+
\bigl(A_{in}^{\operatorname{sym}}\bigr)^{-1}h_i\otimes\bigl(-E^{n+1}_{n+1}+E^{n}_{n}+E^{\overline{n+1}}_{\overline{n+1}}-E^{\overline{n}}_{\overline{n}}\bigr)\\
&+\sum\limits_{j=n+1}^{m+n-1}
\bigl(A_{ij}^{\operatorname{sym}}\bigr)^{-1}h_i\otimes\bigl(E^{j}_{j}-E^{j+1}_{j+1}-E^{\overline{j}}_{\overline{j}}+E^{\overline{j+1}}_{\overline{j+1}}\bigr)\\
&+\bigl(A_{i,m+n}^{\operatorname{sym}}\bigr)^{-1}h_i\otimes\bigl(E^{m+n}_{m+n}-E^{\overline{m+n}}_{\overline{m+n}}\bigr)
\Bigr)\Bigr\}\\
={}&\exp\Biggl\{\hbar\sum\limits_{i=1}^{n+m}\Biggl(\sum\limits_{j=2}^{n}\bigl(\bigl(A_{i,j-1}^{\operatorname{sym}}\bigr)^{-1}-\bigl(A_{ij}^{\operatorname{sym}}\bigr)^{-1}\bigr)h_i\otimes \bigl(E^{j}_{j}-E^{\overline{j}}_{\overline{j}}\bigr)\\
&+\bigl(A_{i1}^{\operatorname{sym}}\bigr)^{-1}h_i\otimes \bigl(E^{\overline{1}}_{\overline{1}}-E^{1}_{1}\bigr)+\sum\limits_{j=n+1}^{m+n}
\bigl(\bigl(A_{ij}^{\operatorname{sym}}\bigr)^{-1}-(A_{i,j-1}^{\operatorname{sym}})^{-1}\bigr)h_i\\
&\otimes \bigl(E^{j}_{j}-E^{\overline{j}}_{\overline{j}}\bigr)\Bigr)\Bigr\}.
\end{align*}
By the formulas of \smash{$\bigl(A_{ij}^{\operatorname{sym}}\bigr)^{-1}$}, it is evident that the image $(1\otimes\pi_V)\mathcal{T}$ forms a diagonal matrix given by
\begin{gather}\label{T12}
 \operatorname{diag}\Biggl(\prod\limits_{i=1}^{n+m}k_i, \dots, \prod\limits_{i=j}^{n+m}k_j, \dots, k_{m+n},1, k_{m+n}^{-1}, \dots, \prod\limits_{i=1}^{n+m}k_i^{-1}\Biggr).
\end{gather}
Following the same calculation procedure as for $\dot{F}^+(z)$, and utilizing the relations $\smash{k_ix_{j,k}^{\pm}k_i^{-1}}=\smash{q_i^{\pm A_{ij}}x_{j,k}^{\pm}}$, we derive the expression for $\dot{E}^+(z)$ in the decomposition.
For $\dot{H}^+(z)$, actually, using the vector representation $\pi_V$, we obtain $\dot{H}^+(z)$ with the form
\begin{align*}
 \dot{H}^+(z)={}&\exp\Biggl(\sum\limits_{k>0}\sum\limits_{i,j=1}^{n+m}\frac{\bigl(q_i-q_i^{-1}\bigr)\bigl(q_j-q_j^{-1}\bigr)}{q-q^{-1}}
 \frac{k}{[k]_q}\bigl(A_{ij}^{\operatorname{sym}}\bigl(q^k\bigr)\bigr)^{-1} z^k
 a_{i,k}\otimes\pi_V(a_{j,-k})\Biggr)\\
 &\times(1\otimes\pi_V)\mathcal{T}\kappa^{+}(z),
\end{align*}
and it is a diagonal matrix due to the action of the generators. For the exponent in this expression, consider the $(i,i)$-entry, we have
\begin{gather}\label{Ai-1}
 \exp\Biggl\{\sum\limits_{k>0}\sum\limits_{i=1}^{n+m}\bigl(q_i^{-1}-q_i\bigr)
 A_{i1}^{\operatorname{sym}}\bigl(q^k\bigr)^{-1} z^k
 a_{i,k}\Biggr\}\otimes E^1_1,
\end{gather}
and
\begin{gather}
 \exp\Biggl\{\sum\limits_{k>0}\Biggl(\sum\limits_{i=1}^{n}\bigl(q_i-q_i^{-1}\bigr)
 \bigl(q^{\nu_{j-1}k}A_{i,j-1}^{\operatorname{sym}}\bigl(q^k\bigr)^{-1}-q^{\nu_jk}A_{ij}^{\operatorname{sym}}\bigl(q^k\bigr)^{-1}\bigr)+ \sum\limits_{i=n+1}^{n+m}\bigl(q_i-q_i^{-1}\bigr)\nonumber\\
 \qquad\times \bigl(q^{\nu_jk}A_{ij}^{\operatorname{sym}}\bigl(q^k\bigr)^{-1}-q^{\nu_{j-1}k}A_{i,j-1}^{\operatorname{sym}}\bigl(q^k\bigr)^{-1}\bigr)\bigr) z^k
 a_{i,k}\bigr\}\otimes E^j_j\label{Aij}
\end{gather}
for $j=2,\dots,n+m$,

As the coefficient of $E^1_1$ in \eqref{Ai-1},
\begin{gather*}
 \exp\Biggl\{\sum\limits_{k>0}\sum\limits_{i=1}^{n+m}\bigl(q_i^{-1}-q_i\bigr)
 A_{i1}^{\operatorname{sym}}\bigl(q^k\bigr)^{-1} z^k
 a_{i,k}\Biggr\}\\
 \qquad=\exp\Biggl\{\sum\limits_{k>0}\Biggl(\sum\limits_{i=1}^{n}\bigl(q-q^{-1}\bigr)
 \frac{[m-n+j]_{q^k}-[m-n+j-1]_{q^k}}{[m-n]_{q^k}-[m-n-1]_{q^k}}\\
 \phantom{\qquad=}{}+\sum\limits_{j=n+1}^{m+n-1}\bigl(q-q^{-1}\bigr)
 \frac{[m+n-j]_{q^k}-[m+n-j-1]_{q^k}}{[m-n]_{q^k}-[m-n-1]_{q^k}}\\
 \phantom{\qquad=}{} +\bigl(q_{m+n}-q_{m+n}^{-1}\bigr)\frac{1}{[m-n]_{q^k}-[m-n-1]_{q^k}}
 \biggr) z^k
 a_{i,k}\biggr\}.
\end{gather*}
By a directly calculation,
\begin{gather*}
 \exp\Biggl\{\sum\limits_{k>0}\sum\limits_{i=1}^{n+m}\bigl(q_i-q_i^{-1}\bigr)
 A_{i1}^{\operatorname{sym}}\bigl(q^k\bigr)^{-1} z^k
 a_{i,k}\Biggr\}\\
 \qquad=\exp\Biggl\{\sum\limits_{k>0}\Biggl(\sum\limits_{i=1}^{n}\bigl(q-q^{-1}\bigr)\times
 \frac{q^{ik}-q^{-ik}\zeta^{-k}}{1+\zeta^{-k}}+
 \sum\limits_{j=n+1}^{m+n-1}\bigl(q-q^{-1}\bigr)
 \frac{q^{(2n-j)k}-q^{(-2n+j)k}\zeta^{-k}}{1+\zeta^{-k}}\\
 \phantom{\qquad=}{} +
 \frac{\bigl(q_{m+n}-q_{m+n}^{-1}\bigr)\bigl(q^{(n-m)k}-q^{(m-n+1)k}\bigr)}{1+\zeta^{-k}}
 \biggr) z^k
 a_{i,k}\biggr\}.
\end{gather*}
Using the Taylor formula, we expand the fractions into power series as
\begin{gather*}
 \exp\Biggl\{\sum\limits_{k>0}\sum\limits_{i=1}^{m+n-1}\sum\limits_{p=0}^{\infty}\bigl(q-q^{-1}\bigr)(-1)^p
 \bigl(\zeta^{-pk}q^{-\nu_ik}1+\zeta^{-(p+1)k}q^{\nu_ik}\bigr)z^k
 a_{i,k}\Biggr\}\\
 \qquad \times\exp\Biggl\{\sum\limits_{k>0}\sum\limits_{p=0}^{\infty}\bigl(q_{m+n}-q_{m+n}^{-1}\bigr)(-1)^p
 \bigl(\zeta^{-pk}q^{(n-m)k}1+\zeta^{-(p+1)k}q^{(m-n+1)k}\bigr)z^k
 a_{m+n,k}\Biggr\}.
\end{gather*}
 Set \smash{$\widetilde{\Phi}_i^+(z)=k_i^{-1}\Phi_i^+(z)$}, where $\Phi^+_i(z)$ is the definition \eqref{homo2}. Then the above expression take the form
\begin{gather*}
 \prod\limits_{p=0}^{\infty}\prod\limits_{k=1}^{n+m}\widetilde{\Phi}^{+}_{k}\bigl(z\zeta^{-2p} q^{-\nu_k}\bigr)^{-1}\widetilde{\Phi}^{+}_{k}\bigl(z\zeta^{-2p-1} q^{-\nu_k}\bigr)\widetilde{\Phi}^{+}_{k}\bigl(z\zeta^{-2p-1} q^{\nu_k}\bigr)^{-1}\widetilde{\Phi}^{+}_{k}\bigl(z\zeta^{-2p-2} q^{\nu_k}\bigr)
.
\end{gather*}
Therefore, applying Propositions \ref{pr5} and \ref{Apr0}, we deduce that
\begin{gather*}
 \exp\Biggl\{\sum\limits_{k>0}\sum\limits_{i=1}^{n+m}\bigl(q_i-q_i^{-1}\bigr)
 A_{i1}^{\operatorname{sym}}\bigl(q^k\bigr)^{-1} z^k
 a_{i,k}\Biggr\}(1\otimes\pi_V)\mathcal{T}\kappa^{+}(z)=h_1^{+}(z)
\end{gather*}
for \smash{$\widetilde{\Phi}^{+}_{i}(z)=k_i^{-1} h_{i+1}^{+}(zq^{-\nu_i})h_{i}^{+}(zq^{-\nu_i})^{-1}$} and
the formulas of diagonal matrix
\eqref{T12}.

Moreover, by the similar arguments and the formulas of \smash{$A_{ij}^{\operatorname{sym}}\bigl(q^k\bigr)$},
the expression in the position $E^j_j$, $j=2,\dots,m+n$, of \eqref{Aij} can be determined as
\begin{align*}
 \Upsilon_j={}&\prod\limits_{p=0}^{\infty}\prod\limits_{k=1}^{j-1}\widetilde{\Phi}^{+}_{k}\bigl(z\zeta^{-2p} q^{-\nu_k}\bigr)^{-1}\widetilde{\Phi}^{+}_{k}\bigl(z\zeta^{-2p} q^{\nu_k}\bigr)\widetilde{\Phi}^{+}_{k}\bigl(z\zeta^{-2p-1} q^{-\nu_k}\bigr)\widetilde{\Phi}^{+}_{k}\bigl(z\zeta^{-2p-1} q^{\nu_k}\bigr)^{-1}\\
&\times \prod\limits_{p=0}^{\infty}\prod\limits_{k=j}^{n+m}\widetilde{\Phi}^{+}_{k}\bigl(z\zeta^{-2p} q^{-\nu_k}\bigr)^{-1}\widetilde{\Phi}^{+}_{k}\bigl(z\zeta^{-2p-1} q^{\nu_k}\bigr)^{-1}\widetilde{\Phi}^{+}_{k}\bigl(z\zeta^{-2p-1} q^{-\nu_k}\bigr)\widetilde{\Phi}^{+}_{k}\bigl(z\zeta^{-2p-2} q^{\nu_k}\bigr),
\end{align*}
and hence
$\Upsilon_j(1\otimes\pi_V)\mathcal{T}\kappa^{+}(z)=h_j^{+}(z)
$
via Propositions \ref{pr5} and \ref{Apr0}.
The remaining expression in the position \smash{$E_{\overline{i}}^{\overline{i}}$}, $i=1,\dots,m+n$, are similar. So, we have the diagonal matrix \smash{$\dot{H}^+(z)$}.
\end{proof}

\begin{Remark}
The submatrix decomposition in this proposition for the indexes $n+1\leq i\leq 2m+n+1$ have the similar result in \cite{JMAB} via Remark \ref{RM2}\,(1).
\end{Remark}

Now, the main result in this section, we claim the following.

\begin{Theorem}\label{th4}
The superalgebra $U(R)$ is isomorphic to $\mathcal{A}_q$.
\end{Theorem}
\begin{proof}
It is straightforward to observe that the map $AR\colon \mathcal{A}_q\rightarrow U(R)$, defined as follows:
\begin{align*}
 &X_i^{+}(z)\mapsto X^{+}_i(z)=f^+_{i}(z_+)-f^-_{i}(z_-), \qquad 1\leq i\leq m+n,\\
 &X^{-}_i(z)\mapsto X^{-}_i(z)=e^+_{i}(z_+)-e^-_{i}(z_-),\qquad 1\leq i\leq m+n,\\
 &h_j^{\pm}(w)\mapsto h_j^{\pm}(w), \qquad 1\leq j\leq m+n+1,
\end{align*}
defines a homomorphism. On the other hand, Propositions \ref{Apr1} and \ref{FHE} collectively imply that the homomorphism $RA$ serves as the inverse map of $AR$, thereby completing the proof.
\end{proof}

\begin{Definition}The $R$-matrix presentation of quantum affine superalgebra \smash{$\mathcal{U}_q^{R}(\mathfrak{\hat{g}})$} is an associative superalgebra over $\mathbb{C}\bigl(q^{1/2}\bigr)$ generated by an invertible central element $q^{c/2}$ and elements~$l_{ij}^{\pm}[\mp p]$, where the indices satisfy $ 1\leq i,j\leq2n+2m+1$,
 subject to the following relations:
\begin{align*}
 &l^+_{ii}[0]l^-_{ii}[0]=l^-_{ii}[0] l^+_{ii}[0]=1, \qquad l^+_{ij}[0]=l^-_{ij}[0]=0\qquad \text{for} \quad i>j,\\
 &R(z/w)L_1^{\pm}(z)L_2^{\pm}(w)=L_2^{\pm}(w)L_1^{\pm}(z) R(z/w),\\
 &R(z_+/w_-)L_1^{+}(z)L_2^{-}(w)=L_2^{-}(w)L_1^{+}(z) R(z_-/w_+),\\
 &DL^{\pm}(z\zeta)^tD^{-1}L^{\pm}(z)=L^{\pm}(z)DL^{\pm}(z\zeta)^tD^{-1}=1.
\end{align*}
Here
 $z_{\pm}=zq^{\pm c/2}$, and $L_i^{\pm}(z)\in \operatorname{End}\mathbb{C}^N\otimes \operatorname{End}\mathbb{C}^N\otimes
 U(R)$, $i=1,2$, written by
\begin{gather*}
 L_1^{\pm}(z)=\sum\limits_{i,j=1}E^i_j\otimes1\otimes l_{ij}^{\pm}(z),\qquad
 L_2^{\pm}(z)=\sum\limits_{i,j=1}1\otimes E^i_j\otimes l_{ij}^{\pm}(z),
\end{gather*}with
\begin{equation*}
 l_{ij}^{\pm}(z)=\sum\limits_{p=0}l_{ij}^{\pm}[\mp p]z^{\pm p}.
\end{equation*}
\end{Definition}

 Using the same notation of the $R$-matrix superalgebra $U(R)$, we have the following result immediately.
\begin{Corollary} The mapping
\begin{align*}
 &q^{c/2}\mapsto q^{c/2},\qquad
 x_i^{\pm}(z)\mapsto(-1)^{[\alpha_i]}\bigl(q_i-q_i^{-1}\bigr)^{-1}X_i^{\pm}(zq^{-\nu_i}),\qquad 1\leq i\leq m+n-1,\\
 &x_{m+n}^{\pm}(z)\mapsto(-1)^{[\alpha_{m+n}]}\bigl(q_{m+n}-q_{m+n}^{-1}\bigr)^{-1}[2]_{q_{m+n}}^{-1/2}X_{m+n}^{\pm}(zq^{n-m}),\\
 &\Phi^{\pm}_{i}(z)\mapsto h_{i+1}^{\pm}(zq^{-\nu_i})h_{i}^{\pm}(zq^{-\nu_i})^{-1},\qquad 1\leq i\leq m+n,
\end{align*}
define an isomorphism $\mathcal{U}_q(\mathfrak{\hat{g}})\rightarrow\mathcal{U}_q^{R}(\mathfrak{\hat{g}})$.
\end{Corollary}

\subsection*{Acknowledgements}

The authors would like to sincerely thank the anonymous referees for their valuable suggestions and contributions to improve the paper.
H.~Zhang is partially supported by the support of the National Natural Science Foundation of China (No.~12271332), and Shanghai Natural Science Foundation grant 22ZR1424600.

\pdfbookmark[1]{References}{ref}
\LastPageEnding


\begin{thebibliography}{99}
\footnotesize\itemsep=0pt

\bibitem{JBECK}
Beck J., Braid group action and quantum affine algebras, \href{https://doi.org/10.1007/BF02099423}{\textit{Comm. Math.
 Phys.}} \textbf{165} (1994), 555--568, \href{https://arxiv.org/abs/hep-th/9404165}{arXiv:hep-th/9404165}.

\bibitem{Bezerra}
Bezerra L., Futorny V., Kashuba I., Drinfeld realization for quantum affine
 superalgebras of type~{$B$}, \href{https://arxiv.org/abs/2405.05533}{arXiv:2405.05533}.

\bibitem{IGA}
Brundan J., Kleshchev A., Parabolic presentations of the
 {Y}angian~{$Y({\mathfrak{gl}}_n)$}, \href{https://doi.org/10.1007/s00220-004-1249-6}{\textit{Comm. Math. Phys.}} \textbf{254}
 (2005), 191--220, \href{https://arxiv.org/abs/math.QA/0407011}{arXiv:math.QA/0407011}.

\bibitem{JSKW}
Cai J.-F., Wang S.-K., Wu K., Zhao W.-Z., Drinfeld realization of quantum affine
 superalgebra~{$U_q(\widehat{\mathfrak{gl}(1|1)})$}, \href{https://doi.org/10.1088/0305-4470/31/8/011}{\textit{J.~Phys.~A}}
 \textbf{31} (1998), 1989--1994, \href{https://arxiv.org/abs/q-alg/9703022}{arXiv:q-alg/9703022}.

\bibitem{Damiani1}
Damiani I., Drinfeld realization of affine quantum algebras: the relations,
\href{https://doi.org/10.2977/PRIMS/86}{\textit{Publ. Res. Inst. Math. Sci.}} \textbf{48} (2012), 661--733,
 \href{https://arxiv.org/abs/1406.6729}{arXiv:1406.6729}.

\bibitem{Damiani2}
Damiani I., From the {D}rinfeld realization to the {D}rinfeld--{J}imbo
 presentation of affine quantum algebras: injectivity, \href{https://doi.org/10.4171/PRIMS/150}{\textit{Publ. Res.
 Inst. Math. Sci.}} \textbf{51} (2015), 131--171, \href{https://arxiv.org/abs/1407.0341}{arXiv:1407.0341}.

\bibitem{JDIB}
Ding J.T., Frenkel I.B., Isomorphism of two realizations of quantum affine
 algebra~{$U_q(\mathfrak{gl}(n))$}, \href{https://doi.org/10.1007/bf02098484}{\textit{Comm. Math. Phys.}} \textbf{156}
 (1993), 277--300.

\bibitem{VGD2}
Drinfeld V.G., A new realization of {Y}angians and of quantum affine algebras,
 \textit{Soviet Math. Dokl.} \textbf{36} (1988), 212--216.

\bibitem{VGD}
Drinfeld V.G., Hopf algebras and the quantum {Y}ang--{B}axter equation, in
 Yang--{B}axter {E}quation in {I}ntegrable {S}ystems, \textit{Adv. Ser. Math. Phys.},
 \href{https://doi.org/10.1142/9789812798336_0013}{World Scientific Publishing}, 1990, 264--268.

\bibitem{HBKD}
Fan H., Hou B.-Y., Shi K.-J., Drinfeld constructions of the quantum affine
 superalgebra~{$U_q(\widehat {\mathfrak{gl}(m|n)})$}, \href{https://doi.org/10.1063/1.532188}{\textit{J.~Math. Phys.}}
 \textbf{38} (1997), 411--433.

\bibitem{Frassek}
Frassek R., Tsymbaliuk A., Orthosymplectic {Y}angians, \href{https://arxiv.org/abs/2311.18818}{arXiv:2311.18818}.

\bibitem{EFEM}
Frenkel E., Mukhin E., The {H}opf algebra~{$\operatorname{Rep} U_q\widehat{\mathfrak{gl}}_\infty$}, \href{https://doi.org/10.1007/PL00012603}{\textit{Selecta Math. (N.S.)}}
 \textbf{8} (2002), 537--635, \href{https://arxiv.org/abs/math.QA/0103126}{arXiv:math.QA/0103126}.

\bibitem{FIBR}
Frenkel I.B., Reshetikhin N.Yu., Quantum affine algebras and holonomic
 difference equations, \href{https://doi.org/10.1007/bf02099206}{\textit{Comm. Math. Phys.}} \textbf{146} (1992), 1--60.

\bibitem{WGMJ}
Galleas W., Martins M.J., New~{$R$}-matrices from representations of
 braid-monoid algebras based on superalgebras, \href{https://doi.org/10.1016/j.nuclphysb.2005.10.025}{\textit{Nuclear Phys.~B}}
 \textbf{732} (2006), 444--462.

\bibitem{IMG}
Gelfand I.M., Retakh V.S., Determinants of matrices over noncommutative rings,
 \href{https://doi.org/10.1007/BF01079588}{\textit{Funct. Anal. Appl.}} \textbf{25} (1991), 91--102.

\bibitem{MDGY}
Gould M.D., Zhang Y.-Z., On super-{RS} algebra and {D}rinfeld realization of
 quantum affine superalgebras, \href{https://doi.org/10.1023/A:1007494602290}{\textit{Lett. Math. Phys.}} \textbf{44} (1998),
 291--308, \href{https://arxiv.org/abs/q-alg/9712011}{arXiv:q-alg/9712011}.

\bibitem{MJB}
Jimbo M., A {$q$}-difference analogue of~{$U(\mathfrak{g})$} and the
 {Y}ang--{B}axter equation, \href{https://doi.org/10.1007/BF00704588}{\textit{Lett. Math. Phys.}} \textbf{10} (1985),
 63--69.

\bibitem{MJB2}
Jimbo M., Quantum {$R$} matrix for the generalized {T}oda system, \href{https://doi.org/10.1007/BF01221646}{\textit{Comm.
 Math. Phys.}} \textbf{102} (1986), 537--547.

\bibitem{JMAB-Y}
Jing N., Liu M., Molev A., Isomorphism between the $R$-matrix and Drinfeld
 presentations of {Y}angian in types~$B$, $C$ and~$D$, \href{https://doi.org/10.1007/s00220-018-3185-x}{\textit{Comm. Math.
 Phys.}} \textbf{361} (2018), 827--872, \href{https://arxiv.org/abs/1705.08155}{arXiv:1705.08155}.

\bibitem{JMA}
Jing N., Liu M., Molev A., Isomorphism between the {$R$}-matrix and {D}rinfeld
 presentations of quantum affine algebra: type~{$C$}, \href{https://doi.org/10.1063/1.5133854}{\textit{J.~Math. Phys.}}
 \textbf{61} (2020), 031701, 41~pages, \href{https://arxiv.org/abs/1903.00204}{arXiv:1903.00204}.

\bibitem{JMAB}
Jing N., Liu M., Molev A., Isomorphism between the $R$-matrix and {D}rinfeld
 presentations of quantum affine algebra: types~$B$ and~$D$, \href{https://doi.org/10.3842/SIGMA.2020.043}{\textit{SIGMA}}
 \textbf{16} (2020), 043, 49~pages, \href{https://arxiv.org/abs/1911.03496}{arXiv:1911.03496}.

\bibitem{NJHZ1}
Jing N., Zhang H., Drinfeld realization of twisted quantum affine algebras,
 \href{https://doi.org/10.1080/00927870701404713}{\textit{Comm. Algebra}} \textbf{35} (2007), 3683--3698.

\bibitem{NJHZ}
Jing N., Zhang H., Drinfeld realization of quantum twisted affine algebras via
 braid group, \href{https://doi.org/10.1155/2016/4843075}{\textit{Adv. Math. Phys.}} \textbf{2016} (2016), 4843075,
 15~pages, \href{https://arxiv.org/abs/1301.3550}{arXiv:1301.3550}.

\bibitem{VGD3}
Kac V.G., Lie superalgebras, \href{https://doi.org/10.1016/0001-8708(77)90017-2}{\textit{Adv. Math.}} \textbf{26} (1977), 8--96.

\bibitem{SMVN}
Khoroshkin S.M., Tolstoy V.N., Twisting of quantum (super)algebras.
 {C}onnection of {D}rinfeld's and {C}artan--{W}eyl realizations for quantum
 affine algebras, \href{https://arxiv.org/abs/hep-th/9404036}{arXiv:hep-th/9404036}.

\bibitem{SMVN2}
Khoroshkin S.M., Tolstoy V.N., Universal {$R$}-matrix for quantized
 (super)algebras, \href{https://doi.org/10.1007/BF02102819}{\textit{Comm. Math. Phys.}} \textbf{141} (1991), 599--617.

\bibitem{DKBM}
Krob D., Leclerc B., Minor identities for quasi-determinants and quantum
 determinants, \href{https://doi.org/10.1007/BF02101594}{\textit{Comm. Math. Phys.}} \textbf{169} (1995), 1--23,
 \href{https://arxiv.org/abs/hep-th/9411194}{arXiv:hep-th/9411194}.

\bibitem{SZLO}
Levendorski\u{\i} S.Z., On generators and defining relations of {Y}angians,
 \href{https://doi.org/10.1016/0393-0440(93)90084-R}{\textit{J.~Geom. Phys.}} \textbf{12} (1993), 1--11.

\bibitem{LYZ}
Lin H., Yamane H., Zhang H., On generators and defining relations of quantum
 affine superalgebra~{$U_q (\widehat{\mathfrak{sl}}_{m|n})$},
 \href{https://doi.org/10.1142/S021949882450021X}{\textit{J.~Algebra Appl.}} \textbf{23} (2024), 2450021, 29~pages.

\bibitem{K-Lu}
Lu K., Isomorphism between twisted $q$-Yangians and affine $i$quantum groups:
 type AI, \href{https://doi.org/10.1093/imrn/rnae248}{\textit{Int. Math. Res. Not.}}, {t}o appear, \href{https://arxiv.org/abs/2308.12484}{arXiv:2308.12484}.

\bibitem{MKMJ}
Mehta M., Dancer K.A., Gould M.D., Links J., Generalized {P}erk--{S}chultz
 models: solutions of the {Y}ang--{B}axter equation associated with quantized
 orthosymplectic superalgebras, \href{https://doi.org/10.1088/0305-4470/39/1/L03}{\textit{J.~Phys.~A}} \textbf{39} (2006),
 L17--L26, \href{https://arxiv.org/abs/nlin.SI/0509019}{arXiv:nlin.SI/0509019}.

\bibitem{AIM}
Molev A.I., A {D}rinfeld-type presentation of the orthosymplectic {Y}angians,
 \href{https://doi.org/10.1007/s10468-023-10227-9}{\textit{Algebr. Represent. Theory}} \textbf{27} (2024), 469--494,
 \href{https://arxiv.org/abs/2112.10419}{arXiv:2112.10419}.

\bibitem{RS}
Reshetikhin N.Yu., Semenov-Tian-Shansky M.A., Central extensions of quantum
 current groups, \href{https://doi.org/10.1007/BF01045884}{\textit{Lett. Math. Phys.}} \textbf{19} (1990), 133--142.

\bibitem{VANDE}
van~de Leur J.W., A classification of contragredient {L}ie superalgebras of
 finite growth, \href{https://doi.org/10.1080/00927878908823823}{\textit{Comm. Algebra}} \textbf{17} (1989), 1815--1841.

\bibitem{XWLHZ}
Wu X., Lin H., Zhang H., Braid group action and quantum affine superalgebra for
 type $\mathfrak{osp}(2m+1|2n)$, \href{https://arxiv.org/abs/2410.21755}{arXiv:2410.21755}.

\bibitem{YZC2}
Xu Y., Zhang R., Drinfeld realisations and vertex operator respresentations of
 the quantum affine superalgebras, \href{https://arxiv.org/abs/1802.09702}{arXiv:1802.09702}.

\bibitem{HYAM}
Yamane H., On defining relations of affine {L}ie superalgebras and affine
 quantized universal enveloping superalgebras, \href{https://doi.org/10.2977/prims/1195143607}{\textit{Publ. Res. Inst. Math.
 Sci.}} \textbf{35} (1999), 321--390, \href{https://arxiv.org/abs/q-alg/9603015}{arXiv:q-alg/9603015}.

\bibitem{Huafeng}
Zhang H., R{TT} realization of quantum affine superalgebras and tensor
 products, \href{https://doi.org/10.1093/imrn/rnv167}{\textit{Int. Math. Res. Not.}} \textbf{2016} (2016), 1126--1157,
 \href{https://arxiv.org/abs/1407.7001}{arXiv:1407.7001}.

\bibitem{HNJING}
Zhang H., Jing N., Drinfeld realization of twisted quantum affine algebras,
 \href{https://doi.org/10.1080/00927870701404713}{\textit{Comm. Algebra}} \textbf{35} (2007), 3683--3698, \href{https://arxiv.org/abs/1301.3550}{arXiv:1301.3550}.

\bibitem{YZC}
Zhang Y.-Z., Comments on the {D}rinfeld realization of the quantum affine
 superalgebra~{$U_q[\mathfrak{gl}(m|n)^{(1)}]$} and its {H}opf
 algebra structure, \href{https://doi.org/10.1088/0305-4470/30/23/028}{\textit{J.~Phys.~A}} \textbf{30} (1997), 8325--8335,
 \href{https://arxiv.org/abs/q-alg/9703020}{arXiv:q-alg/9703020}.

\end{thebibliography}
\end{document}